\documentclass[10pt, a4paper]{amsart} 
\usepackage{color}
\usepackage{amscd,amssymb,graphics}
\usepackage[hidelinks]{hyperref}
\usepackage{amsfonts}
\usepackage{amsmath}
\usepackage{tikz-cd}
\usepackage{multirow}

\input xy
\xyoption{all}
\usepackage{epsfig}

\newtheorem{thm}{Theorem}[section]
\newtheorem{f}[thm]{Fact}
\newtheorem{cor}[thm]{Corollary}
\newtheorem{lem}[thm]{Lemma}
\newtheorem{lemma}[thm]{Lemma}
\newtheorem{prop}[thm]{Proposition}

\newtheorem{conj}[thm]{Conjecture}

\newtheorem{q}[thm]{Question}

\theoremstyle{definition}
\newtheorem{defin}[thm]{Definition}

\theoremstyle{remark}
\newtheorem{remark}[thm]{Remark}
\newtheorem{remarks}[thm]{Remarks}

\newtheorem{ex}[thm]{Example}

\newtheorem{exs}[thm]{Examples}

\numberwithin{equation}{section}

\newcommand{\delete}[1]{} % Comment out text.
\newcommand{\nt}{\noindent}

\def\eps{{\varepsilon}}
\newcommand{\sk}{\vskip 0.1cm}

\newcommand{\nl}{\newline}

\newcommand{\ben}{\begin{enumerate}}

\newcommand{\een}{\end{enumerate}}
\newcommand{\bit}{\begin{itemize}}

\newcommand{\eit}{\end{itemize}}

\newcommand{\bsmall}[2][\mathfrak{B}]{\mathrm{small}\left({#1}, {#2}\right)}

\newcommand{\cobsmall}[2][\mathfrak{B}]{\mathrm{small}^{*}\left({#1}, {#2}\right)}
\newcommand{\localcobsmall}[2]{\mathrm{small}^{*}\left(\mathfrak{B}, {#1}, {#2}\right)}

\def\R {{\mathbb R}}
\def\N {{\mathbb N}}

\newcommand{\Tame}{\mathrm{Tame}}

\def\Iso{{\mathrm{Iso}}\,}

\def\im{{\mathrm Im}\,}

\newcommand{\id}{{\rm{id}}}
\def\diam{{\mathrm{diam}}}

\def\F{{\mathcal F}}

\def\QED{\nobreak\quad\ifmmode\roman{Q.E.D.}\else{\rm Q.E.D.}\fi}

\def\Ros{\mathrm{\mathbf{(Ros)}}}
\def\NP{\mathrm{\mathbf{(NP)}}} 
\def\Rl{\mathrm{\mathbf{(R_1)}}}
\def\coRl{\mathrm{\mathbf{(R_{1}^{*})}}}
\def\ORl{\mathrm{\mathbf{(\overline{R_1})}}}
\def\Tame{\mathrm{\mathbf{(T)}}}
\def\sTame{\mathrm{\mathbf{(mT)}}}
\def\DLP{\mathrm{\mathbf{(DLP)}}}

\def\BTame{\mathrm{\mathbf{[T]}}}
\def\BNP{\mathrm{\mathbf{[NP]}}}
\def\BDLP{\mathrm{\mathbf{[DLP]}}}

\def\a{\alpha}

\def\im{\operatorname{Im}}
\def\spn{\operatorname{Span}}
\def\support{\operatorname{Supp}}

\newcommand{\cls}{\rm{cl\,}}

\newcommand{\al}{\alpha}

\newcommand{\lan}{\langle}
\newcommand{\ran}{\rangle}

\newcommand{\co}{{\rm{co\,}}}
\newcommand{\ext}{{\rm{ext\,}}}
\newcommand{\acx}{{\rm{acx\,}}}
\newcommand{\bo}{{\rm{bal\,}}}
\newcommand{\eqc}{{\rm{eqc\,}}}

\newtheorem{prob}[thm]{Problem}

\graphicspath{ {media/} }

\begin{document}

\title[]
%0808 
{Tameness and Rosenthal type locally convex spaces} 
%{Fragmentability, tameness and Rosenthal type locally convex spaces}

\sk

%Authors
\author[]{Matan Komisarchik}
\address{Department of Mathematics,
	Bar-Ilan University, 52900 Ramat-Gan, Israel}
\email{matan.komisarchik@biu.ac.il}
%\urladdr{http://www.}

\author[]{Michael Megrelishvili}
\address{Department of Mathematics,
	Bar-Ilan University, 52900 Ramat-Gan, Israel}
\email{megereli@math.biu.ac.il}
\urladdr{http://www.math.biu.ac.il/$^\sim$megereli}

\date{2022, April 13}

\begin{abstract}
	Motivated by Rosenthal's famous $l^1$-dichotomy in Banach spaces, 
	Haydon's theorem, and additionally by recent works on tame dynamical systems, 
	we introduce the class of \textit{tame} locally convex spaces.  
	This is a natural locally convex 
	%030422	Change analog to analogue by Saak suggestion
	analogue of \textit{Rosenthal} Banach spaces  
	(for which any bounded sequence contains a  weak Cauchy subsequence). 
	Our approach is based on a bornology of \textit{tame} subsets which in turn is closely related to eventual fragmentability. This leads, among others, to the following results: 
	\begin{itemize}
		\item extending Haydon's characterization of Rosenthal Banach spaces, by showing that a lcs $E$ is tame iff every weak-star compact, equicontinuous convex subset of $E^{*}$ is the strong closed convex hull of its extreme points  iff $\overline{\co}^{w^{*}}(K) = \overline{\co}(K)$ for every weak-star compact equicontinuous subset $K$ of $E^{*}$; 
		\item $E$ is tame iff there is no bounded sequence 
		equivalent to the generalized $l^{1}$-sequence; 
		\item 
		strengthening some results of W.M. Ruess about Rosenthal's dichotomy; 
		\item applying the Davis--Figiel--Johnson--Pelczy\'nski (DFJP) technique one may show that every \textit{tame operator} $T \colon E \to F$ between a lcs $E$ and a Banach space $F$ 
		can be factored through a tame (i.e., Rosenthal) Banach space.   
	\end{itemize} 
\end{abstract}

%abstract from arxiv:
%Motivated by Rosenthal's famous $l^1$-dichotomy in Banach spaces, Haydon's theorem, and additionally by recent works on tame dynamical systems, we introduce the class of tame locally convex spaces. This is a natural locally convex analog of Rosenthal Banach spaces (for which any bounded sequence contains a weak Cauchy subsequence). Our approach is based on a bornology of tame subsets which in turn is closely related to eventual fragmentability. This leads, among others, to the following results:
%$\bullet$ extending Haydon's characterization of Rosenthal Banach spaces, by showing that a lcs $E$ is tame iff every weak-star compact, equicontinuous convex subset of $E^{*}$ is the strong closed convex hull of its extreme points iff $\overline{\rm{co\,}}^{w^{*}}(K) = \overline{\rm{co\,}}(K)$ for every weak-star compact equicontinuous subset $K$ of $E^{*}$;
%$\bullet$ $E$ is tame iff there is no bounded sequence equivalent to the generalized $l^{1}$-sequence;
%$\bullet$ strengthening some results of W.M. Ruess about Rosenthal's dichotomy;
%$\bullet$ applying the Davis-Figiel-Johnson-Pelczy\'nski technique one may show that every tame operator $T \colon E \to F$ between lcs can be factored through a tame lcs.

%\item We also give natural locally convex analogs of Asplund and reflexive Banach spaces. 

%1402 additional subjects

%1402 additional keywords

%0812 
\thanks{This research was supported by a grant of the Israel Science Foundation (ISF 1194/19) and also by the Gelbart Research Institute at the Department of Mathematics, Bar-Ilan  University} 
 
\maketitle
\setcounter{tocdepth}{1}
\tableofcontents

\section{Introduction} 
%030422	Changed analog to analogue by Saak suggestion
In the present work, we introduce and study a locally convex analogue of Rosenthal Banach spaces. 
 As in \cite{GM-rose,GM-fpt,GM-survey}, we say that a Banach space $V$ is \textit{Rosenthal} if any bounded sequence contains a weak Cauchy subsequence.
 %1912
 Equivalently, if $V$ does not contain an isomorphic copy of $l^1$. 
 Such Banach spaces appear in many publications (especially, after Rosenthal's classical work \cite{Ros0}). Mostly without any special name. 

  In order to better understand our approach and related classes, we present our definition in the 
  %1802
  framework of the smallness hieararchy for bounded subsets in lcs. In this way, we also provide natural locally convex 
  %030422	Changed analog to analogue bby Saak suggestion
  analogues of Asplund and reflexive Banach spaces.  

%\sk 
\subsection*{Smallness hierarchy of bounded subsets}

The relationship between a space $E$ and its 
%290322	Saak suggestion
topological 
dual $E^*$,  
%0108 
via various classical bornologies on $E$, is one of the central themes 
in the theory of locally convex spaces. 
For every bounded subset %(in particular, a sequence) 
%2707: slight changes
$B$ of $E$ and an equicontinuous weak-star compact subset $K$ of $E^*$ (notation: $K \in \eqc(E^*)$), we can think of $B$ as a bounded family of real valued functions over $K$ (via the canonical bilinear map $E \times E^* \to \R$). 
This ``tango" between $B$ and $K$ is a source of many interesting properties of the entire space. 
Namely, we want to study whether the family $\tilde{B}:=\{\tilde{b} \colon K \to \R\}_{b \in B}$ is small (in some sense), and then study the locally convex spaces whose 
%1503 
all 
bounded subsets are small in the same way.

This is related to the general topological question: what might be a hierarchy of \textit{smallness} for a bounded family $B \subset {\R}^K$ of real functions on an abstract compact space $K$? 
%2012	New
We present a framework for this kind of comparisons using the concept of bornological classes (Section \ref{s:Born}).

%090122  minor changes 
We suggest three cases which seem to be very natural. They are 
%1802 
%2002 extremely 
very important in the theory of dynamical systems and their representations 
on mainstream classes of Banach spaces. %2609, \cite{GM-survey,GM-rose}.  
See Section \ref{s:repr} and joint works of the second author with Eli Glasner \cite{GM-survey,GM-rose}. 
Consider the following three conditions on $B$: 
\begin{enumerate}
	\item $B$ is \textit{tame}  on $K$ (does not contain any sequence which is combinatorially independent in the sense of Rosenthal, 
	%1503 
	Definition \ref{d:TameFamily}); 
		\item $B$ is a fragmented family (Definition \ref{d:fr-family}) of functions on $K$;  
	\item $B$ has the Grothendieck's double limit property (DLP) on $K$ 
	%1503 
	(Definition \ref{d:DLP}). 
\end{enumerate}

%2707: Changed usage of DLP across the board (to be an adjactive)

%2609 Perhaps better to move this phrase to the coreesponding section about DLP
%We will often write that a subset \emph{is} DLP rather than the more correct \emph{"has the DLP"}.

%2507 
\begin{remark} \label{r:repres} 
	%Regarding tame, Asplund and DLP subsets in Banach spaces, note that %2507 the converse statements are true; 
	These three conditions 
		%090122  minor changes seem to be incomparable. 
		do not seem immediately comparable. 
		However, (3) $\Rightarrow$ (2) $\Rightarrow$ (1). 
	As it follows from results of \cite{GM-tame},
	%2007q: What converse statement?   %2507  really seems to be out of the context. So I removed this 
	every tame (fragmented, DLP) bounded family $B$ of continuous functions on a compact space $X$ can be \textit{represented} 
	on a Rosenthal (Asplund, reflexive) Banach space. 
	%0808 
	These results are based on
	%1802
	the Davis--Figiel--Johnson--Pelczy\'nski factorization technique \cite{DFJP}. See also Lemma \ref{l:DFJP} and Theorem \ref{t:Factoriz}. 

\end{remark}

	Recall that a \textit{representation} of 
	a bounded map $B \times K \to \R$ on %a Rosenthal (reflexive) 
a Banach space $V$ is 
a pair $(\nu,\a)$ of bounded maps $\nu\colon B \to V, \ \alpha\colon K \to V^*$, where $\a$ is weak-star continuous and 
$f(x)= \langle \nu(f), \a(x) \rangle$ for all $f \in B,  x \in K$. 
%090122  In other words, the following diagram commutes
$$\xymatrix{ B \ar@<-2ex>[d]_{\nu} \times K 
	\ar@<2ex>[d]^{\a} \ar[r]  & \R \ar[d]^{id } \\
	V \times V^* \ar[r]  &  \R }
$$
For the converse direction (justifying these representations above), note that a Banach space $V$ is:

\begin{enumerate}
	\item  Rosenthal (not containing a copy of $l^1$)
	iff the closed unit ball $B_V$ of $V$ is a tame family of functions on the weak-star compact unit ball $B_{V^*}$ of $V^*$;  
	%every bounded subset $B \subset V$ is tame on every $K \in eqc(V^*)$; 
%iff there is no bounded subset sequence which is combinatorially independent on some $K \in eqc(V^*)$; 
	
	\item Asplund iff $B_V$ is a a fragmented family of functions on $B_{V^*}$; 
%	every bounded subset $B \subset V$ is an equi-fragmented family of maps on every $K \in eqc(V^*)$;  
	\item reflexive iff $B_V$ has DLP on $B_{V^*}$. 
	%every bounded subset $B \subset V$ has DLP on every $K \in eqc(V^*)$   
\end{enumerate}

These three characterizations and Remark \ref{r:repres} suggest corresponding locally convex 
%030422	Changed analog to analogue by Saak suggestion
analogues 
%0108 
via three bornologies of tame, Asplund and DLP subsets, as defined here. 

\begin{defin} \label{d:MainDef} Let $E$ be a lcs. 
	%090122q maybe here better to replace K by M ?
	\begin{enumerate}
		\item We say that a bounded subset $B \subset E$ is:
		\begin{enumerate}
			\item \textit{tame} if $B$ is a tame family on every weak-star compact equicontinuous subset ${K \in \eqc(E^*)}$; %of the dual $E^*$;  
			%(treating $B$ as a family of functions on $M$ via the canonical duality $E \times E^* \to \R$). 
			\item \textit{Asplund} if $B$ is a fragmented family on every $K \in \eqc(E^*)$; 
			\item DLP if $B$ is DLP on every $K \in \eqc(E^*)$. 
		\end{enumerate}
		  
		\item We say that a lcs $E$ is: 
%090122  minor changes 
		\begin{enumerate}
			\item \textit{tame} ($E \in \Tame$) if every bounded subset in $E$ is tame, Definition \ref{def:tame_lcs};
			\item Namioka-Phelps ($E \in \NP$) if every bounded subset in $E$ is Asplund, Definition \ref{d:NP}; 
			\item DLP  ($E \in \DLP$) if every bounded subset in $E$ is DLP,  Definition \ref{d:DLPlcs}. 
		\end{enumerate}
	\end{enumerate}
\end{defin}
%%%

Asplund subsets play a major role in Banach space theory (sometimes under different names); see \cite{Bourgin,Fabian}. 
 The class $\NP$ first was defined in \cite{Me-fr} using a different but equivalent approach. 

%Hopefully, tame and Asplund subsets will be important in locally convex spaces, too. 

%1403
%\sk
\subsection*{Properties and examples}

The class $\Tame$ is quite large. First of all, note that 
$$
  \DLP \subset \NP \subset \Tame. 
$$
This can be derived from Remark \ref{r:repres}.  
Using results of Diestel--Morris--Saxon \cite{DMS}, we show in Proposition \ref{p:variety} that $\Tame$ is \textit{properly} larger than the variety generated by all Banach Rosenthal spaces. Furthermore, 
$\Tame$  has nice stability properties (Theorem \ref{thm:properties_of_tame_class}). Among others we show that $\Tame$ is closed under: subspaces, arbitrary products, locally convex direct sums and bound covering continuous linear images.  

These properties are verified using the concept of fragmentability (which originally comes from Namioka--Phelps \cite{NP}, Jayne--Rogers \cite{JR}) and its natural generalization for families (borrowed from recent study of tame dynamical systems).

\textit{Fragmented families} (Definition \ref{d:fr-family}) are closely related to 
%2012 the 
tameness, providing an important sufficient condition. 
%More precisely, if a bounded family $B$ defined on a compact space $K$ is (eventually) fragmented then $B$ is a (tame) Asplund family.  
Beyond representation theory (Remark \ref{r:repres}), a more direct reason is that $B$ is tame on $K$ iff $B$ is \textit{eventually fragmented} in the sense of \cite{GM-rose} (i.e., every sequence in $B$ contains a subsequence which is fragmented on $K$).
We apply here some useful results of Rosenthal \cite{Ros0} and Talagrand \cite{Tal}, synthesized in Lemma \ref{f:sub-fr}.

%1407 
One of the challenges is to find 
when standard constructions lead to NP or tame lcs.  
For lcs of the type $C_k(X)$ we have a concrete (and somewhat expected) criterion, Proposition \ref{p:C(X)}, which (up to some reformulations) is quite close to a known result  
by Gabriyelyan--Kakol--Kubi\'{s}--Marciszewski \cite[Lemma 6.3]{Networks}.

%2609 promoting 9 lines to a subsection 
\sk 
\subsection*{Free locally convex spaces}
Another important construction producing  lcs is the classical free locally convex space $L(X)$, 
%$L(X,\mu)$ 
defined for every 
Tychonoff %%1104  topological 
space $X$. 
 %and a uniform space $(X,\mu)$, respectively. 
For every compact space $K$, its free lcs  
$L(K)$ is \textit{multi-reflexive} 
%2208q I think, better to say "pro-reflexive". I already suggested this name to Leiderman. Hope thatr it is not too late  
%1802
(i.e., embedded into a product of reflexive Banach spaces), 
as it was proved in a very recent paper by Leiderman and Uspenskij \cite{LU}. 
%090122  minor changes  
%Since multi-reflexive lcs is $\DLP$,  
%%0808 (by Corollary \ref{c:Ros-pr}), 
%we obtain that $L(K)$ is $\DLP$.  
Since multi-reflexive lcs (by Theorem \ref{thm:dlp_iff_semireflexive}) is $\DLP$, we obtain that $L(K)$ is $\DLP$.

More generally, in Theorem \ref{t:L(X)isNP} we show that $L(X)$ is $\DLP$ (hence, $\NP$ and $\Tame$) for every %1104 Dieudonn\'{e} complete 
Tychonoff space $X$. 
In particular, we get that $L(\N^{\N})$ is $\DLP$ for the Polish space $\N^{\N}$ of all irrationals. In contrast, 
another result from \cite{LU} shows that $L(\N^{\N})$ is not multi-reflexive. 
%2707	Removed: A bounded subset in a Banach space is DLP iff it is relatively weakly compact. 
Moreover, while every semi-reflexive lcs is $\DLP$, the spaces $L(X)$ (which are $\DLP$,   
%1104 for every Dieudonn\'{e} complete $X$), 
%1912 are "almost never" 
very rarely are 
semi-reflexive (Theorem \ref{t:LX_is_almost_never_semireflexive}). 

%2507 we need an example of a COMPLETE lcs which is DLP but not semi-reflexive
%2707: Impossible in virtue of Theorem thm:dlp_iff_semireflexive %0108 RIGHT ! Now the question is: if the completion of a DLP lcs is DLP (equivalently, semireflexive). I guess NOT ... 
%2507q I think the role of DLP subset in LCS are also quite rigid. It is important to clarify this ... 
%2707	Added a description of DLP subsets to theorem thm:dlp_iff_semireflexive

\sk 
\subsection*{Rosenthal type properties}

Recall that a sequence $\{x_n\}_{n \in \N}$ in a lcs 
%1912 $X$ 
$E$ is \textit{weak Cauchy} if the scalar sequence $u(x_n)$ is convergent for every $u \in E^*$. Rosenthal's celebrated dichotomy theorem (see \cite{Ros0}) asserts that every bounded
sequence in a Banach space either has a 
%2208q what is the relation between: a)"weak Cauchy subsequence b) (strongly) fragmented sequence  
%0109q Did you mean "(strongly) eventually fragmented sequence"?
%0609q Anything which gives a bridge ... let's discuss 
%1009q must discuss: I think every weak Cauchy sequence is (strongly) fragmented. That is, (strongly) Asplund.
%Look at my (unpublished) work https://u.math.biu.ac.il/~megereli/tame180916.pdf  at page 7, Lemma 2.4 (especially item (3)).  
%first of all let's check that I understand the setting. Second, if this is true, maybe it can simplift (or/and) strengthen some of our proofs/results ??
weak Cauchy subsequence or a subsequence equivalent to the unit vector basis of $l_1$ (an
$l_1$-\emph{sequence}).

%0908q isn't it a repetition of Definition d:AllDef1
%2208q  let's discuss this 
\begin{defin} \label{d:RosBan} 
We say (as in \cite{GM-rose,GM-fpt,GM-survey}) that a Banach space $V$ is {\em Rosenthal} if every bounded sequence in $A$ has a weak Cauchy subsequence. 
%090122 
We use the same definition for lcs 
%1802
(Definition \ref{d:AllDef1} $\Ros$). 
\end{defin}

\begin{defin} \label{d:AllDef1}
	Let $E$ be a lcs. 
	Define the following properties of $E$: 
	\ben
	\item[$\Ros$] \label{def:R:cauchy_subsequences} Every bounded sequence in $E$ has a subsequence which is weak Cauchy.  
	
	%1109	Removed usual
	\item[$\Rl$] There is no bounded sequence in $E$ which is equivalent to the $l^{1}$-basis    %1611  \label{def:R:l1_basis 
	%1611 
	(in the sense of Definition \ref{defin:equivalent_to_usual_l1}). 
	
	\item[$\ORl$] The 
	Banach space $l^{1}$ cannot be embedded into $E$.
	\een
\end{defin}

%090122  new place (it was before the definition {d:AllDef1}) with changes 
%We collect here some related 
All these three properties 
%in lcs which all 
are equivalent in Banach spaces by Rosenthal's classical results, \cite{Ros0}. 

Note that \cite{TPFLCS} uses some similar notation ($(R_{1})$ and $(R_{2})$) to represent similar concepts ($\Ros$ and $\Rl$, respectively). 
Some authors (e.g., \cite{TPFLCS} and \cite{GSchur}) say that a lcs $E$ has the \emph{Rosenthal property}
if it satisfies the Rosenthal dichotomy (every bounded sequence has a subsequence that is either weak Cauchy or equivalent to the $l^{1}$-basis).
%2208q to be 
%ensuring the conformity to this dichotomy.
In this paper, we always refer to Definition \ref{d:AllDef1}. 

\sk 
In  Section \ref{s:RosType}  %Theorem \ref{thm:tame_iff_R} 
we 
%1802
prove the following theorems:

\sk 
%The class $\Tame$ contains the class $\Ros$ of all \textit{Rosenthal} lcs which was studied in the literature sometimes under different names. 

%In Section \ref{s:RosType} we show the following
%1802 Added theorem number in section 7
\begin{thm}[\ref{t:coinc}] \label{t:intro1}
	For any lcs we have 
	$
	\Ros \Longrightarrow \Tame = \Rl \Longrightarrow \ORl.
	$	
\end{thm}

\sk 
Note that $\Ros \neq \Tame$ (Theorem \ref{p:1And3AreNotReversibleINTRO}) and $\Rl \neq \ORl$ (Example \ref{ex:RR}). 
%0507 
For every locally complete lcs we have $\Rl= \ORl$ (Lemma \ref{lemma:overline_R2_equivalence}). 
%It is (expected) but still unclear if $\mathrm{(strong \ tame)} \neq \mathrm{(tame)} $. 

\begin{thm} 
	%1912  using \rm{}   
	\rm{\textbf{[Tame dichotomy in lcs]}} 
	Let $E$ be a locally convex space.
	Then every bounded 
	subset in $E$ is either tame, or has a subsequence equivalent to the $l^1$-sequence. %1009  usual $l^{1}$-basis.
	%1009q Matan, what if we write:	Then every bounded sequence in $E$ either has a tame subsequence or an $l^1$-sequence.
	%Is this formally weaker "sequential version" really weaker ? Equivalent ? 
	%1009q you already answered me in email about the previous query. So, I ask similar but different question. 
	%Q: Is it true that every bounded sequence in $E$ either contains a fragmented subsequence or a a subsequence equivalent to the $l^1$-sequence.
	
	%1009q every weak Cauchy sequence ia fragmented, right ? 
\end{thm}

\sk 
%1802 Added theorem number in section 7
\begin{thm}	[\ref{t:BoundMetr}]\label{t:intro2} 
	If all bounded sets of a lcs $E$ are metrizable, then 
	$
	\Ros =  \Tame  = \Rl,
	$
	%2507 
	and the following \emph{generalized Rosenthal's dichotomy} holds: 
	any bounded sequence in $E$ either has a weak Cauchy subsequence or an $l^1$-subsequence.  
\end{thm}

%So, one of the applications of our results is to 
%strengthen results of Ruess \cite{Ruess}. 

The latter result gives, as a corollary, a 
%1802 changed "well known" to "well-known" globbaly
well-known result of Ruess which extends 
Rosenthal's non-containtment of $l^1$-criteria   
% Rosenthal's dichotomy 
to a quite large class of lcs.

%2208 However, we will always use the term \emph{Rosenthal space} in accordance with Definition \ref{d:AllDef1}.

%1802	Changed Theorem to thm globally  

\begin{f} \label{fact:Ruess} \emph{(Ruess \cite[Thm~2.1 and Prop.~3.3]{Ruess})}   
	Let $E$ be a locally complete lcs with metrizable 
	%0108q what about "separable" ?
	%0908	We decided to remove it since it wasn't very important and caused some citation dilemmas
	bounded sets. Then 
	$\Ros = \ORl$ 
	and the following dichotomy holds: 
	any bounded sequence in $X$ either has a weak Cauchy subsequence or a subsequence which spans an isomorphic copy of $l^1$. 
\end{f}

The following result shows the limitations in general lcs for the existence of a Rosenthal type dichotomy. 

%1802 added reference to theorem in section 7
\begin{thm} [\ref{p:1And3AreNotReversible}] \label{p:1And3AreNotReversibleINTRO} There exists a 
	%0605 I think by the argument I mentioned in my email we can add here "strongly" ... 
	%1606 You are right
	%0507 strongly 
	tame complete (even reflexive) lcs which: 
	\begin{enumerate}
		\item [(i)] is not a Rosenthal lcs;  
		%in particular, $\Ros \neq \sTame.$ 
		%does not satisfy R1;  
		%		
		%		The implications (1) and (3) in Lemma \ref{lemma:R_tameness_relation} are not reversible (even for $\NP$ complete spaces). 
		\item [(ii)] does not contain any $l^1$-subsequence;  %(hence, also any subspace which is isomorphic to $l^1$). 
		
		\item [(iii)] contains a dense, Rosenthal subspace.
		
	\end{enumerate}  
	As a corollary: 
	Rosenthal's dichotomy 
	%0806 what are these two forms ??
	%1606 I have no idea, removed this bit
	%2206(Important) remark from 1606
	does not hold for such locally convex spaces. 
	%2208q is it true that in every lcs $E$ any bounded sequence either has a tame subsequence or an independent subsequence ? 
	%2208a Something even stronger can be said
	%0109 By Theorem \ref{thm:tame_iff_R}, every bounded sequence is either tame 
	%     or has a subsequence equivalent to the usual $l^{1}$-basis.
	%0609q  Good !! So, it should be carefully formulated and popularized !
	%0909	Added Corollary \ref{cor:tame_dychotomy}
\end{thm}

%0808q 
This also shows that %the class 
$\Ros$ is not closed under the completion. The same is unclear for $\Tame$. 

\sk 
%2208 
For every lcs $E$, there exists the strongest %(Hausdorff) 
topology between all locally convex tame topologies which are weaker than 
%0202 $\tau$.
the original topology. 
%0109 You correctly reminded me that the weak topology is always tame, so we don't need to require compatibility.
%and compatible with the duality $\lan E,E^{*}\ran$. 
This is proved in 
%1611  wrong label  \ref{thm:strongest_tame_topology}
Theorem \ref{t:str-t} using the bipolar theorem. 
%2012
In fact, it is proven for every polarly compatible bornological class (Definition \ref{defin:polarly_compatible}).
%%

%0808 
In Theorem \ref{t:Factoriz} we apply the DFJP technique \cite{DFJP} and show that 
	every tame (NP, DLP) operator $T \colon E \to F$ between 
	%0604 	lcs can be factored through a tame (Asplund, DLP) lcs.  
a lcs $E$ and a Banach space $F$ can be factored through a tame (Asplund,  reflexive) Banach space.

%2609 new subsection (Haydon ...) 
\sk 
\subsection*{Haydon's theorem for tame locally convex spaces} 
%0812 Adding a remark comparing Mazur's thm with Haydon. 
Recall that, according to Mazur's theorem, weak and norm closures are the same for convex subsets in Banach (or, even in locally convex) spaces. This property for weak-star closure in the dual is not true in general.  Haydon's theorem comes as an important compromise.
%0303 
It generalizes an earlier result for separable 
%1503 
Banach 
spaces which was proved by Odell and Rosenthal in \cite{Odell}. 

%1503 new place 
%010322
In Section \ref{s:haydon}, we prove a generalized version of Haydon's theorem for locally convex spaces.  
Theorem \ref{thm:Haydon} states that for a lcs $E$ the following conditions are equivalent:
%a lcs $E$ is tame 

%\begin{thm} 
%	For a locally convex space $E$, the following are equivalent:
	\ben
%	\item $E$ satisfies $\Rl$.
	\item $E$ is tame.
	\item Every weak-star compact, equicontinuous convex subset of $E^{*}$ is the strong closed convex hull of its extreme points.
	\item For every weak-star compact, equicontinuous subset $K$ of $E^{*}$ we have:
	$$
	\overline{\co}^{w^{*}}(K) = \overline{\co}(K).
	$$
	\een 
%\end{thm}

%1503 removong the promo about representations of DS. The essential part went to section 11
 
%1503 
\subsection*{Open questions} 
See \ref{q:completion},   \ref{p:TameInj}, \ref{conj:strong_tame_topology}, \ref{q:multi}. %\ref{q:FreeLCS}.  

\sk 
\nt \textbf{Acknowledgment:} We are 
%0604   :-) 
very 
 grateful to Saak Gabriyelyan for 
 %several 
 many 
 important suggestions and corrections. 

\sk 
\section{Definitions: fragmentability, independence and tameness}

\subsection*{Topological concepts}  
All topological spaces below are assumed to be 
%090122 completely regular. 
completely regular and Hausdorff 
(that is, Tychonoff). 
%290322	Saak said that it is not a standard definition. I couldn't find a reference so far
Recall that a function $f \colon X \to Y$  between topological spaces  is said to be a \textit{Baire class 1 function} 
%30mm 
\cite{Kech} 
if the inverse image of every open set is $F_\sigma$.  $f$ has the \textit{point of continuity property} (in short: PCP) if for every closed
nonempty $A \subset X$ the restriction $f_{|A}\colon A \to Y$ has a continuity point.  

%1012 Moved earlier by your suggestion
%0812 I think better to give this basic notation before ..
%2911 Moved here because some references were made to balance hulls in the DLP section
%2903	Added to support locally complete spaces

%1912 
\sk 
\subsection*{Locally convex spaces}  
We include the following standard definitions. 
% for the convenience of the reader. 
%0202 so, sometimes, in formulations we can skip saying E is a lcs, etc. ... 
In this work $E$ will usually denote a real locally convex space. 
%1910	Moved the exposition here
A subset $B \subset E$ 
%0202 of a real lcs $E$ 
is said to be \textit{bounded} if for every neighborhood $O$ of the zero in $E$ there exists $c \in \R$ such that $B \subset cO$. 
%3003
For every linear continuous operator $u\colon E_1 \to E_2$ and every bounded subset $B \subset E_1$ its image $u(B)$ is also bounded in $E_2$.   
%It is a nontrivial fact 
%0507  trying to reduce the number of text-books we use in citations
%\cite[Thm 3.18]{Rud}) 
Also, $B$ is bounded if and only if it is weakly bounded (see, \cite[Thm. 8.3.4]{Jarchow}). 
%0507 
The boundedness is \textit{countably determined}. That is, $B$ is bounded iff all its countable subsets are bounded.  
%2208 Denote by $b(V)$ the set of all bounded subsets of $V$.   

%1910	Removed since it is no longer needed
%Let $E^*$ be the dual of $E$. That is, the vector space of all continuous functionals $E \to \R$. We endow $E^*$ with the usual \textit{strong topology} denoted by $\mu^*$. It is the uniformity of uniform convergence on bounded subsets. For every bounded subset $A \subset E$ consider the seminorm 
%$$||f||_A:=\sup \{f(a): a \in A\}.$$
%Then $\{||\cdot||_A: A \ \text{is a bounded subset in} \ $E$\}$ is a fundamental system of seminorms for the strong uniformity $\mu^*$; 
%see 
%%0507 \cite[Ch. III, s. 3]{Bourb}. 
%\cite[Ch. IV, S. 5]{Schaefer}. 
%Under this topology $E^*$ is a lcs. 

\begin{defin} \label{d:gauge}  \
%0202 	Let $E$ be a locally convex space.   
	\ben
	\item 
	A subset %1912 $B \subseteq X$ 
	$S \subset E$ 
	is said to be
	\ben 
	\item 
	\emph{convex} if when $x, y \in S$ and $0 \leq \alpha \leq 1$ then $\alpha x + (1 - \alpha)y \in S$.
	%2911 New
	\emph{The convex hull $\co S$ of} $S$ is defined as the smallest convex set containing $S$. 
	Explicitly:
	$$
	\co (S) := \left\{ \sum_{n=1}^{N} \alpha_{n}x_{n}\mid \forall 1 \leq n \leq N: x_{n} \in S, \alpha_{n} \in [0, 1], \sum_{n=1}^{N} \alpha_{n} = 1 \right\}.
	$$
	\item \emph{balanced} if $\alpha S \subseteq S$ for every $\alpha \in \R$ satisfying $\lvert \alpha \rvert \leq 1$.
	\emph{The balanced hull $\bo S$ of} 
	%0812  why "bo", what is the role of "o"
	%1012  I thought that I saw this notation somewhere, but can't find it now.
	%1012  		I changed it to "bal"
	$A$ is defined as the smallest balanced set containing $A$. Explicitly:
	$$
	\bo S := \{\alpha x \mid x \in S, \alpha \in [-1, 1] \}.
	$$
	\item \emph{a disk} (or \emph{absolutely convex}) if it is both balanced and convex.
	%2911	New
	\emph{The absolutely convex hull $\acx S$ of} $S$ is defined as the smallest disked set containing $S$. Explicitly: 
	$${\acx S = \co(\bo S)}.$$
	
	\een

	\item A \textit{barrel} is a closed disk $S$ which is 
	%090122 "absorbing" is not defined ... might be misleading
	absorbing, meaning that
	%2012 changed to E
	$E = \bigcup\limits_{n \in \N} nS$.
	\item %2012	Changed to E
	The \emph{gauge} $q_{S} \colon E \to \R$ of $S$ is defined as:
	$
	q_{S}(x) := \inf \{r >0 \mid x \in rS\}.
	$
	%0808q  "locally complete" and "boundedly-complete", are these definitions somehow related ? 
	%0908 Yes! 
	\item %2012	Changed to E
	$E$ is said to be \emph{locally complete} if for every closed bounded disk $S \subseteq E$, the linear span 
	%0202 $E_{S}$ might be misleading  
	$Span (S) \subseteq E$ 
	of $S$ is complete with respect to $q_{S}$.
	\item %1912 new place. Better to move this def earlier because we already use the bipolar in DLP
	%1910	Moved here 
	As in \cite{Jarchow}, we 
	%%0507 will mark 
	denote the \textit{polar} of a subset $S$ of a locally convex space $E$ as:
	$$
	S^{\circ} := 
	\{ \varphi \in E^{*} \mid \forall x \in S: \lvert \varphi(x) \rvert \leq 1 \}.
	$$
	%1912
	Similarly, for every $S \subset E^*$ its polar is:   
	%2012	Removed * from inside the definition
	$$
	S^{\circ} := 
	\{ x \in E \mid \forall \varphi \in S: \lvert \varphi(x) \rvert \leq 1 \}.
	$$
	Now, for every $S \subset E$ its \textit{bipolar} $S^{\circ\circ}$ is defined as $(S^{\circ})^{\circ}$. 
	\een
\end{defin}

%1912 new place 
%1910	Made into a Fact
%1802 changed as miriam sugested
%1802 added ~ between paged reference globally
\begin{f}[Bipolar Theorem] \label{fact:bipolar_theorem}  \cite[p.~149]{Jarchow} 
%0202	Let $E$ be a locally convex space and 
For every $S \subseteq E$ %be a non-empty set.
	the bipolar $S^{\circ\circ} \subset E$ is equal to the weak
	%2012 removed: -star 
	closed absolutely convex hull $\overline{\acx}^{w} A$.
	%090122  what about the dual case ?  
\end{f}

\begin{f} \emph{(Mazur's Theorem \cite[p.~65, Cor.~2]{Schaefer})}  \label{f:Mazur}   
For every convex subset of a lcs $E$, its closure is identical with its weak closure. Hence, $\overline{co (S)}=\overline{co (S)}^w$ 
%and $\overline{\acx} S=\overline{\acx}^{w} S$ 
for every $S \subset E$.   
\end{f}

\sk 
%Every equicontinuous subset in the dual lcs is relatively weakly-star compact. % (p. 84). 
%1907  new place 
%0202 For every lcs $E$ denote 
Denote by $\eqc(E^*)$ the system of all equicontinuous  weak-star compact subsets in $E^*$ (\textit{equicontinuous compactology} in terms of \cite{Jarchow}). 
It is a basis of the system of all equicontinuous subsets in $E^*$ as it follows from 
%1802 added 's globally
Alaouglu-Bourbaki's theorem.

%1910	Moved here
\begin{f} \emph{(Alaouglu-Bourbaki)}  \label{l:Al-Bo} 
	%0202 Let $E$ be an lcs. 
	For every equicontinuous subset $A \subset E^*$, its weak star closure is equicontinuous and weak-star compact. 
\end{f}

Let $f \colon E_1 \to E_2$ be a continuous linear operator between lcs. 
Write $$f^*\colon E_2 \to E_1, \lan v, f^*(\varphi) \ran=\lan f(v), \varphi \ran$$
for its adjoint, where $$E \times E^* \to \R, (u,\varphi) \mapsto \lan v, \varphi \ran=\varphi(v)$$ is the canonical bilinear form. %Then its adjoint 

\sk 
%150122	Moved here
%0202 moved (now) here (togethr with Fact \ref{f:adjoint_is_continuous})
\begin{defin} \label{d:strongTOP} 
	Recall that the \textit{strong topology on the dual} $E^*$ of a lcs $E$ is the topology of bounded convergence. The standard uniformity of $E^*$ has the uniform basis $\{U[B,\eps]\}$, where $\eps >0$ and $B$ runs over all bounded subsets of $E$. Here 
	$$
	U[B,\eps]:=\{(\varphi_1,\varphi_2) \in E^* \times E^*: \ |\varphi_1(b)-\varphi_2(b)|< \eps \ \forall b \in B\}, 
	$$ 
	where $E \times E^* \to \R, \  (b,\varphi) \mapsto  \lan b, \varphi \ran=\varphi (b)$ is the usual bilinear map. 
\end{defin}

\begin{f} \cite[Thm. 8.11.3]{TVS} \label{f:adjoint_is_continuous}
	Suppose that $T \colon E \to F$ is a linear continuous operator between lcs. 
	%locally convex spaces.	If $T$ is continuous, then 
	Then it is also weakly continuous.
	Moreover, $T^{*}$ is both weak-star and strongly continuous.
\end{f}

%0808
%1109	Moved here
%0202 moved (now) here 
A (dense) subspace $F$ of a lcs $E$ is said to be \textit{large} in $E$ (see \cite[p.~254]{Perez}) if every bounded set  
%(bounded sets in lcs are defined at the beginning of Section \ref{s:Born}) 
in $E$ is contained in the closure of a bounded set in $F$. Every dense subspace in a normed space $V$ is large. Also, the same is true for every separable metrizable lcs $V$. 

%1109	Used for the proof of Haydon's theorem   
%2609 probably there exists a reference 
%090122 after submitting to arxiv, let's ask Leiderman about the reference 
%090122q strong topology is not defined in the text, as well as: m"strong open" 
\begin{lemma} \label{lemma:large_subspace_strong_isomorphism}
	%0202 If $E$ is a locally convex space, and 
	%290322	Rearranged the sentence
	Let $F$ be a large dense subspace of $E$ and $i\colon F \hookrightarrow E$ be the inclusion map.
	%080122 Changed E, F to E^{*}, F^{*}
	%290322	Saak suggested "strongly continuous isomorphism" which might be interperated as an isomorphism which is strongly continuous. 
	%	I've changed it to be clearer
	Then $i^{*}\colon E^{*} \to F^{*}$ is a topological isomorphism with respect to the strong topology.
\end{lemma} 
%0202 I suggest to remove the proof (maybe some reference ?) 
%\begin{proof}
%	It is easy to see that $i^{*}$ is a bijection since $F$ is dense in $E$.
%	Applying Fact \ref{f:adjoint_is_continuous}, all that is left is to show that $i^{*}$ is strongly open.
%	For every bounded subset $B \subseteq E$ and $\eps > 0$ define:
%	$$
%	W_{E}(B, \eps) := \{\varphi \in E^{*}\mid \forall x \in B: \lvert \varphi(x) \rvert < \eps  \}.
%	$$
%	We define $W_{F}(B, \eps)$ analogously.
%	We will show that $i^{*}(W_{E}(B, \eps))$ is always an open neighborhood of zero in $F^{*}$.
%	By definition, for every bounded $B \subseteq E$ there exists a bounded $B' \subseteq F$ such that $B \subseteq \overline{B'}$.
%	It is easy to see that:
%	$$
%	W_{F}\left(B', \frac{1}{2}\eps\right) \subseteq 
%	i^{*}\left(W_{E}\left(B, \eps\right)\right).
%	$$
%\end{proof}

%1109	Stated here since we used it multiple times.
%090122q I think this should be Lemma giving the reference (because, the rest is an easy consequence by Alaouglu and very pronbably is a standard fact). At this moment I replace Prop to Lemma because saying Proposition, as usual means that it is an original result
%0202q  we already have a reference \cite[Cor. 8.7.2]{Jarchow}. New place (I think better to move this to the preliminaries)
%1102	Where is the other reference?
%1302q let's discuss 
\begin{lem} \label{f:equicontinuous_extension}
	Let %$E$ be a locally convex space, 
	$F \subseteq E$ be a subspace and $i\colon F \hookrightarrow E$ 
	%290322	Saak suggestion: Changed "is" to "be
	%30mm I think "is" what OK  
	is 
	the inclusion map.
	If $M \subseteq F^{*}$ is a weak-star compact equicontinuous subset, then there exists a 
	%1910	Strengthened this fact to a proposition: added "weak-star compact".
	weak-star compact, equicontinuous subset $N \subseteq E^{*}$ such that $M = i^{*}(N)$.
\end{lem}
%1403	Made shorter
\begin{proof}
	A consequence of \cite[Cor. 8.7.2]{Jarchow} and the Alaouglu-Bourbaki's Theorem (Fact \ref{l:Al-Bo}).
	%1910	Added proof since we claim more than is proven in Jarchow
%	By , we can find $N' \subseteq E^{*}$ which is equicontinuous and satisfies $M = i^{*}(N')$.
%	Also, by Fact \ref{f:adjoint_is_continuous}, $i^{*}$ is weak-star continuous, and therefore:
%	$$
%	M = i^{*}(N') \subseteq
%	i^{*}(\overline{N'}^{w^{*}}) \subseteq 
%	\overline{i^{*}(N')}^{w^{*}} = 
%	\overline{M}^{w^{*}} = M.
%	$$
%	%090122 
%	%	Note that we used here the fact that compact subsets are necessarily closed.
%	%	
%	%	Define $N := \overline{N'}^{w^{*}}$.
%	%	By the Alaouglu-Bourbaki Theorem (Fact \ref{l:Al-Bo}), $N$ is weak-star compact and equicontinuous.
%	%	Moreover, by our construction, $i^{*}(N) = M$, as required.
%	Note that we use here Alaouglu-Bourbaki's Theorem (Fact \ref{l:Al-Bo}).
%	Now, defining $N := \overline{N'}^{w^{*}}$ it is clear that $i^{*}(N) = M$, as required. 
\end{proof}

%0202 new place 
It is well-known that if $B$ is a bounded disk, then 
%090122 
its gauge %(Definition \ref{d:gauge}) 
$q_{B}$ is a norm. 
%0704:	new
%\begin{remark}
%0109 reformulated into the previous sentence.
%By \cite[Corollary 5.1.8]{Perez}, every sequentially complete space is locally complete and therefore so is every complete space.
%\end{remark}

%090122  
%2903	Added to support locally complete spaces
\begin{f} \label{f:local_disc_norm_is_finer} \cite[Proposition 3.2.2]{Perez}
	Let $B \subseteq E$ be a bounded disc in %a locally convex space 
	$E$. Then $(E_{B}, q_{B})$ is a normed space and its topology is finer than that induced by $E$.
\end{f}
%240222 New

%1403	Made shorter 
The following is a consequence of  
%1503 the following fact 
Fact \ref{f:local_disc_norm_is_finer} 
and \cite[p.~105 Prop.~1]{Jarchow}.
\begin{lem} \label{lemma:ball_of_gauge}
	Let $E$ be a locally convex space.
	If $A \subseteq E$ is bounded, closed and absolutely convex, then $A = B_{E_{A}}$ (the unit ball of 
	%2802 
	the semi-normed space 
	$(E_{A}, q_{A})$).
\end{lem}
%\begin{proof}
%	By \cite[p.~105 Prop.~1]{Jarchow}, we have:
%	$$
%	  \{x \in E_{A} \mid q_{A}(x) < 1 \} \subseteq A \subseteq B_{E_{A}}.
%	$$
%	Moreover, by Fact \ref{f:local_disc_norm_is_finer}, the topology of $E_{A}$ is finer than that of $E$.
%	As a consequence, $A$ is also closed in $E_{A}$.
%	Considering the previous equation we conclude that:
%	$$
%	  B_{E_{A}} = \overline{\{x \in E_{A} \mid q_{A}(x) < 1 \}} \subseteq \overline{A} = A \subseteq B_{E_{A}}.
%	$$
%	In other words, $A = B_{E_{A}}$.
%\end{proof}

\sk \sk 
\subsection*{Fragmentability} 
\label{s:fragment} 

The following definition is a generalized version of the concept of {\it fragmentability}. 

%The following definitions provide natural generalizations of the fragmentability concept \cite{JR}. 
%%

\begin{defin} \label{def:fr} \cite{JOPV,Me-fr} 
	Let $(X,\tau)$ be a topological space and
	$(Y,\mu)$ a uniform space. 
	$X$ is {\em $(\tau,
		\mu)$-fragmented\/} by a 
	(not necessarily continuous)
	function $f\colon X \to Y$ if for every nonempty subset $A$ of $X$ and every $\eps
	\in \mu$ there exists an open subset $O$ of $X$ such that $O \cap
	A$ is nonempty and the set $f(O \cap A)$ is $\eps$-small in $Y$.
	We also say in that case that the function $f$ is {\em
		fragmented\/} 
	%2206 minor change: replaced "Notation: " with "and write"
	and write $f \in {\mathcal F}(X,Y)$, whenever the
	uniformity $\mu$ is understood.
	If $Y=\R$ with its natural uniformity, then we write simply ${\mathcal F}(X)$.  
\end{defin}

When $Y=X, f={id}_X$ and $\mu$ is a
metric uniformity, we retrieve the usual definition of
fragmentability (more precisely, $(\tau,\mu)$-fragmentability) in the sense of Jayne and Rogers \cite{JR}.
Implicitly, it already appears in a paper of Namioka and Phelps \cite{NP}. 
%The topological concept of fragmentability comes from Banach space theory. 

If $f\colon (X,\tau) \to (Y,\mu)$ has PCP 
then it is fragmented. If $(X,\tau)$ is hereditarily Baire (e.g., compact, or Polish) and $(Y,\mu)$ is a pseudometrizable uniform space, then $f$ is fragmented iff $f$ has PCP. If $X$ is Polish and $Y$ is a separable metric space, then
$f\colon X \to Y$ is fragmented iff $f$ is a Baire class 1 function. See \cite{GM1, GM-rose}. 

%\begin{remark}
%	In Definition \ref{def:fr}, it is enough to consider closed subsets $A \subseteq X$.
%\end{remark}

%2206(Important)	Only the last assertion of this lemma is needed. Maybe remove the rest? %2306 Good point ! I think to remove (1) and (2) from this lemma but put them into the text before lemma. 

%0507 removing it at all 
%\begin{lem} \label{r:fr1} 
%	 \cite[Lemma 2.8.2]{GM-rose} 
%	Let $\a\colon X \to X'$ be a continuous onto map between compact
%	spaces. Assume that $(Y, \mu)$ is a uniform space, $f\colon X \to Y$
%	and $f'\colon X' \to Y$ are maps such that $f' \circ \a=f$. 
%	Then $f$ is a fragmented map if and only if $f'$ is a fragmented map. 
%\end{lem}

\begin{defin} \label{d:fr-family} \ 
	%290322	Saak wanted us to remove the enumeration, but we use it later in this page.
	\ben
	\item \cite{GM1}
	We say that a {\it family of functions} $\mathcal{F}=\{f\colon (X,\tau)
	\to (Y,\mu) \}$ is {\it fragmented} if the condition of Definition
	\ref{def:fr}.1 holds simultaneously for all $f \in \mathcal F$.
	That is, $f(O \cap A)$ is $\eps$-small for every $f \in \mathcal
	F$. 
	%	It is equivalent to say that the mapping
	%	$$
	%	\pi_{\sharp}\colon X \to Y^{\mathcal{F}}, \hskip 0.4cm
	%	\pi_{\sharp}(x)(f)=f(x)
	%	$$
	%	is $(\tau, \mu_U)$-fragmented, where $\mu_U$ is the uniform structure of
	%	uniform convergence on the set $Y^{\mathcal{F}}$ of all mappings
	%	from $\mathcal{F}$ into $(Y, \mu)$.
	
	\item \cite{GM-rose}
	$F$ is an \emph{eventually fragmented family} 
	if every 
	%3003
	%infinite subfamily $C \subset F$ contains
	%an infinite fragmented subfamily $K \subset C$.
	sequence in $F$ has a subsequence which is a fragmented family on $X$.
	\een
\end{defin}
%1304  I think we need to remove the proof which is published

Definition \ref{d:fr-family}.1 was introduced 
%1503 
in arxiv preprints of \cite{GM1} and also  
independently 
%1611 
(under the name: \textit{equi-fragmented}) 
in the Ph.D. Thesis of M.M. Guillermo \cite{Guil}. 
%(we thank B. Cascales for pointing out this reference).

%1304
\begin{lem} \label{lemma:fragmented_family_as_simple_fragmentation}
	Let  ${F}=\{f \colon (X,\tau) \to (Y,\mu) \}$ be a family of functions. 
	Then $F$ is a fragmented family 
	iff the mapping
	$
	\pi_{F} \colon X \to Y^{F}, \ 
	\pi_{F}(x)(f)=f(x)
	$
	is $(\tau, \mu_U)$-fragmented, where $\mu_U$ is the uniform structure of
	uniform convergence on the set $Y^{F}$ of all mappings
	from ${F}$ into $(Y, \mu)$. 
\end{lem}
\begin{proof}
	Straighforward. 
\end{proof} 

%090122 new place 
%0704:	New
%1910	Removed proof, only citation
\begin{lem} \label{l:surjection_of_fragmented_family}
	%1910 Added the case where \alpha is not surjective
	Let $\alpha \colon X \to X'$ be a continuous map between compact spaces, $(Y, \mu)$ be a uniform space and $F \subseteq Y^{X'}$ be a 
	%290322	Saak said that "bounded family" is unclear. What do you think?
	%30mm  I think "bounded" here should be removed at all. Probably we wrote this word by inertion 
	%bounded 
	family of functions.
	If $F$ is fragmented, then so is $F \circ \alpha \subseteq Y^{X}$.
	If $\alpha$ is surjective, then the converse is also true.
\end{lem}
\begin{proof}
	%1910	Added a different citation to better fit this lemma
	Combination of Lemma \ref{lemma:fragmented_family_as_simple_fragmentation} and \cite[Lemma 2.3.5]{GM-rose}.
\end{proof}

%2208  new lemma 
\begin{lem} \label{l:convFr} \ 
	\begin{enumerate}
		\item 	Let $F$ be a fragmented family of real valued functions on $X$. Then 
		%290322 Changed from \co to \acx by Saak suggestion
		$\acx (F)$ is also fragmented.
		\item \cite[Prop.~4.15]{GM-rose}	Let $F$ be an eventually fragmented family of real valued functions on a compact space $X$. Then $co (F)$ is also eventually fragmented. 
	\end{enumerate}
\end{lem}
\begin{proof} (1) If $f_i(D)$ is $\eps$-small for every $i = 1, \dotsc, n$ and 
%290322 Made necessary changes to apply for \acx instead of \co (by Saak suggestion).	
$\sum_{i=1}^n \lvert c_i \rvert \leq 1$, $c_i \in \R$, then
$\sum_{i=1}^n c_i f_i(D)$ is $\eps$-small. 
\end{proof}

%090122 
\begin{lem} \label{l:p-clFr} 
	Let $F \subset X$ be a fragmented family of functions from a topological space $X$ into a uniform space $Y$. Then the pointwise closure $\overline{F}^p$ is also a fragmented family. 
\end{lem}
\begin{proof}
	Let $A \subset X$ be a nonempty subset and $\eps \in \mu$. Choose $\delta \in \mu$ such that $\delta^3 \subset \eps$. There exist an open subset $O \subset X$ such that $O \cap A \neq \emptyset$ and $f(O \cap A)$ is $\delta$-small for every $f \in F$. Let $h \in \overline{F}^p$. For this $h$ and a given pair $x,y \in O \cap A$ (by definition of the pointwise topology), there exists $f_0 \in F$ such that 
	$$(h(x),f_0(x))  \in \delta, \  (f_0(y),h(y)) \in \delta.$$ 
	Since $f_0(O \cap A)$ is $\delta$-small, we have $(f_0(x),f_0(y)) \in \delta$. So, we obtain
	$(h(x),h(y)) \in \eps$. Therefore, $h(O \cap A)$ is $\eps$-small, as desired. 
	\end{proof}

%1109	Added for the "bornological class" section.
%1109	It is used to prove a generalized version of Lemma \ref{lemma:dense_larger_tame}. 
%0812 The following lemma can be strengthened assuming that we consider even the WEAK closure of F. Indeed, by the previous Lemma \ref{l:convFr}.2, the convex hull co(F) is eventually fragmented. Now, use that for convex subsets strong closure gives the same as the weak closure (Mazur thm). By the way, perhaps better to compare it with Haydon's thm as a compromize 
%1012	I think that the pointwise topology is a more suitable choice here. However,
%1012 Mazur's thm applies to the weak topology of C(X), induced from the regular measures of X. 
%1012	I don't think that this argument works because the pointwise topology is different from the weak topology:
%1012	https://www.sciencedirect.com/science/article/pii/S0022247X17302263

%0704 
The following lemma is inspired by results of Namioka and it can be deduced after some reformulations from \cite[Theorems 3.4 and 3.6]{N}. 

%0704 
\begin{lem} \label{t:countDetermined} \cite[Theorem 2.6]{GM-MTame}
	Let $F$ be a
	bounded
	family of real valued \textbf{continuous} functions on a compact %(not necessarily metrizable)
	space $X$.
	The following conditions are equivalent:
	\ben
	\item
	$F$ is a fragmented family of functions on $X$.
	\item
	Every \emph{countable} subfamily $C$ of $F$ is fragmented on $X$.
	\item
	For every countable subfamily $C$ of $F$
	%and every $\mu$-uniformly continuous bounded pseudometric $d$ on $Y$
	the pseudometric space $(X,\rho_{C})$ is separable,
	where
	$$
	\rho_{C}(x_1,x_2):=\sup_{f \in C} |f(x_1)-f(x_2)|.
	$$
	\een
\end{lem}

\begin{lemma} \label{lemma:closure_of_eventually_fragmented}
	%1109q	Do you think that if we require $X$ to be compact it will apply to the pointwise closure as well? %0812q Let's discuss 
	%1012	I don't think that this argument works because the pointwise topology is different from the weak topology:
	%1012	https://www.sciencedirect.com/science/article/pii/S0022247X17302263
	Let $X$ be a topological space. 
	If $F \subseteq C(X)$ is a 
	%090122 uniformly bounded, 
	(eventually) fragmented family, then so is its closure $\overline{F}$ in the uniform topology of $C(X)$.
\end{lemma}
\begin{proof}
	%090122 
	%290322 This sentence wasn't clear, I've changed it
	The case of fragmented families is a consequence of Lemma \ref{l:p-clFr}.
	We are left with the case of eventually fragmented families.
	%% 
	%1611 just a comment for us: for fragmented families even the pointwise closure is fragmented. Maybe better to mention this (separately ?)
	%2911	I think that if we mention it separately, we will probably have to add a proof. Unless there is a quick reference, i don't think that it is worth the added length. I've added it in the previous sentence.
	Suppose that $\{f_{n}\}_{n \in \N} \subseteq \overline{F} \subseteq C(X)$.
	By definition, we can find $\{g_{n}\}_{n \in \N} \subseteq F$ such that
	$$
	  \forall x \in X: \lvert f_{n}(x) - g_{n}(x) \rvert < \frac{1}{n}.
	$$
	Since $F$ is eventually fragmented, we can find a subsequence $\{n_{k}\}_{k \in \N} \subseteq \N$ such that $\{g_{n_{k}}\}_{k \in \N}$ is fragmented.
	We claim that $\{f_{n_{k}}\}_{k \in \N}$ is also fragmented.
	
	Let $A \subseteq X$ be non-empty and $\eps > 0$.
	%1910	Changed 1/2 to 1/3 (otherwise it does not work)
	By definition, there exists some open $O \subseteq X$ such that $A \cap O \neq \emptyset$ and $g_{n_{k}}(A \cap O)$ is $\frac{1}{3}\eps$-small for every $k \in \N$.
	Choose $x \in A \cap O$ and $n_{0} \in \N$ such that $\frac{1}{n_{0}} \leq \frac{1}{3}\eps$.
	Since $\{f_{n}\}_{n \in \N} \subseteq C(X)$, we can find a neighborhood $x \subseteq U \subseteq O$ such that for every $1 \leq m \leq n_{0}$, $f_{m}(U)$ is $\eps$-small.
	\ben
		\item If $n_{k} \leq n_{0}$, then $f_{n_{k}}(U \cap A) \subseteq f_{n_{k}}(U)$ is $\eps$-small by our construction.
		\item Otherwise, $\lVert f_{n_{k}} - g_{n_{k}} \rVert \leq \frac{1}{3}\eps$.
		Moreover, $g_{n_{k}}(U \cap A) \subseteq g_{n_{k}}(O \cap A)$ is $\frac{1}{3}\eps$-small.
		Therefore, we conclude that $f_{n_{k}}(U \cap A)$ is $\eps$-small.
	\een
	In either case, $f_{n_{k}}(U \cap A)$ is $\eps$-small, as required.
	Also, note that $x \in U \cap A \neq \emptyset$.
\end{proof}

%2012
\begin{cor} \label{cor:fragmented_bipolarity}
	Let $X$ be a topological space. 
	If $F \subseteq C(X)$ is a 
%090122 uniformly bounded, 
	(eventually) fragmented family, then so is $\overline{\acx}^{w} F$,
	%290322 Added clarification by Saak suggestion
	where the closure is taken with respect to the weak topology induced by the 
	%30mm what is the "maximum norm" ? 
	%310322	Supremum norm (only, the functions are continuous, so it is actually maximum). Changed to supremum
	 supremum norm.
\end{cor}
\begin{proof}
	%2012	Does this require a reference? 
	%090122 we need this in many places (the present citation on Jarchow is not what we need, I think. So I replaced it by a reference on \cite{Schaefer}) 
	By Mazur's Theorem (Fact \ref{f:Mazur}), $\overline{\acx}^{w} F = \overline{\acx} F$.
	Now, we can apply Lemma \ref{l:convFr} and Lemma \ref{lemma:closure_of_eventually_fragmented} to get the desired result.
\end{proof}

\sk 
%The dual $V^*$ of a Banach space $V$ is the Banach space of all continuous linear functionals on $V$. 
%Weak star topology, as usual, is the pointwise topology on $V^*$ with respect to $V$, when $V$ can think as a family of functions on $V$.  

%1705	reforlumated
\begin{remark} \label{r:reform}  \
	An important example for the use of fragmented families (Definition \ref{d:fr-family}) is in the case of bounded sets in Banach spaces.
	If $B \subseteq V$ is a bounded subset of a Banach space and $K \subseteq V^{*}$ is a weak-star compact subset, then we can view $B$ as a family of functions over $K$.
	In this case, $B$ is fragmented iff for every non-empty subset $A \subseteq K$ and $\eps > 0$ there exists a weak-star open subset $O \subseteq V^{*}$ such that $O \cap A$ is not empty and 
	$
	  \diam \{\langle v,x \rangle: x \in O \cap A, v \in B\} < \eps.
	$
	%If $B$ is the unit ball of $V$ then it is equivalent to require $O \cap K$ to be $\eps$-small. %0507  with respect to the norm of $V^{*}$.
\end{remark}

\sk 
For some other properties of fragmented maps and fragmented families,
we refer to \cite{N, JOPV, Me-fr, Me-nz, GM1, GM-rose, GM-tame}.  
Basic properties and more applications of fragmentability in topological dynamics can be found in \cite{GM-rose, GM-survey, GM-tame, Me-fr,Me-nz}.

\sk   
%2212  "Independent" instead of "Independence"
\subsection*{Independent and tame families of functions} 
A sequence of real functions  ${\{f_n\colon X \to \R\}_{n \in \N}}$  on a set
$X$ is said to be 
%2206(Important) added (combinatorically), is this ok?
(combinatorially)  
\emph{independent} (see \cite{Ros0,Tal}) if
there exist real numbers $a < b$ (\emph{bounds of independence}) such that
$$
\bigcap_{n \in P} f_n^{-1}(-\infty,a) \cap  \bigcap_{n \in M} f_n^{-1}(b,\infty) \neq \emptyset
$$
for all finite disjoint subsets $P, M$ of $\N$. 
%090122  In this case we will call $a$ and $b$ \emph{bounds of independence}. Every bounded independent sequence is an $l^1$-sequence \cite{Ros0}. 

%2208  new place 
\begin{defin} \label{d:TameFamily} \cite{GM-tame,GM-MTame}
	A bounded family $F$ of real valued (not necessarily, continuous)  functions on a set $X$ is a {\it tame family} if $F$ does not contain an independent sequence. 
\end{defin}

%0908	Added for use in \ref{thm:tame_iff_R}
\begin{lem} \cite[Lemma 6.4]{GM-MTame} \label{lem:surjective_image_of_independent_sequence}
	%1910 Changed the formulation a bit to apply for non-surjective maps.
	Suppose that $\pi \colon X \to Y$ is a map and $F \subseteq \R^{Y}$ is a family of bounded functions.
	%1611 I think we can remove the def 
%	We write
%	$$
%	  F \circ \pi := \{f \circ \pi \mid f \in F \} \subseteq \R^{X}.
%	$$
	If $F$ is tame then $F \circ \pi$ is tame.
	Moreover, if $\pi$ is onto, the converse is also true.
\end{lem}

%2206(Important)	Do you know a referenec for this fact? maybe we can remove the proof %2306 Right ! It should be known. I will try to find a reference. Meanwhile I remove the proof and put the following line
The following fact
%0507 expected reference (not exact, yet. I need to check the coordinates) 
 from \cite{Ko} can easily be derived using the finite intersection property characterization of the compactness. 
\begin{f} \label{lemma:independence_over_compact_sets} 
	Suppose that $\{f_{n}\}_{n \in \N}$ is 
	%290322	added "an independent" by Saak suggestion.
	%30mm   Right! We missed "independent". Very funny but Saak wrote "bounded"
	an independent 
	family of continuous functions over a compact $X$.
	Then there are $a < b \in \R$ such that for every disjoint, possibly infinite $P, M \subseteq \N$:
	$$
	  \bigcap_{n \in P} f_n^{-1}(-\infty,a) \cap  \bigcap_{n \in M} f_n^{-1}(b,\infty) \neq \emptyset.
	$$
\end{f}
%2306 removing the proof 
%\begin{proof}
%	By definition, there are $a' < b' \in \R$ which satisfy this condition for every finite, disjoint $M, P \subseteq \N$.
%	Define $\eps := b' - a'$ and
%	$$
%	  a := a' + \frac{1}{3}\eps \text{ and } 
%	  b := b' - \frac{1}{3}\eps.
%	$$
%	We now have for every finite, disjoint $M, P \subseteq \N$:
%	$$
%	\bigcap_{n \in P} f_n^{-1}(-\infty,a] \cap  \bigcap_{n \in M} f_n^{-1}[b,\infty) \neq \emptyset.
%	$$
%	%1705	More elaboration
%	For every $A \in \N$ define
%	$$
%	  A^{(k)} := \{n \in A \mid n \leq k\}.
%	$$
%	It is easy to see that if $M$ and $P$ are not necessarily finite then
%	$$
%	  \bigcap_{n \in P} f_n^{-1}(-\infty,a] \cap  \bigcap_{n \in M} f_n^{-1}[b,\infty) =
%	  \bigcap\limits_{k \in \N} \left(\bigcap_{n \in P^{(k)}} f_n^{-1}(-\infty,a] \cap  \bigcap_{n \in M^{(k)}} f_n^{-1}[b,\infty)\right)
%	$$
%	Because $\{f_{n}\}_{n \in \N}$ are continuous, those inverse images are closed sets which satisfy the finite intersection property.
%	Recall that $X$ is compact and therefore even infinite intersections are non-empty.
%\end{proof}

\sk 
By \cite{Me-Helly}, 
every bounded family of (not necessarily continuous) functions 
$[0,1] \to \R$ with total bounded variation (e.g., \textit{Haar systems}) is tame.  This remains true replacing the set  $[0,1]$ by any circularly (e.g., linearly) ordered set.  

As to the negative examples. 
The sequence of projections on the Cantor cube %0808 $$\{\pi_n\colon \{0,1\}^{\N} \to \{0,1\}\}_{n \in \N}$$ 
and the sequence of Rademacher functions 
%0808 
on the unit interval 
%$$\{r_n\colon [0,1] \to \R\}_{n \in \N}, \ \ r_n(x):=sgn (\sin (2^n \pi x))$$
both are independent (hence, nontame).  

%0806 
%2206 Minor change: reformulation 
A critically important example of a nontame sequence is the standard basis sequence ${\{e_n: n \in \N\}}$ in $l^1$ as a family of functions on the unit ball
%2012	Changed to l^{\infty} rather than l_{\infty}
 of $(l^1)^*=l^{\infty}$. 

\sk 
The following useful lemma synthesizes some known results.
It mainly is based on results of Rosenthal and Talagrand. 
%280817 [avoiding that guy] and van Dulst \cite[Thm. 3.11]{Dulst}. 
%280817 [I also added an explanation]  
The equivalence of (1), (3) and (4) is a part of \cite[Thm.~14.1.7]{Tal}. 
For the case (1) $\Leftrightarrow$ (2), note that every bounded independent sequence $\{f_n\colon X \to \R\}_{n \in \N}$ is 
an $l^1$-sequence (in the $\sup$-norm), \cite[Prop.~4]{Ros0}. On the other hand, as the proof of \cite[Thm.~1]{Ros0} shows, if $\{f_n\}_{n \in \N}$ has no independent subsequence then it has a pointwise convergent subsequence. Bounded pointwise-Cauchy sequences in $C(X)$ (for compact $X$) are weak-Cauchy as it follows by Lebesgue's theorem. Now Rosenthal's dichotomy theorem \cite{Ros0} asserts that $\{f_n\}$ has no $l^1$-sequence.  
In \cite[Sect.~4]{GM-rose} we show why eventual fragmentability of $F$ can be included in this list (item (5)).  
%% 

%0604 some changes (let's discuss)
\begin{lem} \label{f:sub-fr} 
	Let $X$ be a compact space and $F$ is a bounded subset 
	in the Banach space  $C(X)$. 
	The following conditions are equivalent:
	\begin{enumerate}
		\item
		$F$ does not contain an $l^1$-sequence. 
		\item $F$ is a tame family on $X$. %1611 (does not contain an independent sequence). 
		%0604q  let's discuss -- I think we need to mention (here or somewhere) that it is equivalent being a tame subset in the Banach space C(X) (that is, F is tame over B_{C(X)} ).  At least this follows from my work with Glasner [14, Prop. 4.19] or/and some bornology results in our work 
		\item
		Each sequence in $F$ has a pointwise convergent subsequence in $\R^X$.
		\item 
		%0808 $F$ is a \emph{Rosenthal family} (meaning that the pointwise closure ${\cls}(F)$ of $F$ in $\R^X$ consists of fragmented maps; that is,
		%${\cls}(F) \subset {\mathcal F}(X)$. 
		 The pointwise closure ${\cls}(F)$ of $F$ in $\R^X$ consists of fragmented maps. 
		\item $F$ is an eventually fragmented family on $X$.
	\end{enumerate}
\end{lem}

%0806  I remove the following lemma because we do not use it 
%Recall also some additional useful results.   
%\begin{lem} \label{l^1} \cite{GM-MTame} \
%	\ben
%	\item
%	Let $q \colon X_1 \to X_2$ be a map between sets and
%	$\{f_n \colon X_2 \to \R\}_{n \in \N}$ a bounded sequence of functions \emph{(with no continuity assumptions on $q$ and $f_n$)}.
%	If $\{f_n \circ q\}$ is an independent sequence on $X_1$
%	then $\{f_n\}$ is an independent sequence on $X_2$.
%	\item
%	If $q$ is onto then the converse is also true. That is $\{f_n \circ q\}$ is independent if and only if $\{f_n\}$ is independent.
%	\item
%	Let $\{f_n\}$ be a bounded sequence of continuous
%	functions on a topological space $X$.
%	Let $Y$ be a \emph{dense} subset of $X$. Then $\{f_n\}$ is an independent sequence on $X$ if and only if
%	the sequence of restrictions
%	$\{f_n|_Y\}$ is an independent sequence on $Y$.
%	\een
%\end{lem}

%1304 removing this lemma. We use it only once. So we can give a direct reference
%\begin{lem} \label{c:conv} \cite[Prop. 4.15]{GM-rose}
%	Let $F \subset C(X)$ be a tame (eq. Rosenthal) family for a compact space $X$. Then its
%	convex hull ${co}(F)$ is also a tame family for $X$.
%\end{lem}
%%\begin{proof}
%%	This is \cite[Prop. 4.15]{GM-rose}. 
%%\end{proof}

\sk 
\subsection*{Rosenthal's dichotomy and Rosenthal's Banach spaces} 
\label{s:RosDich}  

For every topological space $X$ denote by 
$C(X)$ the vector space of all continuous real functions. 
When $X$ is compact, as usual, we suppose that $C(X)$ is endowed with the supremum norm. So, it will be a Banach subspace of $l^{\infty}(X)$. 

Let $\{f_n \colon X \to \R\}_{n \in \N}$ be a %090122 uniformly 
bounded sequence %090122 of real valued 
of functions on a \emph{set} $X$. Following Rosenthal \cite{Ros0}, we say that
this sequence is an \emph{$l^1$-sequence} %on $X$ 
if there exists a %090122 real 
constant $\delta >0$
such that for all $n \in \N$ and choices of real scalars $c_1, \dots, c_n$, we have
$$
\delta \cdot \sum_{i=1}^n |c_i| \leq \left\lVert\sum_{i=1}^n c_i f_i\right\rVert_{\infty}.
$$
Then the closed linear span of $\{f_{n}\}_{n \in \N}$ in $l_{\infty}(X)$ is linearly homeomorphic to the Banach space $l^1$.
In fact, in this case the map $l^1 \to l_{\infty}(X), \ (c_n)_{n \in \N} \to \sum_{n \in \N} c_nf_n$ is a linear homeomorphic embedding.

A sequence of vectors in a Banach space can be defined  equivalent to an $l^1$-sequence analogously.  According to Rosenthal's %celebrated 
dichotomy, every bounded sequence in a Banach space either has a weak Cauchy subsequence or 
%2012q I don't think that "admits an l1 sequence" is a good description 
%090122q   let's discuss  
admits
%1611 a subsequence equivalent to the unit vector basis of $l^1$ (an $l^1$-\emph{sequence}).
an $l^1$-\emph{sequence}. 
%2410
%Consequently a Banach space $V$ does not contain an
%$l^1$-sequence if and only if every bounded sequence
%in $V$ has a weak-Cauchy subsequence \cite{Ros0}.
%2610
%2610e
%Thus, a Banach space $V$ does not contain an isomorphic copy of
%$l^1$ (equivalently, does not contain an $l^1$-sequence) if and
Thus, a Banach space $V$ does not contain an  $l^1$-sequence
(equivalently, does not contain an isomorphic copy of $l^1$) iff every bounded sequence in $V$ has a weak Cauchy
subsequence, \cite{Ros0}. %%
As in \cite{GM-rose,GM-survey}, we call a Banach space satisfying these
equivalent conditions a {\em Rosenthal Banach space}.

%A Banach space $V$ is said to be {\em Rosenthal} if it does
%not contain an isomorphic copy of $l^1$, or equivalently, if $V$ does not contain a sequence 
%of vectors which is equivalent to an $l^1$-sequence. 

%\sk
%
%For every topological space $X$ denote by $C(X)$ the Banach space of all bounded continuous real functions with the $\sup$ norm. 

\begin{defin}  \label{d:tameF} 
	Let $V$ be a %0507 Banach 
	normed space and $M \subset V^*$ be a subset in the dual space $V^*$. A  bounded  subset $F$ of $V$ is said to be \emph{tame for} $M$ if $F$, as a family of functions on $M$, is a tame family. If $F$ is tame for the unit ball $B_{V^*}$ of $V^*$ (equivalently, for every bounded subset), then we simply say that $F$ is a \textit{tame subset in} $V$. 
\end{defin}

\begin{lemma} \label{lemma:specific_l1_when_not_tame}
	Let $V$ be a normed space, $A \subseteq V$ and $M \subseteq V^{*}$ be bounded subsets.
	If $A$ is not tame over $M$, then $A$ contains %1611 a sequence equivalent to the %1109	usual $l^{1}$-basis.
	an $l^1$-sequence. 
\end{lemma}
\begin{proof}
	It is a known consequence of the Hahn-Banach theorem that $V$ is isometrically embedded into $C(B_{V^{*}})$.
	Applying Lemma \ref{f:sub-fr}, we get the desired result.
%%0507 	First recall that 
%As a consequence of the Hahn-Banach Theorem, for every $x \in A$ we have 
%	$
%	\lVert x \rVert =
%	\sup\limits_{\varphi \in B_{V^*}} \lvert x(\varphi) \rvert.
%	$
%	%1805 maybe better to replace $B$ by some other letter 
%	%0607	Changed ball notation to B_{V^{*}}
%	Thus, $A$ is isometrically contained in $C(B_{V^{*}})$.
%	Since $M$ is bounded, there is some $\eps > 0$ such that $\eps M \subseteq B_{V^{*}}$.
%	Because $A$ is not tame over $M$, it is obvious that $A$ is also not tame over $B_{V^{*}}$.
%%1907	By the Banach-Alaoglu Theorem, 
%Clearly, $B_{V^{*}}$ is weak-star compact.
%	By Lemma \ref{f:sub-fr}, we know that $A \subseteq C(B_{V^{*}})$ contains an $l^{1}$-sequence.
%	Because this is an isometric embedding, we conclude that $A \subseteq V$ contains an $l^{1}$-sequence.
\end{proof}

The following characterization of Rosenthal Banach spaces is a reformulation of some known results (see, in particular, \cite{SS} and Lemma  \ref{f:sub-fr}).

\begin{lem} \label{l:RosBanSpCharact} Let $V$ be a Banach space. 
	The following conditions are equivalent:  
	\ben
	\item $V$ is a Rosenthal Banach space; 
	%0808 the following item is removable
		%0908 I Removed it
		%2208q  I added it back (let's discuss) 
	\item 
	each $x^{**} \in V^{**}$ is a fragmented map when restricted to the
	weak${}^*$ compact ball $B_{V^{*}}$ of $V^*$. Equivalently, $B_{V^{**}} \subset \F(B_{V^{*}})$; 
	\item the unit ball $B_{V}$ is a tame subset of $V$;
	\item any bounded subset of $V$ is tame for any bounded subset of $V^*$. 
	\een
\end{lem}
\begin{proof}
	$(1) \Rightarrow (4)$ A consequence of Lemma \ref{lemma:specific_l1_when_not_tame}.
	
	$(4) \Rightarrow (3)$ Trivial.
	
	$(3) \Rightarrow (2)$ Suppose that $B_{V}$ is a tame family over 
	%2208 $B_{V^{**}}$.
	$B_{V^{*}}$.
	Using Lemma \ref{f:sub-fr}, we can conclude that $\cls_{p}(B_{V}) \subseteq \mathcal{F}(B_{V^{*}})$.
	%We claim that $B(V^{**}) \subseteq \cls_{p}(B(V))$, proving our claim.
	On the other hand, $B_{V^{**}} = \cls_{p}(B_{V})$ by Goldstein's theorem. Hence, $B_{V^{**}} \subset \F(B_{V^{*}})$. 
	
	$(2) \Rightarrow (1)$
	Use \cite[Thm.~3]{SS}. 
\end{proof}

%!!!! meore 
%So, in particular, a Banach space $V$ is Rosenthal iff 
%every bounded subset $F \subset V$ is tame  (as a family of functions) on every bounded subset $Y  \subset V^*$ of the dual space $V^*$, iff $F$ is eventually fragmented (Definition \ref{def:fr}.3) on $Y$.  

%1611 removing these 3 lines 
%A Banach space $V$ is {\em Asplund\/} if the dual of every separable Banach subspace is separable iff every bounded subset $F \subset V$ is a fragmented family  of functions on every bounded subset $Y \subset V^*$.  
%Every Asplund space is Rosenthal 
%%because the dual of $l^1$ is $l_{\infty}$, a nonseparable space 
%and every reflexive space is Asplund. 

\sk \sk 
\subsection*{The Double Limit Property (DLP)}
Recall %the 
Grothendieck's double limit property. 

\begin{defin} \label{d:DLP} 
	%090122  replacing X by K 
	Let $F \subset \R^{K}$ be a family of real functions on a set $K$. 
	Then $F$ is said to have the \textit{double limit property} (DLP) %on a subset $Y_1 \subset Y$ 
	if for every sequence $\{f_n\}_{n \in \N}$ in $F$ 
	and every sequence $\{x_n\}_{n \in \N}$ in $K$, the limits 
	$$
	\lim_n \lim_m f_n(x_m) \ \ \ \text{and} \ \ \ \lim_m \lim_n f_n(x_m)
	$$
	are equal whenever they both exist. 
\end{defin}

%2609 
We will often write that a subset \emph{is} DLP rather than the more correct \emph{``has the DLP"}. 
%090122 
The following properties are easy to verify.

\begin{lemma} \label{lemma:double_subsequence_of_bounded_functions_is_DLPable} \ 
	\begin{enumerate}
		\item 
		If $\{f_{m}\}_{m \in \N}$ 
		is a %uniformly 
		bounded 
		sequence of functions on $K$ and $\{x_{n}\}_{n \in \N} \subseteq K$, then there exist subsequences $\{n_{k}\}_{k \in \N}, \{m_{t}\}_{t \in \N} \subseteq \N$ such that 
		$$
		\lim_{k \in \N} \lim_{t \in \N} f_{n_{k}}(x_{m_{t}}) \text{ and }
		\lim_{t \in \N} \lim_{k \in \N} f_{n_{k}}(x_{m_{t}})
		$$
		exist.
		\item  If $A \subset l^{\infty}(K)$ is a bounded family of functions over $K$ satisfying the DLP, then so does the 
		balanced hull $\bo A$ (see Definition \ref{d:gauge}). 
		\item 	If $A_{1}, A_{2}$ are bounded sets of functions over $K$ satisfying the DLP, then so does $A + B$.
		%0202 moved here 
		\item Suppose that $\varphi \colon K_{1} \to K_{2}$ is a continuous map and 
		%090122q %0202 lets' discuss here and (in similar cases) we should mention what we mean. $F$ is DLP as a subset in the Banach space $C(K_1)$ or $F$ is DLP on $K_1$. Or, to say as a promo that this is the same
		$F \subseteq C(K_{2})$ is DLP.
		Then $F \circ \varphi \subseteq C(K_{1})$ is DLP.
		
		Moreover, if $\varphi$ is surjective, then the converse is also true. 
		Namely, if $F \circ \varphi \subseteq C(K_{1})$ is DLP then so is $F$.
	\end{enumerate} 
\end{lemma}
%290322	Added by request of Saak
\begin{proof} ~
	\ben
		\item We will show a more general fact.
		Suppose that $A$ and $B$ are sets and $\langle \cdot, \cdot \rangle \colon A \times B \to \R$ is a map.
		Also, let $\{a_{n}\}_{n \in \N} \subseteq A$ and $\{b_{m}\}_{m \in \N} \subseteq B$ be sequences such that $\langle \cdot, \cdot \rangle$ is bounded over them. 
		Then there exist subsequences $\{n_{k}\}_{k \in \N}, \{m_{t}\}_{t \in \N} \subseteq \N$ such that the limit 
		$$
		\lim_{k \in \N} \lim_{t \in \N} \langle a_{n_{k}}, b_{m_{t}} \rangle
		$$
		exist.
		By applying this fact twice we will get the desired result.
		
		First, note that the sequence $\{\langle a_{n}, b_{m} \rangle\}_{m \in \N} \subseteq \R$ is bounded for every $n \in \N$.
		We can thus use induction and a diagonal argument to construct a subsequence $\{m_{t}\}_{t \in \N}$ such that $\{\langle a_{n}, b_{m_{t}} \rangle\}_{t \in \N}$ converges for every $n \in \N$. 
		Write $\alpha_{n} := \lim\limits_{t \in \N}\langle a_{n}, b_{m_{t}} \rangle$.
		Note that $\{\alpha_{n}\}_{n \in \N} \subseteq \R$ is a bounded sequence so there exist a subsequence $\{n_{k}\}_{k \in \N} \subseteq \N$ such that $\{\alpha_{n_{k}}\}_{k \in \N}$ converges.
		However:
		$$
		\lim_{k \in \N} \alpha_{n_{k}} = 
		\lim_{k \in \N} \lim\limits_{t \in \N}\langle a_{n_{k}}, b_{m_{t}} \rangle,
		$$
		as required.
		\item Suppose that $\{\alpha_{n}f_{n}\}_{n \in \N} \subseteq \bo A$ and $\{x_{m}\}_{m \in \N} \subseteq K$ such that
		$$
		\lim_{n \in \N} \lim_{m \in \N} \alpha_{n} f_{n}(x_{m}) \text{ and }
		\lim_{m \in \N} \lim_{n \in \N} \alpha_{n} f_{n}(x_{m})
		$$
		exist.
		By definition, $\{\alpha_{n}\}_{n \in \N} \subseteq [-1, 1]$ and is therefore bounded.
		Thus, we can apply Lemma \ref{lemma:double_subsequence_of_bounded_functions_is_DLPable} and the Bolzano–Weierstrass theorem  to find subsequences $\{n_{k}\}_{k \in \N} \subseteq \N$ and $\{m_{t}\}_{t \in \N} \subseteq \N$ such that:
		$$
		\lim\limits_{k \in \N} \alpha_{n_{k}},
		\lim_{k \in \N} \lim_{t \in \N} f_{n_{k}}(x_{m_{t}}) \text{ and }
		\lim_{t \in \N} \lim_{k \in \N} f_{n_{k}}(x_{m_{t}})
		$$
		exist.
		Moreover, since $A$ is DLP, we know that
		$$
		\lim_{k \in \N} \lim_{t \in \N} f_{n_{k}}(x_{m_{t}}) =
		\lim_{t \in \N} \lim_{k \in \N} f_{n_{k}}(x_{m_{t}}).
		$$
		Together we get:
		%2012 Fixed alignments
		\begin{align*}
			\lim_{n \in \N} \lim_{m \in \N} \alpha_{n} f_{n}(x_{m}) = &
			\lim_{k \in \N} \lim_{t \in \N} \alpha_{n_{k}} f_{n_{k}}(x_{m_{t}}) = 
			\left(\lim_{k \in \N} \alpha_{n_{k}} \right)\lim_{k \in \N} \lim_{t \in \N} f_{n_{k}}(x_{m_{t}}) = \\
			& \left(\lim_{k \in \N} \alpha_{n_{k}} \right)\lim_{t \in \N} \lim_{k \in \N} f_{n_{k}}(x_{m_{t}}) = 
			\lim_{t \in \N} \lim_{k \in \N} \alpha_{n_{k}} f_{n_{k}}(x_{m_{t}}) = \\
			& \lim_{m \in \N} \lim_{n \in \N} \alpha_{n} f_{n}(x_{m})
		\end{align*}
		as required. 
		
		\sk 
		%30mm 
		(3) and (4) are easy to check. 
%		\item Easy to see.
%		\item Easy to see.
	\een
\end{proof}
\begin{f} \emph{(N.J. Young \cite[Thm.~2]{Young})} \label{f:young}
	%2911	Changed to be closer to the original formulation to apply for both the regular case and the co-small case.
	%2012 Changed If to let
	Let $E$ and $F$ be topological vector spaces and ${A \subseteq E, B \subseteq F}$.
	Furthermore, let $\langle \cdot, \cdot \rangle\colon E\times F \to \R$ be a bilinear map
	%290322	Added "bounded on" rather than "bounded" to clarify (as to Saak question).
	bounded on $A \times B$.
	If $A$ has the DLP as a family of functions over $B$, then so does the bipolar  $A^{\circ\circ}$.
	
	\sk 
	If $E$ is locally convex, we can apply the Bipolar Theorem (Fact \ref{fact:bipolar_theorem}) to conclude that $\overline{\acx}^{w} A$ is DLP over $B$.
\end{f}

%2911	Combined with the following lemma (closure + convex hull)
%0202q  let's discuss 
\begin{lemma} \label{lemma:DLP_is_locally_convex}
	Suppose that $K$ is compact and $A \subseteq C(K)$ is a bounded family of functions that satisfies the DLP.
	Then so does 
	%090122  maybe you mean to write here $\overline{\acx}^{w} A$ (or, both of them) 
	%290322	Changed to \acx instead of \co
	$\overline{\acx}^{w} A$.
\end{lemma}
\begin{proof}
	%30mm  let's discuss 
	%2911	New
	%290322	Added "bounded" to clarify Saak question, also rearranged a little
	Suppose that $A \subseteq C(K)$ is a bounded family with the DLP over $K$.
	Now, consider $K$ as a subspace of the free topological vector space $L(K)$.
	By definition, every $f \in A$ can be extended uniquely 
	to a continuous linear 
	%0812 
	operator 
	on  $L(K)$.
	This gives a bilinear pairing $\langle \cdot, \cdot \rangle\colon A \times K$ defined by $\langle f, x\rangle := f(x)$.
	%290322	Added to clarify Saak question
	Also, since $A$ is bounded, so is the image of $A \times K$ under this bilinear map.
	By applying Fact \ref{f:young}, we conclude that
	%290322	Changed \co to \acx
	$\overline{\acx}^{w} A$ is DLP over $K$.
\end{proof}
\subsection*{Additional Preliminaries}

%1109	Maybe this is obvious. What do you think 
%0812 First of all, let's add the definition of $\ext$ (say, a new item in Definition 2.22). Then we can give this formula without proof 
%1012	Changed as you suggeted
Suppose that $E$ is a vector space and $A \subseteq E$ is a subset.
The open segment between two distinct points $x, y \in E$ is defined as
$$
  (x, y) := \{\alpha x + (1 - \alpha )y \mid 0 < \alpha < 1 \}.
$$
A point $x \in A$ is said to be an \emph{extreme point of $A$} if it does not belong to any open segment contained in $A$, \cite{Jarchow}.
If $A$ is convex, %it is easy to see that 
$x \in A$ is extreme in $A$ if and only if $x = \frac{1}{2}(a + b)$ for $a, b \in A$ implies $x = a = b$.
We write $\ext A$ for the set of all extreme points of $A$.

%150122 New (actually, a better formulation of the lemma about inverse images of extreme points)
%0202  I think it is better (now) to write two lemmas into one (as two items) 
\begin{lem}\label{lemma:extreme_points_of_compact_map} 
	%0202q I think we do not need to speak about not locally convex spaces 
	%0202 Suppose that $T\colon E \to F$ is continuous and linear, and $A \subseteq E$ is compact.
	Let $T\colon E \to F$ be a continuous and linear map between locally convex spaces and $A \subseteq E$ be compact. 
	Then:
	%0202
	\begin{enumerate}
		\item $
		\ext T(A) \subseteq T(\ext A);
		$
		\item if $T$, in addition, is injective then $\ext T(A) = T(\ext A)$;
		\item  	%let $T\colon E \to F$ be a map between locally convex spaces.
		suppose that $M \subseteq F^*$.
		Then
		$$
		T(T^{*}(M)^{\circ}) = (\im T) \cap M^{\circ}.
		$$
	\end{enumerate}  
\end{lem}
\begin{proof} (1) 
	Write $K := \ker  T$.
	Let $y \in \ext T(A)$.
	Since $T$ is continuous, $K_{y} := A \cap T^{-1}(\{y\})$ is a closed, non-empty subset of $A$.
	It is also compact because $A$ is compact.
	We claim that $\ext K_{y} \subseteq \ext A$.
	
	Indeed, suppose that $x \in \ext K_{y}$  and $x = \frac{1}{2}(a + b)$ for $a, b \in A$.
	As a consequence, 
	$${y = T(x)=\frac{1}{2}(T(a) + T(b))}.$$
	However, since $y \in \ext T(A)$, we conclude that $y = T(a) = T(b)$.
	By definition, $a, b \in K_{y}$.
	Finally, since $x \in \ext K_{y}$, we conclude that $x = a = b$, proving that $x \in \ext A$.
	
	Applying the Krein-Milman Theorem, $\ext K_{y} \neq \emptyset$.
	Choose $x_{0} \in \ext K_{y} \subseteq \ext A$.
	By definition, $y = T(x_{0}) \in T(\ext A)$, as required.
	
	\sk 
	%120222	Added a citation
	(2) \cite[Thm.~9.2.3]{TVS}.
%	Let $T(x) \in T(\ext A)$ for $x \in \ext A$.
%	By contradiction, suppose that $T(x) \notin \ext T(A)$.
%	Therefore, there exists $T(x) \neq T(a), T(b) \in T(A)$ such that
%	$$
%	T(x) = \frac{1}{2} (T(a) + T(b)).
%	$$
%	As a consequence, $T(x) = T\left(\frac{1}{2}(a + b)\right)$.
%	Since $T$ is injective, $x = \frac{1}{2}(a + b)$.
%	However, $x \in \ext A$ and therefore $a = b = x$.
%	This stands in contradiction to $T(x) \neq T(a), T(b)$.
	
	\sk 
	(3) \begin{align*}
		y \in T(T^{*}(M)^{\circ}) & \iff
		\exists x \in T^{*}(M)^{\circ}: y = T(x) \\
		& \iff \exists x \in E: y = T(x) \text{ and } \forall\ \varphi \in M: \lvert (T^{*}(\varphi))(x)\rvert \leq 1\\
		& \iff \exists x \in E: y = T(x) \text{ and } \forall\ \varphi \in M: \lvert \varphi(T(x))\rvert \leq 1\\
		& \iff y \in \im T \text{ and } \forall\ \varphi \in M: \lvert \varphi(y)\rvert \leq 1\\
		& \iff y \in (\im T) \cap M^{\circ}.
	\end{align*}
\end{proof}

\begin{defin} \label{d:rho_M} 
	Let 
	%0202 $E$ be a locally convex space, 
	$M \in \eqc(E^{*})$ and 
	 %equicontinuous, weak-star compact subset.
	$\rho_{M}$ be the continuous seminorm on $E$ defined by
	$$
	\rho_{M}(x) := \sup_{\varphi \in M} \lvert \varphi(x) \rvert.
	$$
	We say that $M$ is $\coRl$ if there is no bounded $l^{1}$-sequence in $E$ with respect to $\rho_{M}$.
\end{defin}
%2302 separated lemma to the normed case
%2302 Fixed problem with the new proof
\begin{lemma} \label{lemma:equicontinuous_factor_normed} 
	Let $M \subseteq E^{*}$ be an equicontinuous compact subset. There exist: a continuous, 
	%290322	Changed "dense" to "onto" by Saak suggestion
	onto and open 
	%290322 Removed: onto its image
	%0202 
	%30mm let's discuss 
	linear 
	map $\pi \colon E \to V$ to a normed space $V$ and a weak-star continuous linear operator ${\Delta \colon \spn M \to V^{*}}$ 
	%290322	Added by Saak suggestion
	with dense image, 
	such that
	%1302 replacing the order ... 
	$B_{V^{*}} = \overline{\acx}^{*}(\Delta(M))$ and 
	$
	\id_{M} = \pi^{*} \circ \Delta. 
	$
\end{lemma}
\begin{proof}
	Write $N := \overline{\acx}^{w^{*}} M$ and consider the space $W := \spn N$ with the gauge $q_{N}$ as norm.
	For every $x \in E$, consider the evaluation $e_{x}\colon N \to \R$ defined by $e_{x}(\varphi) := \varphi(x)$.
	%240222	Added reference
	By Lemma \ref{lemma:ball_of_gauge}, $B_{W} = N$.
	Also, $N$ is weak-star compact, so $e_{x}(N) = N(x)$ is a bounded subset of $\R$ for every $x \in E$.
	As a consequence, $e_{x} \in W^{*}$ is a bounded functional.
	
	%1302F2 changed so we later define V so the image is dense in it.
	Write $V_{1}:= W^{*}$ and define the 
	%1402 
	linear map 
	$\pi_{1}\colon E \to V_{1}$ by $\pi_{1}(x) := e_{x}$.
	Consider the continuous seminorm $\rho_{N}$ (Definition \ref{d:rho_M}).
	We claim that  $\pi_{1}$ is a seminorm preserving map from $(E, \rho_{N})$ to $(W^{*}, \lVert \rVert)$.
	%%%%%%%%
	Indeed, for every $x \in E$ we have:
	$$
	\lVert \pi_{1}(x) \rVert := 
	\sup_{\varphi \in N} \lvert (\pi_{1}(x))(\varphi) \rvert :=
	\sup_{\varphi \in N} \lvert \varphi(x) \rvert :=
	\rho_{N}(x).
	$$
	Thus, $\pi_{1}$ is necessarily continuous and open onto its image.
	
	Let $\Delta_{1}\colon W \to W^{**} = V_1^{*}$ be the canonical map. Note that it is not necessarily weak-star continuous. 
	We claim that ${\id_{M} = \pi_{1}^{*} \circ \Delta_{1}}$.
	Suppose that $\varphi \in M$ and $x \in E$:
	\begin{align*}
		((\pi_{1}^{*} \circ \Delta_{1})(\varphi))(x) := & 
		(\pi_{1}^{*}(\Delta_{1}(\varphi)))(x) \\
		= & (\Delta_{1}(\varphi))(\pi_{1}(x)) \\
		= & (\pi_{1}(x))(\varphi) \\
		= & \varphi(x) = (\id_{M}(\varphi))(x).
	\end{align*}
	
	%1302F2 Added a reference
	Also, by Goldstine's theorem \cite[p.~IV.17 Proposition 5]{Bourb}:
	$$
	B_{V_{1}^{*}} = B_{W^{**}} = \overline{(\Delta_{1}(B_{W}))}^{w^{*}} = \overline{\Delta_{1}(N)}^{w^{*}}.
	$$
	
	%1302F2 Added to complete the proof
	Finally, define $V := \im \pi_{1}$ and the onto operator $\pi\colon E \to V$  induced by $\pi_1$.
	It is easy to see that $\pi$ remains continuous and open onto its image. 
	Also, let $i\colon V \to V_{1}$ be the inclusion map.
	Define $\Delta\colon \spn N \to V^{*}$ by $\Delta = i^{*} \circ \Delta_{1}$.
	
	We claim that $\Delta$ is weak-star continuous.
	For the purpose of this proof, write:
	$$
	U_{X}(a_{1}, \dotsc, a_{n}; \eps) := \{f \in X \mid \forall 1 \leq i \leq n: \lvert f(a_{i}) \rvert < \eps\}.
	$$
	Suppose that $y_{1}, \dotsc y_{n} \in V$ and $\eps > 0$.
	By definition, we can find $x_{1}, \dotsc, x_{n} \in E$ such that $\pi_{1}(x_{i}) = y_{i}$ for every $1 \leq i \leq n$.
	We will now show that:
	$$
	\Delta\left(U_{\spn N} \left(x_{1}, \dotsc, x_{n}; \eps\right)\right) \subseteq U_{V^{*}}(y_{1}, \dotsc, y_{n}; \eps).
	$$
	Indeed, for every $\varphi \in \Delta\left(U_{\spn N} \left(x_{1}, \dotsc, x_{n}; \eps\right)\right)$ and $1 \leq i \leq n$ we have:
	$$
	\lvert(\Delta(\varphi))(y_{i}) \rvert =
	\lvert(\Delta(\varphi))(\pi_{1}(x_{i})) \rvert =
	\lvert((\pi_{1}^{*} \circ\Delta)(\varphi))(x_{i}) \rvert =
	\lvert \varphi(x_{i}) \rvert < \eps.
	$$
	Now:
	$$
	\pi^{*} \circ \Delta = 
	\pi^{*} \circ i^{*} \circ \Delta_{1} = 
	(i \circ \pi )^{*} \circ \Delta_{1} = 
	\pi_{1}^{*} \circ \Delta_{1} = \id_{M}.\\
	$$
	%2302	added a bit more explanation
	In virtue of the Hahn-Banach theorem, $B_{V^{*}} = i^{*}(B_{V_{1}^{*}})$.
	Moreover, it is easy to see that 
	$$
	  \Delta(M) = i^{*} (\Delta_{1}(M)) \subseteq 
	  i^{*} (\Delta_{1}(B_{W})) \subseteq 
	  i^{*} (B_{V_{1}^{*}}) = B_{V^{*}},
	$$
	and therefore
	\begin{align*}
		B_{V^{*}} =\ & 
		i^{*}(B_{V_{1}^{*}}) = 
		i^{*}\left(\overline{\Delta_{1}(N)}^{w^{*}}\right) \subseteq
		\overline{(i^{*} \circ \Delta_{1})(N)}^{w^{*}} = 
		\overline{\Delta(N)}^{w^{*}} \\
		=\ & \overline{\Delta(\overline{\acx}^{w^{*}} M)}^{w^{*}} \subseteq
		\overline{\acx}^{w^{*}} \Delta(M) \subseteq B_{V^{*}}.
	\end{align*}
	Note that we used the continuity of $i^{*}$ and $\Delta$.
	In other words: $B_{V^{*}} = \overline{\acx}^{w^{*}} \Delta(M)$.
	Finally, $\acx M$ is dense in $N$ and therefore $\Delta(\acx M)$ is dense in $\Delta(N)$.
	Since $\Delta(N)$ is absorbing in 
	%290322 Changed V to V*
	$V^{*}$, $\Delta(\spn M)$ is dense in 
	%290322 Changed V to V*
	$V^{*}$.
\end{proof}

%150122 removed lemma about equicontinuous subsets in Banach disks (Lemma \ref{lemma:extreme_points_of_compact_map}.3 is better)
%1302F
%2002M Restored original proof
\begin{lemma} \label{lemma:equicontinuous_factor} 
	%090122
	\rm{(Equicontinuous Factor Lemma)}  
	%0202 Let $E$ be a locally convex space and 
	Let $M \subseteq E^{*}$ be an equicontinuous compact subset. There exist: a continuous, dense and open onto its image
	%0202 
	linear 
	 map $\pi \colon E \to V$ to a Banach space $V$ and a linear operator ${\Delta \colon \spn M \to V^{*}}$ 
	 %290322	Added by Saak suggestion
	 with dense image, 
	 %2302	Added specification for continuity only over $M$.
	 weak-star continuous over $M$, such that
	 %1302 replacing the order ... 
	  $B_{V^{*}} = \overline{\acx}^{*}(\Delta(M))$ and 
	$
	\id_{M} = \pi^{*} \circ \Delta. 
	$
	%150122 New, used for haydon

	%090122 I wanted to move the following two lines into the proof at the appropriated place but I did not find any such place ... 
	%0202 suggest to remove this line at all 
	%Note that, by Fact \ref{f:adjoint_is_continuous}, $\pi^{*}$ is strongly continuous as the adjoint of a continuous map. 
	
	%0202q  did you try to use simply the restriction operator E \to C(M) into the Banach space V=C(M) for a compact equicontinuous M ? I wrote this in email   
\end{lemma}
\begin{proof}
	Let $\Delta\colon \spn M \to V^{*}$ and $\pi\colon E \to V$ be the maps described in Lemma \ref{lemma:equicontinuous_factor_normed}.
	Consider the completion $\widehat{V}$ and the inclusion map $i\colon V \to \widehat{V}$.
	Note that $\Delta(M)$ is weak-star compact as a continuous image of a compact set.
	%050122	Moved a bit and added an explanation
	Also, $V$ is normed so we can use the Banach-Steinhaus Theorem to conclude that $\Delta(M)$ is equicontinuous.  
	Using Lemma \ref{f:equicontinuous_extension}, we can find a weak-star compact, equicontinuous subset $\widehat{M} \subseteq V^{*}$ such that $\Delta(M) = i^{*}(\widehat{M})$.
	As a consequence, $i^{*}$ is a closed map on $\widehat{M}$.
	Also, since $V$ is dense in $\widehat{V}$, $i^{*}$ is injective.
	Therefore, $i^{*}$ is a weak-star homeomorphism on $\widehat{M}$.
	We define $\widehat{\pi} := i \circ \pi$ and $\widehat{\Delta} := (i^{*})^{-1} \circ \Delta$.
	
	%310322	Added by Saak suggestion
	Forgetting the earlier notations of $V,\pi$ and $\Delta$, we can define $V:=\widehat{V}, \pi:=\widehat{\pi}$ and $\Delta:=\widehat{\Delta}$ to get the desire result.
\end{proof}
%2302 New
\begin{remark}
	Note that unlike in Lemma \ref{lemma:equicontinuous_factor_normed}, we can't guarantee that $\Delta$ is weak-star continuous over $\spn M$, but only over $M$.
\end{remark}
\begin{lemma} \label{lemma:disk_hahn_banach}
	Let 
	%0202 $E$ be a locally convex space and let 
	$D \subseteq E$ be a disk, 
	$\varphi \in E^{*}$ and ${\eps := \varphi^{-1}((-1, 1))}$.
	Suppose that for some 
	%2206(Important)	Removed "bounded" since it wasn't necessary and was problematic later
	disk neighborhood $\delta \subseteq E$ we have 
	$
	\delta \cap D \subseteq \eps \cap D. 
	$
	Then there exists $\widehat{\varphi}\in \delta^{\circ}$ 
	%2306 where is the definition of \delta^{\circ}  ? 
	%3006 Added before the Lemma
	such that $\varphi_{\mid D} = \widehat{\varphi}_{\mid D}$.
\end{lemma}
\begin{proof}
	%0607	Consistently changed \spn * to \spn(*) to prevent some alignment problems
	It is easy to see that for every $x \in \spn(D)$ we have $\lvert \varphi(x) \rvert \leq q_{\delta}(x)$, 
	%090122 
	where $q_{\delta}(x) := \inf \{r >0 \mid x \in r \delta\}$ is the gauge of $\delta$. 
	By the Hahn-Banach Theorem, we can find a continuous functional $\widehat{\varphi}$ on $E$ which 
	%1805 identifies 
	agrees with $\varphi$ on $\spn(D)$ and for every $x \in E: \lvert \widehat{\varphi}(x) \rvert \leq q_{\delta}(x)$.
	
	Clearly, this also implies that $\widehat{\varphi}$ is continuous and therefore belongs to $E^{*}$.	
	Moreover, for every $x \in \delta$ we have 
	$
	\lvert \widehat{\varphi}(x) \rvert \leq 
	q_{\delta}(x) \leq
	1.
	$
	By definition of the polar, $\widehat{\varphi}\in \delta^{\circ}$.
\end{proof}

%1012	Added for the next lemma
\begin{f} \cite[p.~131, Corollary 5]{Jarchow} \label{fact:absolutely_convex_separation}
	Suppose that  %0202 $E$ is a locally convex space and 
	$A, B \subseteq E$ are non-empty subsets.
	If $A$ is closed and absolutely convex, $B$ is compact and $A \cap B = \emptyset$, then there exists $\varphi \in E^{*}$ such that:
	$$
	  \sup_{x \in A} \lvert \varphi(x) \rvert < \inf_{y \in B} \lvert \varphi(y) \rvert.
	$$
\end{f}

%1012	New, instead of a reference.
%1403 Made shorter

The following is a 
%1503  
 (probably known) consequence of the Banach-Grothendieck theorem 
 %290322 Added exact theorem number reference (Saak suggestion)
 \cite[p.~147 Thm.~2]{Jarchow}. 
\begin{lemma} \label{lemma:discrete_measures_are_dense}
	Let $K$ be a compact 
	%1912 set, and let $i\colon K \to C(K)^{*}$ be the map sending $x \in K$ to the evaluation $e_{x}(f) := f(x)$.
	space, and let $i \colon K \hookrightarrow C(K)^{*}$ be the standard weak-star  embedding.  
	%0202 (sending $x \in K$ to the evaluation $e_{x}(f) := f(x)$). 
	Then $\overline{\acx}^{w^{*}} i(K) = B_{C(K)^{*}}$.
\end{lemma}

\sk
%0812 \section{Bornological class and applications to lcs}
%1912 new (longer) title 
\section{Bornological classes} 
\label{s:Born} 

%090122 in the book \cite{Born}, Bornology is defined as: 
A \textit{bornology} $\mathcal{B}$ on a lcs $E$ is a family of subsets in $E$ which covers $E$, and is hereditary 
%290322 Saak said that "under inclusion" is not clear.
%	However, it is the same terminology used in \cite{Born}.
%	I've added some clarification
under inclusion (i.e. if $A \in \mathfrak{B}$ and $B \subseteq A$ then $B \in \mathfrak{B}$) and finite unions. So, every bornology contains all finite subsets. 

%1012 added definitions
%1912  These definitions are new ?  If YES, then better to present it into the Definition frame. If NOT, plaese give a reference. 
%Do you mean that they are from  \cite{Born} ? 
%2012 Added citation
A \emph{vector bornology} \cite[Definition 1:1'2]{Born} is a bornology $\mathcal{B}$ on %090122 a vector space 
$E$ such that whenever $A, B \in \mathcal{B}$ and $\alpha \in \R$ we have:
\ben
	\item $A + B \in \mathcal{B}$;
	\item $\alpha A \in \mathcal{B}$;
	\item $\bo A \in \mathcal{B}$.
\een
%2012	Changed "locally convex" to "convex" since it appears this way in \cite{Born}
If $\mathcal{B}$ is closed 
%290322	changed "for" to "under taking" (Saak suggestion)
under taking convex hulls, it is said to be a \emph{convex bornology}. 
%050122
It is said to be \emph{saturated} if the closure of sets in $\mathcal{B}$ remains in $\mathcal{B}$.
%240202
Moreover, it is 
%2802  let's discuss 
\emph{separated} if it has no non-trivial bounded subspaces.

%090122q    (which details?) 
%  More details can be found in \cite{Born}.

%0812  It is worth to make the definition more precise. In particular, 
%what is $\mathfrak{B}$ ... 
\begin{defin} \label{def:bornological_class}
	%1012 I tried to be more precise:
	A \emph{bornological class} $\mathfrak{B}$ is  
	%1912 a mapping  
	%perhaps instead of "mapping" better to say "assignment " 
	an assignment 
	$$Comp \to \{Bornologies\},  \ \ \ K \mapsto \mathfrak{B}_{K}$$ 
	from the class of all compact spaces $Comp$ to the class of 
	%1912  what is "linear bornologies" 
	%2012	Changed to vector bornologies
	%150122 added "convex"
	vector bornologies such that 
	%0202 $\mathfrak{B}(K)$ (often denoted $\mathfrak{B}_{K}$) 
	$\mathfrak{B}_{K}$  
	is a 
	%240222
	separated
	convex vector bornology on
	%090122 $C(K)$ 
	the Banach space $C(K)$ 
	satisfying the following properties:	
	\ben
	\item \textbf{boundedness}: %0812 For every compact $K$, 
	$\mathfrak{B}_{K}$ consists of %090122 uniformly 
	bounded subsets in $C(K)$.
	\item \textbf{consistency}: Suppose that $\varphi \colon K_{1} \to K_{2}$ is a continuous map. 
	\ben
	\item 
	If $A \in \mathfrak{B}_{K_{2}}$, 
	then $A \circ \varphi \in \mathfrak{B}_{K_{1}}$.
	%0812 more items? (you have only (a)) 
	%1012 fixed
	\item If $\varphi$ is surjective, then the converse is also true, namely that $A \circ \varphi \in \mathfrak{B}_{K_{1}}$ implies $A \in \mathfrak{B}_{K_{2}}$.
	\een
	%2012	New
	\item \textbf{Bipolarity}: If $A \in \mathfrak{B}_{K}$, then $A^{\circ\circ} = \overline{\acx}^{w} A \in \mathfrak{B}_{K}$ where the polar is taken with respect to the dual $C(K)^{*}$ (note that we 
	%090122 used Fact \ref{fact:bipolar_theorem}).
	use the Bipolar Theorem). 
	\een

	%2012 Removed	
%	We also have two more 
%	%1912 why you write "optional" ?
%	optional properties:
%	\ben
%	\setcounter{enumi}{2}
%	%0812 do you mean the items (3) and (4) ?
%	\item %0812 We say that 
%	$\mathfrak{B}$ is \emph{saturated} if for every compact $K$ and $A \in \mathfrak{B}_{K}$, its 
%	%0812 "weak closure" would be sound more naturally doing the same job (via the convex hull) it is worth to examine this 
%	%1012 If we go in this direction, I think that pointwise closure is more natural
%	norm closure $\overline{A} \in \mathfrak{B}_{K}$.
%	\item %07 We say that 
%	$\mathfrak{B}$ is \emph{locally convex} if for every compact $K$, $\mathfrak{B}_{K}$ is 
%	%0812 not defined 
%	locally convex. 
%	\een
\end{defin}

\sk 
%1910	New
We  write $\BTame, \BNP$ and $\BDLP$ for the classes of tame, fragmented and DLP function families, respectively.
Recall that by Lemma \ref{f:sub-fr}, $\BTame$ coincides with the class of 
eventually fragmented %0812 maps.
families. 

\begin{prop} \label{prop:examples_of_bornological_categories}
	$\BTame$, $\BNP$ and $\BDLP$ are 
	%2012 Removed: locally convex, saturated 
	bornological classes.
\end{prop}
\begin{proof}
	First, it is obvious that all three of these classes consists of 
	%0812  "linear bornology" is not defined
	%1012  changed to vector and added a definition
	 vector bornologies. 
	We organize the rest of the proof in the following table:
	\begin{center}
		\begin{tabular}{||c || c | c | c||} 
			\hline
			Property & $\BTame$ & $\BNP$ & $\BDLP$ \\ [0.5ex] 
			\hline\hline
			boundedness & \multicolumn{3}{|c|}{by definition} \\ 
			\hline
			consistency & Lemma \ref{lem:surjective_image_of_independent_sequence} & Lemma \ref{l:surjection_of_fragmented_family} & Lemma \ref{lemma:double_subsequence_of_bounded_functions_is_DLPable}.4 \\ 
			\hline
			%2012	Changed to apply for new definition
			bipolarity & 
				\multicolumn{2}{|c|}{
					Corollary \ref{cor:fragmented_bipolarity}
				} &
				%2911	Changed since the saturation and convexivity were combined
				{Lemma \ref{lemma:DLP_is_locally_convex}} \\
			\hline
			%2012	Removed
%			local convexity & \multicolumn{2}{|c|}{Lemma \ref{l:convFr}} &
%			%2911	Changed since the saturation and convexivity were combined
%			 {Lemma \ref{lemma:DLP_is_locally_convex}} \\
			\hline
		\end{tabular}
	\end{center}
\end{proof}

%240222	The following may be trivial, but I couldn't find a reference
%1403	Made shorter
The following is easy to see.
\begin{lem}\label{lemma:euclidean_bornology}
	The only separated, convex, vector bornology on $\R$ is the 
	%2802 what is the "euclidean bornology" ? 
	euclidean bornology $B_{e}$.
\end{lem}
%\begin{proof}
%	Suppose that $B$ is a separated, convex, vector bornology on $\R$.
%	By definition, all the finite subsets belong to $B$.
%	As a consequence, so does all the closed intervals $[a, b]$ for $a, b \in \R$.
%	In other words, $B_{e} \subseteq B$.
%	
%	By contradiction, assume that $A \in B \setminus B_{e}$.
%	Therefore, $A$ is not bounded.
%	It is easy to see that $\acx A = \R$, which is impossible since $B$ is separated.
%\end{proof}

%1910	New
\begin{defin} \label{def:b_small_space}
	Let $\mathfrak{B}$ be a bornological class. %0202 and let $E$ be a locally convex space.
	A bounded subset $B \subseteq E$ is said to be \emph{$\mathfrak{B}$-small} if for every 
	%0202 weak-star compact, equicontinuous subset 
	$M \in \eqc(E^{*})$, 
	%310322 Changed by Saak suggestion
	$r_{M}(B) \in \mathfrak{B}_{M}$ where $r_{M} \colon E \to C(M)$ is the 
	%0202 \textit{restriction map}. [PLEAE CHECK]
	\textit{restriction operator}.
	
%	%0202 viewed as a %090122 uniformly bounded family of functions on $M$ 
%	belongs to 
%	%0812  few lines before we had $\mathfrak{B}_{K}$. For the consistency I suggest to write here  $\mathfrak{B}_{K}$ and not $\mathfrak{B}_{M}$ (and maybe in the sequel, as far as is possible) 
%	%1012	I used $K$ for abstract compact spaces and $M$ for weak-star compact subsets of the dual. 
%	%1012		If you think that it is confusing, I will change it
%	%1912  I understand now your preference. No, it is not confusing. 
%	 $\mathfrak{B}_{M}$.
%	We often write $r \colon E \to C(M)$ for the 
%	%0202 \textit{restriction map}. [PLEAE CHECK]
%	\textit{restriction operator}.
%	%1910	new sentence
%	In this case, $B$ is $\mathfrak{B}$-small if and only if $r(B) \in \mathfrak{B}_{M}$ for every $M \in \eqc(E^{*})$.
	
	A locally convex space is said to be \emph{$\mathfrak{B}$-small} if every bounded subset is $\mathfrak{B}$-small.
\end{defin}

\sk 
\begin{lemma} \label{lemma:class_induces_bornology}
	Let $E$ be a locally convex space and $\mathfrak{B}$ a bornological 
	%310322	Changed "category" to "class" by Saak suggestion
	class.
	The family of $\mathfrak{B}$-small subsets in $E$
	%0812  you mean: "in $E$" ?
	%1012	yes, added.
	%1912	are a vector bornology, written 
	%2012	Changed to accomodate the new definition
	is a 
	%090122 "saturated" is not defined 
	saturated, convex vector bornology, denoted by 
	%0812 warning: it is not easy to distinguish and remember that we have both supscript and upscript 
	%1012 What notation is better in your opinion 
	%1912  Let's think !
	%2012  changed to a command so it would be easy to change.
	$\bsmall{E}$.
	%2012 Removed because of the change in the definition.
	%If $\mathfrak{B}$ is also locally convex (resp. saturated), then so is $\mathfrak{B}^{E}$.
\end{lemma}
%1109	Added proof
\begin{proof}
	%1109q	I've skipped over many of the details. Does this seem ok to you?
	The only non-trivial assertion is that $\bsmall{E}$ is saturated.
	%Indeed, s
	Suppose that ${A \in \bsmall{E}}$. We show that $\overline{A} \in \bsmall{E}$.
	Let 
	%090122 $M \subseteq E^{*}$ be equicontinuous and weak-star compact.
	$M \in \eqc(E^*)$ 
	and ${r \colon E \to C(M)}$ be the restriction 
	%0202 map (taking each element of $E$ to the evaluation function on $M$).
	operator. 
	By definition, $r(A) \in \mathfrak{B}_{M}$.
	Also, since $\mathfrak{B}$ 
	%2012 changed is saturated
	satisfies bipolarity, $\overline{r(A)} \subseteq \overline{r(A)}^{w} \subseteq (r(A))^{\circ\circ} \in \mathfrak{B}_{M}$.
	%0202 It is easy to see that 
	By continuity, $r(\overline{A}) \subseteq \overline{r(A)}$.
	%0202 and therefore $\overline{A} \in \bsmall{E}$.
	Hence, $r(\overline{A}) \in \mathfrak{B}_{M}$. Therefore, $\overline{A} \in \bsmall{E}$. 
\end{proof}
  
%150122	New to make things simpler
\begin{lem} \label{lemma:b_small_and_adjoint_maps}
	Suppose that $T\colon E \to F$ is a continuous linear map between locally convex spaces, $B \subseteq E$ is bounded and $M \in \eqc(F^{*})$.
	%310322 Changed r_{E} and r_{F} to r_{T^{*}(M)} and r_{M} respectivly by Saak suggestion
	Let $r_{T^{*}(M)} \colon E \to C(T^{*}(M))$ and $r_{M}\colon F \to C(M)$ be the restriction maps.
	Then $r_{T^{*}(M)}(B) \in \mathfrak{B}_{T^{*}(M)}$ if and only if $r_{M}(T(B)) \in \mathfrak{B}_{M}$.
\end{lem}
\begin{proof}
	We will show that:
	$$
	r_{M}(T(B)) = r_{T^{*}(M)}(B) \circ T^{*}.
	$$
	The claim is then obvious from the consistency 
	%0202 property. 
	property (which we apply for the map $M \to T^*(M)$).
	Indeed, for every $x \in B$ and $\varphi \in M$, we have:
	$$
	(r_{M}(T(x)))(\varphi) := 
	\varphi(T(x)) = 
	(T^{*}(\varphi))(x) = 
	(r_{T^{*}(M)}(x))(T^{*}(\varphi)) = 
	(r_{T^{*}(M)}(x) \circ T^{*})(\varphi).
	$$
\end{proof}

%1109	New
\begin{lem} \label{lemma:image_of_a_B_small_set}
	Suppose that $T\colon E \to F$ is a continuous linear map and %090122 suppose that
	 $A \in \bsmall{E}$.
	Then $T(A) \in \bsmall{F}$.
\end{lem}
\begin{proof}
	%150122 made simpler with Lemma \ref{lemma:b_small_and_adjoint_maps}
	Let $M \in \eqc(F^{*})$.
	By definition, $r_{E}(A) \in \mathfrak{B}_{E}$.
	Applying Lemma \ref{lemma:b_small_and_adjoint_maps}, we conclude that $r_{F}(T(A)) \in \mathfrak{B}_{M}$.
	This is true for every $M \in \eqc(F^{*})$, proving the desired result.
\end{proof}

\sk 
%0202 
%310322	Added "continuous" by Saak suggestion
Recall that a continuous linear map $f\colon E_1 \to E_2$ is said to be \textit{bound covering} if for every bounded $B_2 \subset E_2$ there exists a bounded subset $B_1 \subset E_1$ such that $f(B_1)=B_2$. 

%090122  "bound covering map" is not defined 
%1109	New
\begin{cor} \label{cor:bound_covering_b_small}
	The class of $\mathfrak{B}$-small locally convex spaces is closed under bound covering maps.
\end{cor}
In Proposition \ref{p:variety}, we show that 
%1802
``bound covering" is really essential.

%1109	New
\begin{prop} \label{prop:b_small_subspace}
	The class of $\mathfrak{B}$-small locally convex spaces is closed under 
	%310322	Added "taking" by Saak suggestion
	taking linear subspaces. 
	
	%1109	Needed to make thm \ref{thm:tame_iff_R} simpler.
	In fact, if $E$ is a locally convex space, $F \subseteq E$ is a subspace and $B \subseteq F$ is a bounded subset, then $B \in \bsmall{E}$ if and only if $B \in \bsmall{F}$.
\end{prop}
\begin{proof} 
	%1109 added to support the new change.
	First, if $B \in \bsmall{F}$ then clearly $B \in \bsmall{E}$ in virtue of Lemma \ref{lemma:image_of_a_B_small_set}.
	
	%1503 
	Conversely, suppose that $B \in \bsmall{E}$.  We will show that $B \in \bsmall{F}$.
%090122  Let $M \subseteq F^{*}$ be an equicontinuous, weak-star compact subset. 
Let ${M \in \eqc(F^*)}$. 
	By Lemma \ref{f:equicontinuous_extension}, we can find 
	%090122 a weak-star compact, equicontinuous $N \subseteq E^{*}$ 
		$N \in \eqc(E^*)$.
	such that ${i^{*}(N) = M}$ where $i\colon F \to E$ is the inclusion map.  
	%(meaning that we extended $M$ to $E$ and maintained its equicontinuity).
	%1910:	The following argument is no longer needed since Proposition \ref{f:equicontinuous_extension} got stronger.
	%Using Fact \ref{l:Al-Bo} 
	%and the weak-star continuity of the adjoint map $i^* \colon E^* \to F^*$, we can assume that $N$ is also weak-star %1907 precompact.
	%compact. 
	
	Let $r_{E}\colon E \to C(N)$ and $r_{F}\colon F \to C(M)$ be the restriction maps.
	%150122 Made simpler with Lemma \ref{lemma:b_small_and_adjoint_maps}	
	%310322	Changed \mathfrak{B}^{E} to \bsmall{E} by Saak suggestion
	${B = i(B) \in \bsmall{E}}$ by definition, and therefore $r_{E}(i(B)) \in \mathfrak{B}_{N}$.
	Applying Lemma \ref{lemma:b_small_and_adjoint_maps}, we conclude that $r_{F}(B) \in \mathfrak{B}_{i^{*}(N)} = \mathfrak{B}_{M}$.
	This is true for every 
	%090122 equicontinuous, weak-star compact $M \subseteq F$ 
	$M \in \eqc(F^*)$ and therefore $B \in \bsmall{F}$.
\end{proof} 

%0202 Suppose that $\{E_{i}\}_{i=1}^{n}$ are topological vector spaces and let $E := \prod\limits_{i=1}^{n} E_{i}$. 
Suppose that $E := \prod\limits_{i=1}^{n} E_{i}$ is the product of topological vector spaces $\{E_{i}\}_{i=1}^{n}$. 
Write $\pi_{i} \colon E \to E_{i}$ for the projection map.
Also, let $\Delta_{i} \colon E_{i} \to E$ be the \textit{dissection map} sending an element $x \in E_{i}$ to an element of $E$ whose $i$'th entry is $x$ and the rest are $0$.
Finally, consider the adjoint map $\Delta_{i}^{*} \colon E^{*} \to E_{i}^{*}$ defined by
$
\Delta_{i}^{*}(\varphi) :=
\varphi \circ \Delta_{i}.
$
%1910	New
\begin{lemma} \label{lemma:product_is_sum_of_restrictions}
	Let $E_{1}, \dotsc, E_{n}$ be locally convex spaces, %2012 removed: and let 
	$A_{i} \subseteq E_{i}$ for $i \in \{1, \dotsc, n\}$ and $M \subseteq E^{*}$ be subsets.
	Consider $E = \prod\limits_{i=1}^{n} E_{i}$ and $A = \prod\limits_{i=1}^{n} A_{i}$.
	For every $1 \leq i \leq n$, let $r_{i} \colon E_{i} \to C(\Delta_{i}^{*}(M))$ and $r \colon E \to C(M)$ be the restriction maps.
	Then
	$$
	r(A) \subseteq \sum_{i=1}^{n} r_{i}(A_{i}) \circ \Delta_{i}^{*}.
	$$
\end{lemma}
\begin{proof}
	Indeed, suppose that $a = (a_{1}, \dotsc, a_{n}) \in A$, meaning that $a_{i} \in A_{i}$ for every $1 \leq i \leq n$.
	We claim that 
	$$
	r(a) = \sum_{i=1}^{n} r_{i}(a_{i}) \circ \Delta_{i}^{*} \in \sum_{i=1}^{n} r_{i}(A_{i}) \circ \Delta_{i}^{*}.
	$$
	Let $\varphi \in M$.
	For every $x \in E$ we can write
	$$
	x = \sum_{i=1}^{n} \Delta_{i}(\pi_{i}(x)),
	$$
	and therefore
	$$
	\varphi(x) = 
	\varphi \left( \sum_{i=1}^{n} \Delta_{i}(\pi_{i}(x)) \right) =
	\sum_{i=1}^{n} (\varphi \circ \Delta_{i}) (\pi_{i}(x)) =
	\sum_{i=1}^{n} (\Delta_{i}^{*}(\varphi)) (\pi_{i}(x)).
	$$
	In particular, 
	$$
	(r(a))(\varphi) := 
	\varphi(a) = 
	\sum_{i=1}^{n} (\Delta_{i}^{*}(\varphi)) (a_{i}) = 
	\sum_{i=1}^{n} (r_{i}(a_{i}) \circ \Delta_{i}^{*})(\varphi).
	$$
%090122 	as we have claimed.
\end{proof}

\begin{remark} \label{r:dualities}  \ 
	\begin{enumerate}
		\item \cite[IV 4.3, Theorem]{Schaefer} Let $E=\prod_{i \in I} E_i$ be a product of lcs $E_i$. Then its dual $E^*$ is \textbf{algebraically} the locally convex direct sum $\bigoplus_{i \in I} E_i^*$ with the corresponding duality 
		$$
		E \times E^* \to \R, \ (v,u) = \left((v_i)_{i \in I},\sum_{i \in I} u_i\right) \mapsto \sum_{i \in I} \langle v_i,u_i \rangle.$$
		\cite[Section 8.8, Proposition 1]{Jarchow} A basis of the equicontinuous compactology $\eqc(E^*)$ on $E^*$ is obtained by taking all sets of the form $\sum_{j \in J} H_j$, where $J$ is finite and $H_j \in \eqc(E_j)$.
		
		\item \cite[IV 4.3, Corollary 1]{Schaefer} Similarly, if $E=\bigoplus_{i \in I} E_i$ is a lc sum then its dual $E^*$ is \textbf{algebraically} the locally convex product $\prod_{i \in I} E_i^*$ with the corresponding duality 
		$$
		E \times E^* \to \R, \ (v,u)=\left(\sum_{i \in I} v_i,(u_i)_{i \in I}\right) \mapsto \sum_{i \in I} \langle v_i,u_i \rangle.$$
		A basis of the equicontinuous compactology $\eqc(E^*)$ on $E^*$ is obtained by taking all sets of the form $\prod_{i \in I} H_i$, where $H_i \in \eqc(E_i)$.

		By \cite[II, 6.3]{Schaefer}, for every bounded subset $B$ of a locally convex direct sum
		${\bigoplus}_{i\in I} E_i,$ there exists a finite set
		$J\subset I$ such that $pr_i (B)$ is zero for 
		every $i\not\in J.$
	\end{enumerate} 
\end{remark}
 
\sk 
\begin{lemma} \label{lemma:product_of_finite_B_small}
	Let $E_{1}, \dotsc, E_{n}$ be locally convex spaces and let $A_{i} \subseteq E_{i}$ for $i \in \{1, \dotsc, n\}$ be $\mathfrak{B}$-small subsets.
%090122	Write $E = \prod\limits_{i=1}^{n} E_{i}$ and $A = \prod\limits_{i=1}^{n} A_{i}$. Then $A$ is $\mathfrak{B}$-small in $E$. 
Then $A = \prod\limits_{i=1}^{n} A_{i}$ is $\mathfrak{B}$-small in $E = \prod\limits_{i=1}^{n} E_{i}$.
\end{lemma}
\begin{proof}
	%1910	The proof became much simpler
	Let $M \in \eqc(E^{*})$. 
	%0202 Consider the maps ${\Delta_{i}^{*}: E^{*} \to E_{i}^{*}}$ for $1 \leq i \leq n$. Note that ${M_{i} := \Delta_{i}^{*}(M) \in \eqc(E_i^{*})}$.  
	Then ${M_{i} := \Delta_{i}^{*}(M) \in \eqc(E_i^{*})}$, where 
	${\Delta_{i}^{*}: E^{*} \to E_{i}^{*}}$ are defined as in Lemma \ref{lemma:product_is_sum_of_restrictions}.   
	As before, let ${r_{i}: E_{i} \to C(M_{i})}$ be the restriction map.
	By definition, $r_{i}(A_{i}) \in \mathfrak{B}_{M_{i}}$.
	Since $\mathfrak{B}$ is a 
	%090122 "bornological category" is not defined 
	bornological 
	%310322	Changed "category" to class by Saak suggestion
	class, this implies that $r_{i}(A_{i}) \circ \Delta_{i}^{*} \in \mathfrak{B}_{M}$ for every $1 \leq i \leq n$.
	Also, $\mathfrak{B}_{M}$ is a linear bornology and therefore $A ' := \sum\limits_{i=1}^{n} r_{i}(A_{i}) \circ \Delta_{i}^{*} \in \mathfrak{B}_{M}$.
	By Lemma \ref{lemma:product_is_sum_of_restrictions}, $r(A) \subseteq A'$, proving the desired result.
\end{proof}

%2402	Moved here
The following remark
%2802 you mean: "following remark" ?    
 is a consequence of the previous lemma and Lemma \ref{lemma:euclidean_bornology}.
\begin{remark} \label{remark:extensibility_induction}
	If $\mathfrak{B}$ is a bornological class and $F$ is finite, then $\mathfrak{B}_{F}$ is simply the euclidean bornology.
\end{remark}

%1109	New
\begin{cor} \label{cor:B_small_prod-sums} 
	Arbitrary products and direct sums of $\mathfrak{B}$-small spaces are $\mathfrak{B}$-small.
\end{cor}
\begin{proof}
	First consider the case of products.
	Let $\{E_{i}\}_{i \in I}$ be a family of $\mathfrak{B}$-small spaces.  
	%090122 and write $E := \prod\limits_{i \in I} E_{i}$.
	Let ${B \subseteq E := \prod\limits_{i \in I} E_{i}}$ be bounded, $M \in \eqc(E^{*})$ and $r \colon B \to C(M)$ be the restriction map.
	We show that $r(B) \in \mathfrak{B}_{M}$. 
	Indeed, using Remark \ref{r:dualities}.1, we know that there is a finite $J \subseteq I$ and $H_{j} \in \eqc(E_{j}^{*})$ such that $M \subseteq \sum\limits_{j \in J} H_{j}$.
	Thus, the system $(B, M)$ where $B$ is considered as a family of functions over $M$ can be 
	%090122q what do you mean under "isomorphically embedded"
	isomorphically embedded in $(\prod\limits_{j \in J} E_{j}, \prod\limits_{j \in J} E_{j}^{*})$.
	However, this family is $\mathfrak{B}$-small as a consequence of Lemma \ref{lemma:product_of_finite_B_small}.
	%2007 Added reference to DLP product lemma
	%2507 (Lemma \ref{lemma:finite_dlp_product}).
	
	For the case of direct sums, we use a very similar technique, this time leveraging Remark \ref{r:dualities}.2 by factoring the bounded set to finite components rather than the equicontinuous family.
\end{proof}

%240222	Moved here since polar compatibility is not necessary
\begin{cor} \label{cor:weak_topology_is_b_small}
	If $\mathfrak{B}$ is a bornological class, then every locally convex $E$ with the weak topology is $\mathfrak{B}$-small.
\end{cor}
\begin{proof}
	First, recall that $E_{w}$ can be embedded in $\R^{E^{*}}$.
	In virtue of Theorem \ref{thm:properties_of_b_small_classes}, $E_{w}$ is $\mathfrak{B}$-small as a subspace of the product of $\mathfrak{B}$-small spaces.
\end{proof}

%1109	New
\begin{lem} \label{lemma:large_subspace_of_b_small}
	Suppose that $\mathfrak{B}$ is 
	%2012 changed because of the new definition: saturated 
	%310322	Changed category to class
	a bornological class and 
	%090122 let $E$ be a locally convex space. Also, let 
	$F$ is a dense large subspace of a lcs $E$.
	Then $F$ is $\mathfrak{B}$-small if and only if $E$ is $\mathfrak{B}$-small.
\end{lem}
\begin{proof}
	If $E$ is $\mathfrak{B}$-small then so is $F$ in virtue of Proposition \ref{prop:b_small_subspace}.
	
	Conversely, assume that $F$ is $\mathfrak{B}$-small.
	Suppose that $B \subseteq E$ is bounded.
	By definition, there is a bounded $C \subseteq F$ such that
	$B \subseteq \overline{C}$.
	Since $F$ is $\mathfrak{B}$-small, $C \in \bsmall{F}$.
	Applying Proposition \ref{prop:b_small_subspace}, we conclude that $C \in \bsmall{E}$.
	Finally, using Lemma \ref{lemma:class_induces_bornology}, we conclude that 
	$
	B \subseteq \overline{C} \in \bsmall{E}.
	$
\end{proof}

\begin{thm} \label{thm:properties_of_b_small_classes}
	The class of $\mathfrak{B}$-small locally convex spaces is closed under:
	\ben
	\item subspaces
	\item bound covering maps
	\item %(possibly infinite) 
	products
	\item %(possibly infinite) 
	direct sums
	\item inverse limits.
	\een
	Moreover, if 
	%1511 minor change
	%2012 removed saturated:	$\mathfrak{B}$ is saturated,
	%310322 change "subset" to subspace
	$F$ is a large, dense subspace of the locally convex space $E$, and $F$ is $\mathfrak{B}$-small, then so is $E$.
	In particular, if $V$ is a normed $\mathfrak{B}$-small space, then so is its completion. 
\end{thm}
%1109	New
\begin{proof}
	%Direct applications 
	Apply  Corollary \ref{cor:bound_covering_b_small}, Proposition \ref{prop:b_small_subspace}, Corollary \ref{cor:B_small_prod-sums} and Lemma \ref{lemma:large_subspace_of_b_small}.
\end{proof}

\sk 
\subsection{Relation to the Mackey Topology}
%090122q  
%One may ask why we restrict ourselves to equicontinuous subsets in Definition \ref{def:b_small_space}.
%%1511 reformulated
%It turns out that this has a satisfying answer, as can be seen in Proposition \ref{prop:strong_b_small_and_mackey}.
%For the purpose of the claim, we make the following definition.
\begin{defin} \label{def:strong_b_small}
	Let $\mathfrak{B}$ be a bornological class. %0202 and let $E$ be a locally convex space.
	A bounded subset $B \subseteq E$ is said to be \emph{Mackey $\mathfrak{B}$-small} if for every 
	%310322 Added absolutely convex to fix the problem Saak found
	absolutely convex weak-star compact, (not necessarily equicontinuous) subset $M \subset E^{*}$, $B$ viewed as a 
%090122 	uniformly bounded function family 
bounded family of functions 
	on $M$ belongs to $\mathfrak{B}_{M}$.
	
	A locally convex space is said to be \emph{Mackey $\mathfrak{B}$-small} if every bounded subset is $\mathfrak{B}$-small.
\end{defin}

\sk 
%2911	Added by your suggestion
%0812 Before we continue, let us recall the definition of a Mackey space. 
Recall that the
\emph{Mackey topology} on a %locally convex space 
lcs $(E, \tau)$ is the strongest topology compatible with its dual $E^{*}$. 
We will often denote it as $\tau_{\mu}$.
It is exactly the polar topology induced by all weak-star compact, absolutely convex subsets of $E^{*}$ \cite[p.~131]{Schaefer}.
$E$ is said to be a \emph{Mackey space} if $\tau = \tau_{\mu}$.

\begin{prop} \label{prop:strong_b_small_and_mackey}
	A bounded subset $B$ in %a locally convex space 
	$E$ is Mackey $\mathfrak{B}$-small if and only if it is $\mathfrak{B}$-small with respect to the Mackey topology.
	The same can be said for the entire space $E$.
\end{prop}
\begin{proof}
	Recall (\cite[p.~158, Thm.~5]{Jarchow}) that the Mackey topology is compatible with the dual $E^{*}$.
	As a consequence, it has the same 
	%310322 added absolutely convex
	absolutely convex, weak-star compact subsets 
	%0202 of the regular topology. 
	for the usual topology.
	Therefore, if $B$ is Mackey $\mathfrak{B}$-small, then $B$ is $\mathfrak{B}$-small with respect to the Mackey topology.
	
	Conversely, suppose that $B$ is $\mathfrak{B}$-small with respect to the Mackey topology, and let $M \subseteq E^{*}$ be an
	%310322	Added "absolutely convex"
	absolutely convex weak-star compact subset.
	%310322 Removed: 	Write $N = \overline{\acx}^{w^{*}} M$.
	By definition, the polar $M^{\circ}$ is an open neighborhood of zero in the Mackey topology.
	Thus, $M^{\circ\circ}$ is equicontinuous.
	As a consequence, $B$ is $\mathfrak{B}$-small over $M^{\circ\circ}$.
	However, $M \subseteq M^{\circ\circ}$, and therefore $B$ is $\mathfrak{B}$-small over $M$, as required.
\end{proof}

%1910	Moved here
A lcs $E$ 
is said to be \emph{barreled} if every barrel (Definition \ref{d:gauge}) is a neighborhood of zero.
This class %is rather large and 
includes all 
%2208 normed, 
reflexive and complete metrizable (i.e., Frechet) spaces.
%0507 The most useful property of these spaces is shown by the Banach-Steinhaus Theorem \cite[Thm. III.25.1]{Bourb}.
%the Banach-Steinhaus Theorem \cite{Perez} states that every weak* bounded subset of $E^{*}$ is equicontinuous.
%In particular, 
By \cite[p.~132, Lemma 3.4]{Schaefer}, every barreled or metrizable space is a 
 Mackey space.

\begin{cor} \label{cor:mackey_spaces_and_b_smallness}
	If $E$ is a Mackey space (e.g., barrelled or metrizable), then it is $\mathfrak{B}$-small if and only if it is Mackey $\mathfrak{B}$-small.
\end{cor}

%1611   More "public relations" 
For basic information about Mackey topologies and related topics, we refer to \cite{Jarchow,Schaefer}. 
For some generalizations %090122 for topological groups 
see \cite{AD}. 

\sk 
%1910	New!
\subsection{The Co-Bornology and Strongest Topologies}
\label{subsection:co_bornology}

%240222 No more need for the extendability property
%For every compact $K$, let $K[\star]$ be the original space together with a new 
%%0812 maybe you need to hint why you need to adjoint an isolated point
% discrete point $\star$.
% %1012 Added by your suggestion
% The motivation for this definition will be apparent after Remark \ref{remark:extensibility_induction}.

\begin{defin} \label{defin:polarly_compatible}
	A bornological class $\mathfrak{B}$ is said to be \emph{polarly compatible} if whenever $A \in \mathfrak{B}_{K}$ for compact $K$, then $r_{B_{C(K)^{*}}}(A) \in 
	\mathfrak{B}_{B_{C(K)^{*}}}$ where 
	%240222 added details
	$r_{B_{C(K)^{*}}}\colon C(K) \to C(B_{C(K)^{*}})$ is the canonical map
	%2502
	defined by:
	$$
	  (r_{B_{C(K)^{*}}}(f))(\varphi) := \varphi(f).
	$$
%240222 No more need for the extendability property
%	the following conditions are satisfied: %apply: 
%	\ben
%		\item \emph{Extensibility}: If $K$ is a compact space and $A \in \mathfrak{B}_{K}$, then:
%		$$
%		  A[\star] := 
%		  \{f \in C(K[\star]) \mid f_{\mid K} \in A, \lvert f(\star) \rvert \leq 1 \} \in \mathfrak{B}_{K[\star]}. 
%		$$
%		\item \emph{Duality}: Whenever $A \in \mathfrak{B}_{K}$ for compact $K$, then $A \in \mathfrak{B}_{B_{C(K)^{*}}}$.
%	\een
\end{defin}

%240222 Removed since we got it more easily
%
%\begin{remark} \label{remark:extensibility_induction}
%	By inductively applying the extensibility property on the empty set, we conclude that $\mathfrak{B}_{\{x_{1}, \dotsc, x_{n}\}}$ is simply the 
%	%090122 what is "Euclidean bornology"? 
%	Euclidean bornology on $\R^{n}$.
%\end{remark}

For every lcs (e.g., Banach space) $(E,\tau)$ and a
%2402 no need: polarly compatible 
bornological class $\mathfrak{B}$, one may define the 
%090122q  I think we need to add "coarser" 
\emph{strongest locally convex $\mathfrak{B}$-small topology}  
$\tau_{\mathfrak{B}}$ on $E$. 
We mean the supremum of all $\mathfrak{B}$-small locally convex topologies on $E$ which are coarser than $\tau$. 
Since the class of $\mathfrak{B}$-small spaces is closed under products and subspaces (by Theorem \ref{thm:properties_of_b_small_classes}), 
we obtain that $\tau_{\mathfrak{B}}$ 
%090122 
is well-defined and 
it is a $\mathfrak{B}$-small locally convex topology on $E$ (such that $\tau_{\mathfrak{B}} \subseteq \tau$). 
The weak topology $\tau_w$ on $E$ is always $\mathfrak{B}$-small (by Corollary \ref{cor:weak_topology_is_b_small}), and therefore 
$$
  \tau_{w} \subseteq \tau_{\mathfrak{B}} \subseteq \tau.
$$
As a consequence, $\tau_{\mathfrak{B}}$ is always a Hausdorff locally convex topology on $E$. 
Below, 
in Theorem \ref{thm:strongest_b_topology}, we give a description of this topology
%2402
for polarly compatible classes
 as a naturally defined polar topology.

\begin{lem} \label{p:from X to B*} \cite[Prop.~4.19]{GM-rose} 
	Let $K$ be a compact space and $F \subset C(K)$ is bounded.
	Then  
	$F$ is a tame family for $K$ if and only if $F$ is a tame family for the weak-star compact unit ball $B_{C(K)^*}$. 
	Equivalently, $F$ is a tame subset (in terms of Definition \ref{d:TameSet}) of the Banach space $C(K)$. 
\end{lem}

\begin{cor} \label{cor:tame_is_polarly_compatible}
	The class $\BTame$ is polarly compatible.
\end{cor}
%2402 no longer necessary
%\begin{proof}
%	Duality (Definition \ref{defin:polarly_compatible}) is a consequence of the previous lemma.
%	Extensibility follows easily from the definition of eventually fragmented families.
%\end{proof}

%2911q Whay do you think  %0812 let's discuss 
\begin{remark}
%090122 	It is possible to show that 
$\BNP$ is also polarly compatible.
	In fact, an analogous statement to Lemma \ref{p:from X to B*} 
	%can be made 
	holds about fragmented maps 
	%0812 adding the reference 
	\cite{Me-b}.
	%0812  We do not provide proof here but separately in a book by the second author.
\end{remark}

%1910	New. This proved to be harder than expected
%090122q   Let's try to find a reference 
\begin{lem} \label{lemma:weak_star_convex_hull_of_equicontinuous}
%090122	Let $T \subseteq E^{*}$ be weak-star compact and equicontinuous. Then so is ${M := \overline{\acx}^{w^{*}}(T)}$. 
${\overline{\acx}^{w^{*}}(M)} \in \eqc(E^*)$ for every $M \in \eqc(E^*)$.
\end{lem}
\begin{proof}
	By definition, there is a neighborhood $\eps \subseteq E$ of zero such that $M \subseteq \eps^{\circ}$.
	Note that $\eps^{\circ}$ is convex and weak-star compact by %the
	 Alaouglu-Bourbaki's theorem. %(Fact \ref{l:Al-Bo}). 
	%In particular, 
	Finally note that 
	$\overline{\acx}^{w^{*}}(M)$ is a weak-star closed subspace of $\eps^{\circ}$. 
	% and therefore weak-star compact and equicontinuous.
\end{proof}

%090122q  Let's discuss this lemma. Maybe it is enough to consider the extension to $PM)$ (probablity measures) which is a subset of $B_{C(M)^{*}}$. The affine compactification $M \to P(M)$ has the "maximality property" ... So it majores any other affine compactification. If this approach is enough for your purposes, then very good.  

%0222q let's discuss  
\begin{lem} \label{lemma:functional_dual_surjection}
	Let %0202 $E$ be a locally convex space and 
	$M \subseteq E^{*}$ be 
	%310322	Added equicontinuous to fix Saak problems
	an equicontinuous, 
	weak-star compact subset.
	Write 
	%090122 which closure ?
	${T := \overline{\acx} M}$.
	%1012	Made simpler with Lemma 
	Then there exists a 
	%0202 "surjective" what ... continuous map ? ...
	surjective $j\colon B_{C(M)^{*}} \to T$ such that 
	%2012 Made slight change r_{M} to r_{T}
	%0202q  hard to understand
	%1101q	Yes, is is a very formal explanation to a very intuitive fact
	%190222 minor fixes and added a diagram
	${r_{T}(x) \circ j  = r_{B_{C(M)^{*}}}(r_{M}(x))}$ 
	for every $x \in E$.
	% https://q.uiver.app/?q=WzAsMyxbMCwwLCJCIl0sWzEsMCwiVCJdLFsxLDEsIlIiXSxbMCwxLCJqIl0sWzEsMiwiciJdLFswLDIsInJyIiwyXV0=
	%1503 let's discuss one of the arrows notation  
	\[\begin{tikzcd}
		B_{C(M)^{*}} & T \\
		& \R
		\arrow["j", from=1-1, to=1-2]
		\arrow["r_{T}(x)", from=1-2, to=2-2]
		\arrow["\gamma"', from=1-1, to=2-2]
	\end{tikzcd}\]
where $\gamma:=r_{B_{C(M)^{*}}}(r_{M}(x))$. 
	%090122 do you mean "for every $x \in B_{C(M)^{*}}}$." ?
\end{lem}
\begin{proof}  
	Let $i\colon M \to B_{C(M)^{*}}$ be the inclusion map.
	Define $M' := \acx i(M)$ and $j' \colon M' \to \acx M$ via:
	$$
	  j'\left( \sum_{m=1}^{n} \alpha_{m} i(\varphi_{m}) \right) :=
	  \sum_{m=1}^{n} \alpha_{m} \varphi_{m}.
	$$
	It is easy to see that this function is well-defined and linear. 
	%090122 "linear" ? Or, do you mean, "affine" ?  
 	We will now show that it is 
	%090122 "uniformly continuous" under which uniform structures??
	uniformly continuous
	%090122 
	with respect to the standard uniformities of the weak-star topologies.
	Let us write
	$$
	U_{X}(p_{1}, \dotsc, p_{t}; \eps) := 
	\{(f, g) \in X \times X \mid \forall 1 \leq k \leq t: \lvert f(p_{k}) - g(p_{k}) \rvert < \eps \},
	$$
	%310322 added for clarification after Saak questions
	where $X \in \{M', \acx M\}$.
	%190222 removed following line
%	and $r\colon E \to C(M)$ for the restriction operator.
	%090122 Indeed, 
	%310322 Changed $A$ to $E$ by Saak suggestion
	Suppose that $x_{1}, \dotsc, x_{t} \in E$ and $\eps > 0$.
	It is easy to see that:
	$$
	  j'\left( U_{M'}(r_{M}(x_{1}), \dotsc, r_{M}(x_{t}); \eps) \right) \subseteq
	  U_{M}(x_{1}, \dotsc, x_{t}; \eps).
	$$
	By definition, $j'$ is uniformly continuous.
	
	Now, by Lemma \ref{lemma:weak_star_convex_hull_of_equicontinuous}, $T := \overline{\acx}^{w^{*}} M$ is compact, and in particular complete.
	%0812 I don't understand your comment. Let's discuss
	%2911	Actually, N=B_{C(M)^{*}}, but I couldn't find a reference so far
	%1012	Made it simpler
	We can thus extend $j'$ to a continuous function $j \colon \overline{M'}^{w^{*}} \to T$.
	By Lemma \ref{lemma:discrete_measures_are_dense}, $\overline{M'}^{w^{*}} = B_{C(M)^{*}}$.
		
	Since $B_{C(M)^{*}}$ is compact, $j(B_{C(M)^{*}})$ is closed.
	Moreover, $\acx M = j(M') \subseteq j(B_{C(M)^{*}})$ so it is also dense.
	As a consequence, $j$ is surjective.
	
	Now suppose that $x \in E$ and $\Phi = \sum_{m=1}^{n} \alpha_{m} i(\varphi_{m}) \in M' \subseteq B_{C(M)^{*}}$.
	We thus have:
	\begin{align*}
		(r_{T}(x) \circ j)(\Phi) =\ &
		r_{T}(x)(j(\Phi)) = 
		r_{T}(x)(j'(\Phi)) \\
		=\ & r_{T}(x)\left(\sum_{m=1}^{n} \alpha_{m} \varphi_{m}\right) = 
		 \left(\sum_{m=1}^{n} \alpha_{m} \varphi_{m}\right)(x) \\
		=\ & \sum_{m=1}^{n} \alpha_{m} \varphi_{m}(x) = 
		\sum_{m=1}^{n} \alpha_{m} (r_{M}(x))(\varphi_{m}) \\
		=\ & \sum_{m=1}^{n} \alpha_{m} (i(\varphi_{m}))(r_{M}(x)) = 
		\left( \sum_{m=1}^{n} \alpha_{m} i(\varphi_{m}) \right) (r_{M}(x)) \\
		=\ &\Phi(r_{M}(x)) =
		(r_{B_{C(M)^{*}}}(r_{M}(x)))(\Phi).	
	\end{align*}
	As a consequence, $(r_{T}(x) \circ j)_{\mid M'} = (r_{B_{C(M)^{*}}}(r_{M}(x)))_{\mid M'}$.
	Since $r_{T}(x) \circ j$ and $r_{B_{C(M)^{*}}}(r_{M}(x))$ are continuous, and $M'$ is dense in $B_{C(M)^{*}}$, we have $r_{T}(x) \circ j = r_{B_{C(M)^{*}}}(r_{M}(x))$.
\end{proof}

\begin{lem} \label{lemma:sufficient_for_locally_convex_bornology}
	Let $V$ be a vector space and let $\mathcal{B} \subseteq \mathcal{P}(V)$ be a family of subsets of $V$.
	Then the following requirements are enough to conclude that $\mathcal{B}$ is a convex bornology:
	\ben
		\item Singletons belong to $\mathcal{B}$.
		\item If $M \in \mathcal{B}$ and $N \subseteq M$ then $N \in \mathcal{B}$.
		\item Union of two elements of $\mathcal{B}$ remains in $\mathcal{B}$.
		\item If $M \in \mathcal{B}$ and $r \in \R$ then $r M \in \mathcal{B}$.
		\item For every $M \in \mathcal{B}$, the absolutely convex hull $\acx M \in \mathcal{B}$.
	\een
\end{lem}
\begin{proof}
	First, it is easy to see that
	$$
	  \bigcup_{A \in \mathcal{B}} A \supseteq
	  \bigcup_{x \in V} \{x\} \supseteq
	  V,
	$$
	so it only remains to show that sums of elements in $\mathcal{B}$ remains in $\mathcal{B}$.
	Indeed, if $A_{1}, A_{2} \in \mathcal{B}$, then so are $A_{1} \cup A_{2}$ and $2\acx (A_{1} \cup A_{2})$.
	Finally, note that:
	$$
	  A_{1} + A_{2} \subseteq 2\acx (A_{1} \cup A_{2}).
	$$
\end{proof}

\begin{defin}
	Let $\mathfrak{B}$ be a bornological class and 
	%0202 let $E$ be a locally convex space and 
	$A \subseteq E$ is a bounded subset.
	An equicontinuous, $M \subseteq E^{*}$ is said to be \emph{co-$\mathfrak{B}$-small 
	%2911	Added localization to help with Lemma \ref{lemma:converse_of_locally_convex_co_b_small}
	with respect to $A$} if $r(A) \in \mathfrak{B}_{\overline{M}^{w^{*}}}$, where $r \colon E \to C(\overline{M}^{w^{*}})$ is the restriction map.
	If this is true for every bounded subset of $E$, then we will simply say that $M$ is \emph{co-$\mathfrak{B}$-small}.
\end{defin}

%\begin{lemma}
%	Suppose that $K$ is a compact set and $A \subseteq C(K)$ is a bounded family.
%	Furthermore, suppose that $\{K_{n}\}_{n \in \N}$ is a family of compact sets and $\varphi_{n}\colon K_{n} \to K$ are continuous such that
%	$$
%	  I := \bigcup_{n \in \N} \im(\varphi_{n}) \subseteq K
%	$$
%	is dense.
%	If $A \circ \varphi_{n} \subseteq C(K_{n})$ is fragmented for every $n \in \N$, then $A$ is fragmented.
%\end{lemma}
%\begin{proof}
%	Let $C \subseteq K$ be a non-empty set and $\eps > 0$.
%	We need	to find an open subset $O \subseteq K$ such that $C \cap O \neq \emptyset$ and $f(C \cap O)$ is $\eps$-small for every $f \in A$.
%	
%	By Lemma \ref{l:FrFam}.1, we can assume without loss of generality that $C$ is closed in $K$.
%	Therefore, $C$ is compact.
%\end{proof}

%240222 No more need for the extendability property
\begin{lem} \label{lemma:co_small_extension}
	Suppose that $\mathfrak{B}$ is a bornological class, $E$ a locally convex space, $M \in \eqc(E^{*})$ and $A \subseteq E$ is bounded.
	If $M$ is co-$\mathfrak{B}$-small with respect to $A$, then so is $M' := M \cup \{0\}$.
\end{lem}
\begin{proof}
	If $0 \in M$ then we are done.
	Otherwise, recall that $M$ is compact and therefore closed.
	As a consequence, $0$ is isolated in $M'$.
	Choose $\varphi_{0} \in M$ and write $N := \{0, \varphi_{0}\}$.
	We now define $j\colon M' \to M$ and $s\colon M' \to N$ by
	$$
	  j(\varphi) := \begin{cases}
	  	\varphi & \varphi \neq 0 \\
	  	\varphi_{0} & \varphi = 0
	  \end{cases},\ 
  	  s(\varphi) := \begin{cases}
  	  	0 & \varphi \neq 0 \\
  	  	\varphi_{0} & \varphi = 0.
  	  \end{cases}.
	$$
	Note that both of these maps are continuous.
	%030422	Made more explicit by suggestion of Saak
	In virtue of Remark \ref{lemma:euclidean_bornology}, $\mathfrak{B}_{N}$ is the euclidean bornology.
	Namely, every bounded subset belongs to $\mathfrak{B}_{N}$.
	In particular, so does $r_{N}(A) \in \mathfrak{B}_{N}$.
	Thus, ${r_{N}(A) \circ s \in \mathfrak{B}_{M'}}$.
	Moreover, it is easy to see that $r_{M}(A) \circ j \in \mathfrak{B}_{M'}$.
	We claim that
	$$
	 r_{M'}(A) \subseteq r_{M}(A) \circ j - r_{N}(A) \circ s.
	$$
	More specifically, we claim that for every $x \in A$:
	$$
	  r_{M'}(x) = r_{M}(x) \circ j - r_{N}(x) \circ s.
	$$
	Indeed, suppose that $x \in A$, then for every $\varphi \in M$ we have:
	$$
	 (r_{M'}(x))(\varphi) = \varphi(x) = 
	 (j(\varphi))(x) =
	 (j(\varphi))(x) - (s(\varphi))(x) = 
	 (r_{M}(x) \circ j - r_{N}(x) \circ s)(\varphi).
	$$
	Also:
	$$
	  (r_{M'}(x))(0) = 
	  0(x) =
	  0 =
	  \varphi_{0}(x) - \varphi_{0}(x) = 
	  (j(0))(x) - (s(0))(x) = 
	  (r_{M}(x) \circ j - r_{N}(x) \circ s)(0).
	$$
	As a consequence, we have verified the identity for every $\varphi \in M'$.
	Since $\mathfrak{B}_{M'}$ is a vector bornology we conclude that $r_{M'}(A) \in \mathfrak{B}_{M'}$ as required. 
\end{proof}

%0202q let's discuss 
\begin{lemma} \label{lemma:co_b_small_is_locally_convex}
	Let $\mathfrak{B}$ be a polarly compatible bornological class  
	%0202 $E$ be a locally convex space 
	and let $A \subseteq E$ be bounded.
	The family of co-$\mathfrak{B}$-small subsets with
	%2911	Added localization to help with Lemma \ref{lemma:converse_of_locally_convex_co_b_small}
	respect to $A$ of $E^{*}$ is a weak-star saturated, convex bornology. 
	%1511	Sparated notation to new sentence
	Denote this bornology as
	%1910	I think this notation is bad, but I couldn't find a better one so far.
	$\localcobsmall{E}{A}$.
	We also write 
	$$
	  {\cobsmall{E} := \bigcap\limits_{A \subseteq E} \localcobsmall{E}{A}}
	$$ where $A$ runs over bounded subsets.
	Clearly, $\cobsmall{E}$ is also a weak-star saturated, locally convex bornology.
\end{lemma}
\begin{proof}
	%2911 Localized the proof
	We will prove the requirements of Lemma \ref{lemma:sufficient_for_locally_convex_bornology}:
	\ben
		\item Saturated: by definition. 
		%2911 Added
		Without loss of generality, we will assume that $M$ is weak-star closed for the rest of the proof.
		\item Singletons belong to $\localcobsmall{E}{A}$: 
		%2911	Added a bit more explanation.
		Consequence of Remark \ref{remark:extensibility_induction}.
		\item If $M \in \localcobsmall{E}{A}$ and $N \subseteq M$, then $N \in \localcobsmall{E}{A}$: Consider the inclusion map $i: N \to M$. 
		Let $r_{M}: E \to C(M)$ and $r_{N}: E \to C(N)$ be the restriction maps.
		By definition, $r_{M}(A) \in \mathfrak{B}_{M}$.
		In virtue of the consistency property, $r_{M}(A) \circ i \in \mathfrak{B}_{N}$.
		Also, note that $r_{N}(A) = r_{M}(A) \circ i$.
		By definition, $N \in \localcobsmall{E}{A}$.
		\item Union of finite elements of $\localcobsmall{E}{A}$ remains in $\localcobsmall{E}{A}$: Let $M$ and $N$ be members of   $\localcobsmall{E}{A}$.  We show that ${M \cup N \in \localcobsmall{E}{A}}$.
		Consider the space $P := M \times N \times \{1, 2\}$ and the continuous function $T\colon P \to (M \cup N)$ defined by:
		$$
		  T(\varphi, \psi, n) := \begin{cases}
		  	\varphi & n = 1 \\
		  	\psi & n = 2\\
		  \end{cases}.
		$$
		By definition, $r_{M}(A) \in \mathfrak{B}_{M}, r_{N}(A) \in \mathfrak{B}_{N}$;
		we will check that $r_{M \cup N}(A) \in \mathfrak{B}_{M \cup N}$.
		Because $T$ is surjective, and in virtue of the consistency property, we can equivalently show that $r_{M \cup N}(A) \circ T \in \mathfrak{B}_{P}$.
		
		%240202	No more need for the extendability property.
		In virtue of Lemma \ref{lemma:co_small_extension}, we can assume without loss of generality that $0 \in M$ and $0 \in N$.
		Now, consider the maps $\theta_{M} \colon P \to M$ and $\theta_{N} \colon P \to N$ defined as
		$$
		  \theta_{M}(\varphi, \psi, n) := \begin{cases}
		  	\varphi & n = 1 \\
		  	0 & n = 2 \\
		  \end{cases}
		$$
		$$
		  \theta_{N}(\varphi, \psi, n) := \begin{cases}
		  	0 & n = 1 \\
		  	\psi & n = 2 \\
		  \end{cases}.
		$$
		Using the consistency property, we know that $(r_{M}(A) \circ \theta_{M}) + (r_{N}(A) \circ \theta_{N}) \in \mathfrak{B}_{P}$.
		We claim that
		$$
		  r_{M \cup N}(A) \circ T \subseteq 
		  (r_{M}(A) \circ \theta_{M}) + (r_{N}(A) \circ \theta_{N}),
		$$
		completing this part of the proof.
		Indeed, for every $x \in A$, we have:
		$$
		  r_{M \cup N}(x) \circ T = 
		  (r_{M}(x) \circ \theta_{M}) + (r_{N}(x) \circ \theta_{N}).
		$$
		\item Suppose that $M \in \localcobsmall{E}{A}$ and $\alpha \in \R$, we show that $\alpha M \in \localcobsmall{E}{A}$.
		Consider the scalar map $S_{\alpha} \colon M \to \alpha M$ defined by $S_{\alpha}(\varphi) := \alpha\varphi$.
		By definition, ${r_{M}(A) \in \mathfrak{B}_{M}}$.
		Note that 
		%030422	Fixed S_{r} to S_{\alpha} by Saak suggestion.
		$r_{M}(A) = r_{\alpha M}(A) \circ S_{\alpha}$.
		In virtue of consistency, $r_{\alpha M}(A) \in \mathfrak{B}_{\alpha M}$. 
		\item If $M \in \localcobsmall{E}{A}$, then so is its closed absolutely convex hull 
		%090122  double dollar is better 
		$${T := \overline{\acx}^{w^{*}} M \in \localcobsmall{E}{A}}.$$ 
		Let ${i\colon M \hookrightarrow B_{C(M)^{*}}}$ be the inclusion map.
		%1012	Changed because lemma:functional_dual_surjection became simpler
		%2012	There was a slight error here. Changed r_{T} to r_{M}
		%310322 Specifically mentioning equicontinuity by Saak suggestion.
		By definition, $M$ is equicontinuous, so using Lemma \ref{lemma:functional_dual_surjection}, consider a continuous surjection ${j\colon B_{C(M)^{*}} \to T}$ such that for every $x \in E$, $r_{T}(x) \circ j = r_{B_{C(M)^{*}}}(r_{M}(x))$.
		In virtue of its construction, we have:
		$$
		  r_{T}(A) \circ j = r_{B_{C(M)^{*}}}(r_{M}(A)).
		$$
		Moreover, since $\mathfrak{B}$ is polarly compatible, $r_{B_{C(M)^{*}}}(r_{M}(A)) \in \mathfrak{B}_{B_{C(M)^{*}}}$.
		%2012 no longer necessary
		%Applying (2) we conclude that $r_{B_{C(M)^{*}}}(r_{M}(A)) \in \mathfrak{B}_{B_{C(M)^{*}}}$.
		Finally, using the consistency property, we know that $r_{T}(A) \in \mathfrak{B}_{T}$, as required.
	\een
\end{proof}

%030422	Changed corollary to a definition by Saak suggestion
Recall that every convex bornology $\mathcal{B}$ on the dual $E^{*}$, induces a locally convex topology on $E$ \cite[Thm.~5:1'1(a)]{Born}.
This is the \emph{polar} topology defined by the basis:
$$
  \{ M^{\circ} \mid M \in \mathcal{B} \}.
$$
%030422	Added by Saak suggestion
Also, if $\mathcal{B}$ consists of equicontinuous subsets only, then its polar topology is weaker than the original topology \cite[Thm.~5:1'3]{Born}
\begin{defin}
	Let $\mathfrak{B}$ be a polarly compatible bornological class, and $(E, \tau)$ be a locally convex space.
	Recall that Lemma \ref{lemma:co_b_small_is_locally_convex} applies in this case so $\cobsmall{E}$ is a convex bornology.
	We define $\tau_{\mathfrak{B}}$ to be the polar topology generated by $\cobsmall{E}$.
	Since $\cobsmall{E}$ consists of equicontinuous subsets, $\tau_{\mathfrak{B}} \subseteq \tau$.
\end{defin}

%2911	Added to explain the definition
Note that the requirements of Definition \ref{defin:polarly_compatible} are necessary to obtain Lemma \ref{lemma:co_b_small_is_locally_convex}, as can be seen in the following lemma.

%0812 After searching this label I could not find if you ever use this lemma 
%1012 I am using it in the next corollary (without numeric reference)
%1012	I also think that it stands by its own since it validates our definition.
\begin{lemma} \label{lemma:converse_of_locally_convex_co_b_small}
	Suppose that $\mathfrak{B}$ is a bornological class such that for every 
	%locally convex space 
	%010322 added some details
	$E$ and every bounded $A \subseteq E$, $\localcobsmall{E}{A}$ is a saturated locally convex bornology (with respect to the weak-star topology).
	Then $\mathfrak{B}$ is polarly compatible.
	
	%2404 no longer necessary
%	More specifically, if $\localcobsmall{E}{A}$ is closed under the operation of taking weak-star closed absolutely convex hulls, then $\mathfrak{B}$ has the duality property.
\end{lemma} 
\begin{proof}
	%2002M made a bit simpler
	Suppose that $K$ is compact and $A \in \mathfrak{B}_{K}$.
	Write $r\colon C(K) \to C(B_{C(K)^{*}})$ for the restriction map and $i\colon K \to B_{C(K)^{*}}$ for the evaluation map.
	Also, define $K' := i(K)$.
	%240222 made simpler
	We need to show that ${r(A) \in \mathfrak{B}_{B_{C(K)^{*}}}}$.
	%250222 minor change
	It is easy to see that $A = r(A) \circ i$.
	As a consequence, $r(A) \in \mathfrak{B}_{K'}$.
	%2012 added more details
	In other words, ${r(A) \in \localcobsmall{C(K)^{*}}{K'}}$ and therefore
	$$
	  {r(A) \in \localcobsmall{C(K)^{*}}{\overline{\acx}^{w^{*}}K'}}.
	$$
	As a consequence, 
	$$r(A) \in \mathfrak{B}_{\overline{\acx}^{w^{*}} K'}.$$
	However, by Lemma \ref{lemma:discrete_measures_are_dense}, $\overline{\acx}^{w^{*}} K' = B_{C(K)^{*}}$, as required.
	%0812  let's check BJM89 70 book page, together
	%2911q	Do you know a reference for this fact? I couldn't find one. 
	%0812 I will try to find a reference. If not, then we can refer to that book of BJM (exercise). It implies what we need using Hahn-Banach thm. 
	%1012	I've made a reference.
	%2012	Added details
\end{proof}

\begin{cor}
	The class $\BDLP$ is polarly compatible.
\end{cor}
\begin{proof}
	%2402	Made simpler
	A consequence of Fact \ref{f:young} and Lemma \ref{lemma:converse_of_locally_convex_co_b_small}.
	Note that in Fact \ref{f:young}, $E$ and $F$ are interchangeable.
\end{proof}

\begin{lemma} \label{lemma:b_topology_relations}
	For every locally convex space $(E, \tau)$, we have:
	$$
	  \tau_{w} \subseteq \tau_{\mathfrak{B}} \subseteq \tau \subseteq \tau_{\mu}
	$$
	where $\tau_{w}$ is the weak topology and $\tau_{\mu}$ is the Mackey topology.
\end{lemma}
\begin{proof}
	First, let us note that
	%030422	Added explanation about equicontinuous subsets (by saak suggestion)
	$$
	\mathcal{F} \subseteq \cobsmall{E} \subseteq \mathcal{E} \subseteq \mathcal{C}
	$$
	where $\mathcal{F}, \mathcal{E}$ and $\mathcal{C}$ are the bornologies of finite subsets, equicontinuous subsets, and weak-star compact absolutely convex subsets, respectively.
	It is known that their respective polar topologies are $\tau_{w}, \tau$ and $\tau_{\mu}$ (\cite[p.~131, Cor.~1]{Schaefer}).
	As a consequence:
	$$
	  \tau_{w} \subseteq \tau_{\mathfrak{B}} \subseteq \tau \subseteq \tau_{\mu}.
	$$
\end{proof}
%1910	New
\begin{lemma} \label{lemma:polar_is_b_small}
	Let $(E, \tau)$ be a locally convex space.
	The 
	%090122 what does this mean ? 
	topology $\tau_{\mathfrak{B}}$ is $\mathfrak{B}$-small.
\end{lemma}
\begin{proof}
	Let $A \subseteq E$ be a bounded subset and $M \subseteq E^{*}$ be a weak-star compact, equicontinuous subset (with respect to $\tau_{\mathfrak{B}}$).
	We will show that $r(A) \in \mathfrak{B}_{M}$.
	
	By Lemma \ref{lemma:b_topology_relations}:
	$$
	  \tau_{w} \subseteq \tau_{\mathfrak{B}} \subseteq \tau_{\mu}.
	$$
	%030422	Removed by saak suggestion: where $\tau_{\mu}$ is the Mackey topology.
	By \cite[p.~132, Corollary 2]{Schaefer}, we conclude that the bounded subsets of $\tau_{\mathfrak{B}}$ are the same as those of $\tau$.
	Thus, $A$ is bounded with respect to the original topology $\tau$.
	
	Since $M$ is equicontinuous, we can find a neighborhood $\eps$ of zero in $E$ such that $M \subseteq \eps^{\circ}$.
	By definition, we can find a subset $N \in \cobsmall{E}$ such that $N^{\circ} \subseteq \eps$.
	We therefore have
	$$
	  N^{\circ \circ} \supseteq \eps^{\circ} \supseteq M.
	$$
	By the bipolar theorem (Fact \ref{fact:bipolar_theorem}), 
	$$
	  N^{\circ \circ} = \overline{\acx}^{w} N.
	$$
	By Lemma \ref{lemma:co_b_small_is_locally_convex}, $\cobsmall{E}$ is weak-star saturated and locally convex, so 
	$${N^{\circ \circ} = \overline{\acx}^{w^{*}} N \in \cobsmall{E}}.$$
	In particular, ${r_{N^{\circ\circ}}(A) \in \mathfrak{B}_{N^{\circ\circ}}}$,  hence $r_{M}(A) \in \mathfrak{B}_{M}$.
\end{proof}
%1910	New
\begin{thm} \label{thm:strongest_b_topology}
	For every lcs %locally convex space 
	$(E, \tau)$, $\tau_{\mathfrak{B}}$ is the strongest locally convex, $\mathfrak{B}$-small topology coarser than $\tau$.
\end{thm}
\begin{proof}
	First, by Lemma \ref{lemma:polar_is_b_small}, $\tau_{\mathfrak{B}}$ is indeed a $\mathfrak{B}$-small topology.
	Now, suppose that $\tau_{\mathfrak{B}} \subseteq \sigma \subseteq \tau$ is a locally convex $\mathfrak{B}$-small topology.
	By Lemma \ref{lemma:b_topology_relations}, we can write
	$$
	  \tau_{w} \subseteq \tau_{\mathfrak{B}} \subseteq \sigma \subseteq \tau \subseteq \tau_{\mu}.
	$$
	In virtue of \cite[p.~132, Cor.~2]{Schaefer}, we know that $\tau_{\mathfrak{B}}$ and $\sigma$ have the same bounded sets.
	Now, suppose that $\eps \in \sigma$ is a neighborhood of zero.
	We can find a $\delta \in \sigma$ such that ${\overline{\acx} \delta \subseteq \eps}$.
	Note that because both $\tau_{\mathfrak{B}}$ and $\sigma$ are compatible with the dual $E^{*}$, the closure of convex sets (like $\acx \delta$) is equal to the weak closure 
	%1910 the reference actually states that every point outside the closure can be separated with a functional, but it is close enough in my opinion.
	(\cite[p.~131, Cor.~6]{Jarchow}).
	We will show that $\overline{\acx} \delta \in \tau_{\mathfrak{B}}$, proving that $\sigma \subseteq \tau_{\mathfrak{B}}$.
	
	Since $\sigma$ is $\mathfrak{B}$-small, we know that $r(A) \in \mathfrak{B}_{\delta^{\circ}}$ for every bounded $A \subseteq E$.
	Since we already established that the bounded subsets of $\tau_{\mathfrak{B}}$ and $\sigma$ agree, it means that $\delta^{\circ} \in \cobsmall{E_{\tau}}$.
	Again, by definition, $\delta^{\circ \circ} \in \tau_{\mathfrak{B}}$.
	Using the Bipolar Theorem (Fact \ref{fact:bipolar_theorem}), we know that $\overline{\acx} \delta = \delta^{\circ \circ}$, as required.
\end{proof}
%150122 removed factorization for bornological classes
%
%%2402	New, might have importance for functionals on Banach algebras
%%		I think that I have a proof but it isn't complete yet.
%\begin{conj}
%	Let $\mathfrak{B}$ be a polarly compatible bornological class and $V$ a Banach space.
%	A compact subset $M \in \eqc(V^{*})$ is co-$\mathfrak{B}$-small if and only if it is $\mathfrak{B}$-small in $V^{*}$.
%\end{conj}
%\begin{proof}
%	First, suppose that $M$ is co-$\mathfrak{B}$-small, we will prove that $M$ is $\mathfrak{B}$-small in $V^{*}$.
%	Let ${\mathcal{M} \in \eqc(V^{**})}$.
%	We need to show that ${r_{\mathcal{M}}(M) \in \mathfrak{B}_{\mathcal{M}}}$.
%	By definition
%	
%	By Goldstine's theorem, 
%\end{proof}
\sk 
%1910	changed this section to the language of bornological classes
%030422	Change analog to analogue by Saak suggestion
\section{Locally convex analogues of reflexive and Asplund Banach spaces} 

\subsection*{DLP locally convex spaces}

The following well-known observation can be derived by results of \cite[Appendix A]{BJM}. 
\begin{f} \label{f:SepCont} 
	Let $F \times K \to \R$ be a bounded separately continuous map where $F$ and $K$ are compact. Then $F$, as a family of maps on $K$, is DLP.  
\end{f}

%0108 new place of this Fact
The following result follows from Remark \ref{r:repres}. 

\begin{f} \label{f:DLPisFRAGM}  \cite{GM-tame}
	Let $K$ be a compact space and $F \subset C(K)$ be a norm bounded subset. So if $F$ is DLP on $K$ then $F$ is a fragmented family on $K$. 
\end{f}

Following the lines of Proposition \ref{prop:examples_of_bornological_categories} and Definition \ref{def:b_small_space}, we give the following definition.
%0108  Two definitions into one  
\begin{defin} \label{d:DLPlcs} 
	Let $E$ be a locally convex space.
	\begin{itemize}
		\item We say that a bounded subset $B$ of $E$ is \emph{DLP} if $B \in \bsmall[\BDLP]{E}$. 
		Explicitly, $B$ is DLP if it is DLP as a family of functions over every %weak-star compact, equicontinuous subset $M \subseteq E^{*}$.
		$M \in \eqc(E^{*})$.
		\item $E$ is said to be DLP, and write $E \in \DLP$ if $E$ is $\BDLP$-small.
		In other words, ${E \in \DLP}$ if and only if every bounded subset of $E$ is DLP.
	\end{itemize}  
\end{defin}

The following is a direct consequence of Theorem \ref{thm:properties_of_b_small_classes}.
\begin{thm} \label{thm:properties_of_dlp_class}
	The class $\DLP$ is closed under
	%030422	Added by Saak suggestion
	taking:
	\ben
	\item subspaces
	\item bound covering maps
	\item %(possibly infinite) 
	products
	\item %(possibly infinite) 
	direct sums
	\item inverse limits.
	\een
	Moreover, if $F$ is a large, dense subspace of the locally convex space $E$, and $F \in \DLP$, then $E \in DLP$.
	In particular, if $V$ is a normed DLP space, then so is its completion.
\end{thm}

%2306 
%0108 upgrating it from Remark to Proposion
\begin{prop} \label{p:w-compDLP} 
	Every relatively weakly compact subset $B$ in a lcs $E$ 
	is DLP. 
\end{prop}
\begin{proof}
	%0108 
	%2012 Changed K to M
	Let $B$ be weakly compact in a lcs $E$ and $M \in \eqc(E^*)$. Then the natural map $w \colon B \times M \to \R$ is separately continuous. Observe that $w$ is a bounded map. Indeed, there exists a neighborhood $O$ of zero in $E$ such that $|u(x)| <1$ for every  $x \in O$ and $u \in M$. Since $B$ is bounded, there exists $c \in \R$ such that $B \subset c O$. Then $|u(x)| <c$ for every  $x \in B$ and $u \in M$. Now, by Fact \ref{f:SepCont}, we obtain that $B$ is DLP on $M$. 
	\end{proof}

The converse is true for Banach spaces.
Namely, a bounded subset $B$ of a Banach space $V$ is DLP iff $B$ is relatively weakly compact (see \cite[Thm.~A5]{BJM}). 
As a consequence, a Banach space $V$ is reflexive iff it is DLP.

\begin{f} \label{f:semi_reflexive_heine_borel}  %1907 \cite[15.2.4]{TVS} %Groth. p. 69 Thm 8
	\cite[Ch. IV, 5.5]{Schaefer} 
	A locally convex space is semi-reflexive if and only if every bounded subset is %2507 weakly precompact.
	relatively weakly compact.
\end{f}

Recall that a lcs $E$ is said to be \textit{boundedly-complete} (or, \textit{quasi-complete, \cite{Jarchow}}) if every closed bounded subset in $E$ is complete. An equivalent condition is that every bounded Cauchy net converges.  
%0202 new place 
Every boundedly-complete lcs is sequentially complete and every sequentially complete lcs is locally complete (Definition \ref{d:gauge}).  
%0108 
Note that every weakly compact subset in a lcs is complete, \cite[p.~90]{Groth}. This implies (by Fact \ref{f:semi_reflexive_heine_borel}) that every semi-reflexive space is boundedly-complete. In Theorem \ref{thm:dlp_iff_semireflexive} we make this more precise.

\begin{thm} \label{thm:dlp_iff_semireflexive}  
	$E$ is semi-reflexive if and only if $E$ is boundedly-complete and DLP.  
\end{thm}
\begin{proof}
	%1910	Now also using f:semi_reflexive_heine_borel
	First suppose that $E$ is semi-reflexive.
	Then our claim is a consequence of Proposition \ref{p:w-compDLP} and Fact \ref{f:semi_reflexive_heine_borel}. 
		The converse is a conclusion of Proposition \ref{p:DLP in boundedly-complete} below. 
\end{proof}

%\begin{lem} \label{l:semireflexive_implies_dlp}
%	If $E$ is a semi-reflexive lcs, then $E$ is DLP 
%	%0108 
%and boundedly-complete. %0808 Every semi-reflexive lcs is quasi-complete by \cite[p. 144]{Schaefer} 
%\end{lem}
%\begin{proof}
%	%0108
%	Suppose that $B$ is a bounded subset in $E$ then $B$ is relatively weakly compact (Fact \ref{f:semi_reflexive_heine_borel}). Hence, $B$ is DLP by Proposition \ref{p:w-compDLP}. 
%
%	%0108 
%	If $B$ is closed and bounded then $B$ is weakly compact (because $E$ is semi-reflexive). %Then $B$ is weakly complete.
%	 Finally note that every weakly complete subset in a lcs is complete. %see for example, Grothendieck's book, page 90
%\end{proof}

%0108 This fact really confirms that DLP subsets are very rigidly related to weak compactness. What happens in lcs which are not DLP ? 
\begin{prop} \label{p:DLP in boundedly-complete} 
	%1611 
	\emph{(version of Grothendieck's result \cite[Thm.~17.12]{KN})}  
	
	Let $E$ be a boundedly-complete lcs. Then the following are equivalent for a subset $B \subseteq E$:
	\ben
		\item $B$ is bounded and DLP.
		\item $B$ is relatively weakly compact.
		%0908:	Changed letter from M to C to avoid confusion with equicontinuous subsets
		\item The closed convex hull $C := \overline{co(B)}$ is weakly compact.
	\een
\end{prop}
\begin{proof} \
	$(3) \Rightarrow (2)$  Obvious.
	
	$(2) \Rightarrow (1)$ Let $B$ be a 
	%030422	Added by Saak suggestion
	relatively
	weakly compact subset in $E$. Then it is well-known that $B$ is bounded. %\cite[Thm.~8.3.4]{Jarchow}.
	Also $B$ is DLP by Proposition \ref{p:w-compDLP}. 
	%0908:	Removed because it is no longer relevant
	%By remarks above $B$ is complete.   

	$(1) \Rightarrow (3)$ Let $B \subseteq E$ be a %convex, 
	bounded DLP subset. We have to show that $C := \overline{co(B)}$ is weakly compact. 
	First, $C$ is uniformly complete being bounded and closed in a boundedly-complete space $E$.  
	By \cite[II, 5.4, Corollary 2]{Schaefer}, 
	there exists an embedding $T \colon E \hookrightarrow V$ 
	where $V := \prod\limits_{\lambda \in \Lambda} V_{\lambda}$ and $\{V_{\lambda}\}_{\lambda \in \Lambda}$ are Banach spaces.
	%2012	Added details
	By Fact \ref{f:adjoint_is_continuous}, it is enough to show that $T(C)$ is weakly compact.
	
	%0808
	The subset $T(C)$ is complete in $T(E)$ and also in $V$. Therefore, $T(C)$ is closed in $V$. Moreover, $T(C)$ is convex. So, it is even weakly closed in $V$ (\cite[p.~131, Corollary 6]{Jarchow}). %2012 added citation
	%changed to be more precise
	As a consequence, it is enough to show that $T(C)$ is
	%1910 
	relatively weakly compact. 
	
	By Theorem \ref{thm:properties_of_dlp_class}, the projection $B_{\lambda} := \pi_{\lambda}(T(B)) \subseteq V_{\lambda}$ is DLP for every $\lambda \in \Lambda$.
	Since $B_{\lambda} \subseteq V_{\lambda}$ is a bounded DLP subset in a Banach space $V_{\lambda}$, we can apply \cite[Thm.~A.5]{BJM} to conclude that it is weakly relatively compact. 
	%0108 
	Moreover, $\overline{co(B_{\lambda})}$ is weakly compact in the Banach space $V_{\lambda}$ (Krein-Smulian theorem). 
	%Observe that $\overline{co(B_{\lambda})}=M_{\lambda}:=\pi_{\lambda}(T(M))$
	
	Now observe that 
	$T(C) \subseteq \widehat{C}:=\prod_{\lambda \in \Lambda} \overline{co(B_{\lambda})}.$ 
	By \cite[p.~137, Thm~4.3]{Schaefer}, the weak topology of the product is the product of the weak topologies.
	Using Tychonoff's Theorem, we conclude that $\widehat{C}$ is also weakly compact in $V= \prod\limits_{\lambda \in \Lambda} V_{\lambda}$.
	Finally, note that $T(C) \subseteq \widehat{C}$ is 
	%1910 relatively 
	weakly compact by definition, as required.
%0808 
%0908 made this point a bit shorter
 %Finally, note $T(M)$ is weakly closed in $V$ and also in $\widehat{M}$. This implies that $T(M)$ is weakly compact, as required. 
	\end{proof} 

%\begin{remark} \ 
%	\begin{enumerate}
	%	\item %0108 
	
	%We include our independent short proof for the sake of completeness.
	%	\item 	%0808q very probably convex hull of a DLP subset is DLP (by direct arguments) in any lcs. It is worth trying to prove this.
	%0808q 
	Proposition \ref{p:DLP in boundedly-complete} implies that in the boundedly-complete space $E$, the closed convex hull of every weakly compact subset is weakly compact (generalized Krein-Smulian theorem).  
	This need not be true in arbitrary lcs if $E$ is not boundedly-complete. 
	%0108 The following in a paper of Bill Casselman Quasi-complete.pdf I found in WEB 
	%Let us try better reference  
	For example, let $E$ be the normed subspace of $l^2$ which consists of all sequences of finite support.
	Then the set $\left\{\frac{1}{n} e_n\right\}_{n \in \N} \cup \{0_{l^2}\}$ is compact but the closure of its convex hull is not compact (and even not complete). Note that $E$ is DLP because it is a subspace of $l^2$. 
	By Theorem \ref{thm:properties_of_dlp_class}, the DLP is hereditary.  
	Since $E$ is not semi-reflexive, this example shows that boundedly-completeness is essential also in Theorem \ref{thm:dlp_iff_semireflexive}. 
	%2208q Maybe we need to mention also this: in the same space $E$ (which is a subspace of $l^2$ every bounded zero-neihborhood is DLP in $E$ which is not relatively weakly compact in $E$.
	%0109: added by your suggestion
	%0609 Even more so, 
	Moreover, 
	in this space every bounded neighborhood of the origin is DLP but not relatively weakly compact.

\begin{remark} 
	We mention some interesting subclasses in DLP. 
	Among others:
	\begin{enumerate}
		\item Semi-reflexive lcs (Theorem \ref{thm:dlp_iff_semireflexive}); 
		%(the evaluation map $V \to V^{**}$ is onto (hence, bijective)). 
		%1907	Semireflexivity means that every bounded subset is relatively weakly compact. 
		\item Schwartz lcs (as a subspace of a reflexive lcs); 
		\item Quasi-Montel lcs (every bounded subset is uniformly precompact). Every nuclear lcs is quasi-Montel. 
		%Every Schwart space is quasi-Montel. 
		 Important examples in Analysis: the spaces  $C^{\infty}(\Omega)$ and $D(\Omega)$ (for an open subset $\Omega$ in $\R^n$). 
		Also, the space of analytic functions $H(\Omega)$ over a domain; 
		\item For every locally convex space $E$, the lcs $(E,w)$ with its weak topology is $\DLP$. Indeed, $(E,w)$ is a subspace of $\R^E$; 
		\item Every space $C_p(X)$, in its pointwise topology (for every topological space $X$), is $\DLP$. Indeed, $C_p(X)$ is a subspace of $\R^X$. 
%		\item 
%		\item 
%		\item 
%		\item 
%		\item 
	\end{enumerate} 
\end{remark}

\sk 
%1910	Turned into a subsection
\subsection*{Namioka--Phelps $\NP$ locally convex spaces}

\sk 
Recall the following well-known characterization of Asplund Banach spaces. 

\begin{f} \label{f:Asp-charact}  
	%1611 
%	Let $(V,||\cdot||)$ be a Banach space. The following are equivalent 
%	\begin{enumerate}
%		\item $V$ is Asplund (the
%		dual of every separable (Banach) subspace is separable);  
%		\item Every bounded weakly-star compact subset $K \subset E^*$ is (weak$^*$,strong)-fragmented. 
%	\end{enumerate}
A Banach space $(V,||\cdot||)$ is Asplund (the dual of every separable (Banach) subspace is separable) iff very bounded weakly-star compact subset $K \subset E^*$ is 
%090122 (weak$^*$,strong)-fragmented.  
(weak$^*$,norm)-fragmented.  
\end{f}

The second assertion can be reformulated 
%in terms of fragmentability
 as follows: the unit ball in $E$ is a fragmented family of functions on the unit ball of the dual space.

\sk

The following locally convex version of Asplund spaces was introduced in \cite{Me-fr}.

\begin{defin} \label{d:NP} \cite{Me-fr} 
	Let $E$ be a lcs. Say that $E$ is a Namioka--Phelps $\NP$ space if every %equicontinuous weakly-star compact subset $K  \subset E^*$ is 
	$K  \in \eqc(E^*)$ is (weak$^*$,strong)-fragmented. 
\end{defin}

%3003
\begin{defin} \label{d:AspSet} 
	We say that a bounded subset $A$ of a lcs $E$ is an \textit{Asplund set} in $E$ if 
	%2707 slight reformulation
	$A$ is fragmented on every %weak-star compact equicontinuous subset 
	$K  \in \eqc(E^*)$.  
\end{defin}

%1910	New
Like in Definition \ref{d:DLPlcs}, both of these definitions can be reformulated in terms of a bornological class $\BNP$.
We can also formulate the following theorem as a consequence of Theorem \ref{thm:properties_of_b_small_classes}.

\begin{thm} \label{thm:properties_of_np_class}
	The class $\NP$ is closed under
	%030422	Added by Saak suggestion
	taking:
	\ben
	\item subspaces
	\item bound covering maps
	\item %(possibly infinite) 
	products
	\item %(possibly infinite) 
	direct sums
	\item inverse limits.
	\een
	Moreover, if $F$ is a large, dense subspace of the locally convex space $E$, and $F \in \NP$, then $E \in \NP$.
\end{thm}

%\begin{remark} \label{r:ReformAsp} 
%	In terms of Definition \ref{d:AspSet} a lcs $E$ is $\NP$  means that every bounded subset of $E$ is an Asplund set. 
%\end{remark}

%3003
\begin{remark} \label{r:Asp} 
	Recall that a bounded subset $A$ in a Banach space $V$ is said to be an \textit{Asplund subset} (Fabian \cite[p.~22]{Fabian}), or, a \textit{Stegall subset} in terms of Bourgin \cite[p.~121]{Bourgin} if the pseudometric space $(V^*,\rho_C)$ is separable for every 
	%0704 "countable" was missed
	countable 
	$C \subset A$, where 
	$$
	\rho_C(\phi,\psi):=\sup_{x\in C} |\phi(x)-\psi(x)|. 
	$$	
%	It is a nontrivial fact (see, for example,  \cite{GM-MTame}) that  
%0704
Equivalently, $(B_{V^*},\rho_C)$ is separable. 

\sk 
%0704 
This definition is compatible with Definition \ref{d:AspSet} as it follows from Lemma \ref{t:countDetermined}. 
\end{remark}

\begin{lem} \label{lem:NP_equivalence}
	The following conditions are equivalent:
	\begin{enumerate}
		\item $E$ is $\NP$;  
		% (every equicontinuous weak$^*$ compact subset $M \subset E^*$ is (weak$^*$,strong)-fragmented.)
		
		\item every 
		%2212 
		(countable) 
		%2206(Important) Do you mean an Aspluns subset as defined in the previous remark? there is no other definition and it isn't the one you are using
		%2306 seems to be OK
		bounded $B \subset E$ subset is an Asplund subset in $E$. 
	\end{enumerate}
\end{lem}
\begin{proof}
	%0704 changing slightly the proof in the "style if and only if" 
	$E$ is $\NP$ means that every 
	subset $K \in \eqc(E^{*})$ is (weak*, strong)-fragmented. 
	That is, for every subset $A \subseteq K$, $\eps > 0$ and bounded set $B \subseteq V$, there exists a weak-star open set $O \subseteq E^{*}$ such that $A \cap O$ is nonempty and 
	$(\eps, B)$-small.
	This is equivalent to $(A \cap O)(x)$ being $\eps$-small for every $x \in B$. 
	This just means that $B$ is a fragmented family of functions on $K$.
	%100422	Changed proof from here
	By Lemma \ref{t:countDetermined}, it is equivalent that for every countable subset $C$ of $B$ is fragmented on $K$.

%	It is easy to conclude that $(E^{*}, \rho_{C})$ is separable, making $B$ into an Asplund subset (for every bounded $B$).
%	
%	
%	
%	proving the equivalence of these two statements.
%
%	The fragmentability of families of continuous real functions on compact spaces is ``countably 
%	%090122 defined" 
%	determined" by \cite[Thm.~2.9]{GM-MTame}. So, without loss of generality, we may suppose that $B$ is countable. 
\end{proof} 
%\begin{proof} Use Remark \ref{r:reform}. 	\end{proof} 

Note that $\NP$ lcs has several remarkable properties. For the continuity of dual actions, see \cite{Me-fr} and for the fixed point theorems, \cite{GM-fpt} and \cite{Wis}. 

\sk 
\begin{exs} \cite{Me-fr}  
	\label{ex:fromNZ} 
	%0704 Some examples of $\NP$ lcs: 
	\begin{enumerate}
		\item A Banach space is $\NP$ iff it is Asplund. 
		\item Every Frechet differentiable lcs is $\NP$. 
		\item If the dual $V^*$ 
		%090122 is uniformly Lindelof (e.g., $V^*$
		 is a linear subspace in a product of separable lcs, then $V$ is $\NP$. 
		\item 
		%0108 
		Every DLP subset in a lcs is Asplund and hence $\DLP \subset \NP$. 
	\end{enumerate}
\end{exs}

\sk  
%030422	Changed analog to analogue by Saak suggestion
\section{Locally convex analogue of Rosenthal Banach spaces} 

\sk 

%3003
\begin{defin} \label{d:TameSet} 
	We say that a bounded subset $B$ of a lcs $E$ is \textit{tame} in $E$ if one of the following equivalent conditions 
	(by Lemma \ref{f:sub-fr}) 
	is satisfied:
	\begin{itemize} 
		\item [(i)] $B$ is tame %1805 (i.e., no independent subsequence) 
		(Definition \ref{d:TameFamily}) 
		%as a family of real functions 
		over every 
		%weak-star compact equicontinuous subset 
		$K \in \eqc(E^{*})$.  
		In other words, $B \in \bsmall[\BTame]{E}$.
		\item [(ii)] $B$ is eventually fragmented over every 
		%weak-star compact equicontinuous subset 
		$K \in \eqc(E^{*})$. 
	\end{itemize}

	%Similarly, a 
	%1910	Changed to reference strong bornological classes
	In the spirit of Definition \ref{def:strong_b_small}, we will say that a
	%1511 changed strong to mackey. maybe it should be removed completely
	bounded subset is \emph{Mackey} tame if it is tame over \textit{every} weak-star compact
	%090122 
	(not necessarily equicontinuous) 
	subset $K \subseteq E^{*}$.
\end{defin}

%1304 the following properties should remain true for LCS. That is why I removed this remark to this section. Good to check these properties carefully for LCS (at least for us) 
%Better to mention this (later) 
%2208q  giving more attention and explanations 
\begin{remark} \label{r:preserv} 
	By Lemma \ref{lemma:class_induces_bornology} and Proposition \ref{prop:examples_of_bornological_categories}, the family of tame 
	%0202 about (Asplund, DLP), we do not need to say this before?
	 (Asplund, DLP) subsets in a given lcs is a
	%2208 for weak closures (in Banach spaces) see Lemma 2.2 in https://arxiv.org/pdf/1710.01044.pdf
	 \textit{convex bornology} in the sense of \cite{Born} and \textit{a saturated bornology} in the sense of \cite[p.~153]{Jarchow}. 
	 
	 %2208 
	 %2012 I think that it can be removed
%	 If $E$ is boundedly-complete then the bornology of DLP subsets is just the system of relatively weakly compact subsets and at least for such lcs this bornology is also convex and saturated. 
	 %2208  the case of DLP we still need to discuss 
\end{remark}

%0202 removing this q. It is too general 
%\sk
%\begin{prob} 
%	Find sufficient conditions which guarantee that a bounded subset $B$ in a lcs $E$ is tame or Asplund. 
%\end{prob} 

\sk 
By Lemma \ref{l:RosBanSpCharact}, a Banach space $V$ is Rosenthal iff 
every bounded subset $F \subset V$ is tame. 
%  (as a family of functions) on every bounded weak star compact (bounded) subset $K  \subset V^*$ of the dual space $V^*$. 
Motivated by this reformulation,
we introduce here the following locally convex 
%030422 Changed analog to analogue by Saak suggestion
analogue of Rosenthal Banach spaces.

%3003
%1705	Reformulated
%1805  Not sure that it is good to give here the MAIN definition together with "strong" option ... 
%0607	Separated tame and strong tame.
\begin{defin} \label{def:tame_lcs}
	A locally convex space $E$ is said to be \textit{tame} if every bounded subset $B$ of $E$ is tame.
	In other words, $E$ is tame if and only if $E$ is $\BTame$-small.
	In this case we  write $E \in \Tame$.
	
	\sk 
	Similarly, %we can define 
	%2012 Changed strong to Mackey
	%1802 reframed
	a space is \textit{Mackey tame} if every bounded subset is Mackey tame. 
	In this case we will write $E \in \sTame$.
\end{defin}

%Clearly, every strongly tame lcs is tame.
%2212 
%\begin{remark}
\sk 
	As in Lemma \ref{lem:NP_equivalence}, we may assume without loss of generality that $B$ is countable. 
	We have the inclusions (note that s. stands for ``subsets"): 
%\end{remark}
 
%\begin{remark}
%1907 adding (DLP) twice !! and new place
%2007q changed class notation, is this ok? %2507 OK 
\sk
\nt \{weakly compact s.\} $\subset$ (DLP s.) $\subset$ \{Asplund s.\} $\subset$ \{tame s.\} $\subset$  \{bounded s.\}
\sk
\nt \{semi-reflexive lcs\} $\subset \DLP \subset \NP \subset \Tame \subset \{lcs\} $

%\end{remark}

%Basic definitions and results about Banach representations of dynamical systems can be found in \cite{Me-nz,GM1,GM-rose,GM-survey,GM-tame}. 

%0507 the definition of tameness and strong tameness are equivalent for barreled spaces.
%1805 probably we have to say also something "and by a property of barreled lcs"
%0507 changing lem to prop 
\begin{prop} \label{p:tameness_in_barrelled_spaces} \ 
	\begin{enumerate}
		%1910	Changed barreled to Mackey
		\item For Mackey spaces (e.g., barreled or metrizable), strong tameness is equivalent to tameness.
		\item 	A Banach space is a tame lcs iff it is a Rosenthal Banach space. 
	\end{enumerate}
\end{prop} 
\begin{proof} 
	%1403	Made shorter
	Corollary \ref{cor:mackey_spaces_and_b_smallness} and Lemma \ref{l:RosBanSpCharact} respectively.
%	\ 
%	\ben
%		\item Corollary \ref{cor:mackey_spaces_and_b_smallness}.
%		%1403	Made shorter
%		\item Lemma \ref{l:RosBanSpCharact}.
%		By Lemma \ref{l:RosBanSpCharact}, a Banach space $V$ is Rosenthal if and only if every bounded subset of $V$ is tame for every bounded subset of $V^{*}$.	
%		Recall that %by Banach–Alaoglu Theorem, 
%		for a normed space $V$ every weak-star compact  subset of $V^{*}$ is equicontinuous.
%	\een 
	% by definition.
%	Therefore, it coincides with our definition of tameness.
\end{proof}
The following is a direct consequence of Theorem \ref{thm:properties_of_b_small_classes}.
\begin{thm} \label{thm:properties_of_tame_class}
	The class $\Tame$ is closed under
	%030422	Added by Saak suggestion
	taking:
	\ben
	\item subspaces
	\item bound covering maps
	\item %(possibly infinite) 
	products
	\item %(possibly infinite) 
	direct sums
	\item inverse limits.
	\een
	Moreover, if $F$ is a large, dense subspace of the locally convex space $E$, and $F \in \Tame$, then $E \in \Tame$.
	In particular, if $V$ is a normed tame space, then so is its completion.
\end{thm}

%090122q  What we know about the same question for (NP) and (DLP) ?? 
\begin{q} \label{q:completion} 
	Is it true that the completion of a tame lcs is always tame ? 
\end{q}

\sk 

 %1304  adding more explanation in the following proposition
 \sk 
A nonempty class of lcs is said to be a \textit{variety} \cite{DMS}
if it is closed under the operations of taking subspaces, quotients, arbitrary products and isomorphisms.  In particular, it is closed also under the inverse limits. For every subclass $K$ of locally convex spaces, the intersection of all varieties containing $K$ is a variety \textit{generated} by $K$. Notation: ${\mathcal V}(K)$.

%0806 
\begin{prop} \label{p:variety} 
	The variety ${\mathcal V}(R)$ generated by the class of all Banach tame (i.e., Rosenthal) spaces is \textbf{properly} contained in $\Tame$.  In particular, not every tame lcs can be embedded into a product of tame Banach  spaces. Similar assertion is true for $\NP$. 
\end{prop}
\begin{proof} 
	Use results of this section and also the following facts: 
	\begin{itemize} 
		\item [(a)] %\cite[Example 4.11]{Me-fr} 
		The classes %2208 
		$\DLP$, 
		$\NP$ and $\Tame$ are not closed under quotients. As it was mentioned in \cite[IV, Ex. 20]{Schaefer}, there exists a Frechet Montel 
		%2208 (hence, $\NP$) 
		(hence, $\DLP$)
		space $E$ and a closed subspace $M$ such that the quotient space $E/M$ is the Banach space $l^1$ (which, of course, is not tame). 
		\item [(b)] %\cite[Proposition 4.13]{Me-fr} 
		The class $R$ of all Banach tame (i.e., Rosenthal) spaces 
		%:=Tame-LCS $\cup$ Ban of all Rosenthal Banach spaces (i.e., Tame Banach spaces) 
		is closed under finite products and quotients, Hence by 
		\cite[Thm.~1.4]{DMS},
		the variety ${\mathcal V}(R)$ of all lcs generated by R is just the class of all subspaces in products of Rosenthal Banach spaces. This implies that the Montel Frechet space $E$ from (a), which is %2208 $\NP$ 
		$\DLP$ 
		and hence tame, cannot be embedded into a product of tame Banach spaces because a quotient of $E$ is $l^1$ (and 
		%by definition 
		variety is closed under quotients). 
	\end{itemize}  
\end{proof}

\begin{f} \label{fact:scattered_compact_spaces} \cite[Main Theorem]{CKIII}
	Let $K$ be a compact space. The following are equivalent:
	\ben
		\item $l^{1}$ cannot be embedded in $C(K)$.
		\item The dual of every separable Banach subspace of $C(K)$ is separable.
		\item $K$ is scattered.
	\een
	In other words, $C(K)$ as a Banach space is Rosenthal, iff it is Asplund iff $K$ is scattered.
\end{f}
%0607	Reformulation
This result motivates a generalization to the locally convex space $C_{k}(X)$ with the compact open topology when $X$ is not compact.
In \cite[Lemma 6.3]{Networks}, Gabriyelyan--Kakol--Kubi\'{s}--Marciszewski 
%1107 
gave a natural generalization showing 
that the same statement remains valid 
%1107 
for the lcs $C_k(X)$ where $X$ is a Tychonoff space $X$ which is not necessarily compact,
if we require that every compact subset of $X$ is scattered.
%1107 We give now a slightly different formulation in the language of tame and NP spaces.%% 
We give now another generalization of Fact \ref{fact:scattered_compact_spaces} 
involving tame and NP lcs.

%0605	I think that the proof utilizes the same core ideas in both cases, but in different framing.
%We show a different proof using the language of Tame locally convex space.
%1606	I don't think that our version is stronger
%2206(Important) The remark from 1606

%0507 prop instead of cor
\begin{prop} \label{p:C(X)} For every topological space $X$ 
	the following are equivalent:
	\ben
	\item $C_{k}(X)$ is a tame lcs.
	%0601
	\item $C_{k}(X)$ is $\NP$.
	\item Every compact subset of $X$ is scattered.
	\een	
\end{prop}
\begin{proof}
	%2306 it would be better to add the arrows indicationg the steps of the proof 
	%1705	Added citation
	For every compact $K \subseteq X$, write $E_{K}:=C(K)$ for 
	%0202q you mean C(K) ?  
	%1102 yes, I thought that it gives a more "reverse limit" feeling. Should I simply write C(K)?
	%1302q I will think 
	the Banach space of continuous functions over $K$.
	Let $D$ be the directed family of all compact subsets in $X$.
	%0806 add citation
	%1606 added
	By \cite[p.~70, Proposition 3]{Jarchow}, $E = C_{k}(X)$ can be embedded in the inverse limit $\varprojlim\limits_{K \in D} E_{K}$.
	%1910
	To make notation easier, we will assume without loss of generality that $E \subseteq \varprojlim\limits_{K \in D} E_{K}$.
	%1606

	%3006 Added arrows
	$(3) \Rightarrow (2):$ Suppose that every compact $K \subseteq X$ is scattered.
	We will show that $C_{k}(X)$ is $\NP$.
	Applying Fact \ref{fact:scattered_compact_spaces}, we know that $E_{K}$ is $\NP$ if and only if $K$ is scattered.
	Using Theorem \ref{thm:properties_of_tame_class}, we conclude that
	%2012 fixed missing word
	$\varprojlim\limits_{K \in D} E_{K}$ is also $\NP$, and therefore so is $E = C_{k}(X) \subseteq \varprojlim\limits_{K \in D} E_{K}$.
	
	%3006: Added for completeness
	$(2) \Rightarrow (1):$ Obvious.
	
	$(1) \Rightarrow (3):$ We will show that if $E$ is tame, then every compact subset of $X$ is scattered.
	Let $K \subseteq X$ be a compact subset.
	Consider the restriction map $r\colon C_{k}(X) \to C(K)$.	
	It is easy to see that $r$ is continuous. 
	We claim that it is 
	%0806 surjective and 
	bound covering.
	Then, applying Theorem \ref{thm:properties_of_tame_class} we will conclude that $C(K)$ is also tame.
	Finally, we will use Fact \ref{fact:scattered_compact_spaces} to show that $K$ is indeed scattered.
	
	Now, to see that $r$ is bound covering, consider the Stone–\u{C}ech compactification $\beta X$.
	%2208 This is a compact Hausdorff space and is therefore normal.
	Suppose that $f \in C(K)$.
	Since $K$ is compact, $f$ is bounded.
	Also, $K$ is closed in $\beta X$ as a compact subset in a compact Hausdorff space.
	We can therefore apply the Tietze extension theorem to find $\widehat{f} \in C(\beta X)$ such that $\widehat{f}_{\mid K} = f$ and
	$$
	  \sup\limits_{x \in \beta X} \lvert\widehat{f}(x)\rvert =
	  \sup\limits_{x \in K} \lvert f(x)\rvert.
	$$
	%0607	Revised proof
	%2208q as I understand you do not require that $e \colon C(K) \to C_{k}(X)$ is linear. Right ? 
	%0109a what do you mean? it is linear.
	%0609q  we must discuss this !
	%0609q Something looks suspected. Namely, about
%	$e \colon C(K) \to C_{k}(X)$ below.   
%	I think the existence of LINEAR CONTINUOUS extension is much stronger than Tietze thm. 
%	What you state it looks as  a kind of the so-called Dugundji extension thm. See the file I sent you. 
%	Maybe you meant something different (I hope so)  ? 
	%0909a I don't think there is a problem here. 
	%0909a the map e is the restriction map, which is always linear. The Tietze is used only to show that it is surjective and bound covering.
	%1009 I think we spend time about a thing where we both agree !  You probably meant each time  the map "r" and not "e". Of course, the restriction map r is linear but look at my original comment marked %2209q, I ask about the map e. 
	Let us write $e \colon C(K) \to C_{k}(X)$ for the map sending every $f \in C(K)$ to $\widehat{f}_{\mid X}$.
	Clearly, $r(e(f)) = f$ and therefore $r(e(B)) = B$ for every $B \subseteq C(K)$.
	%0806 so $r$ is surjective.
	Moreover, for every $f \in C(K)$ and compact $K' \subseteq X$, we have
	$$
	  \sup\limits_{x \in K'} \lvert (e(f))(x) \rvert \leq
	  \sup\limits_{x \in \beta X} \lvert \widehat{f}(x) \rvert =
%	  \sup\limits_{x \in K} \lvert f(x) \rvert = %fixed \beta K to K
	  \lVert f \rVert.
	$$
	As a consequence, if $B \subseteq C(K)$ is bounded by $M > 0$, then so is $e(B)$.
	By definition, $r$ is bound covering.
%	if $B \subseteq C(K)$ is bounded, then by choosing such extensions 
%	$$
%	  \widehat{B} := \{\widehat{f} \mid f \in B\} \subseteq C(\beta X)
%	$$
%	we get a bounded $\widehat{B}\subseteq C(\beta X)$ 
%	%0806 bounded in what lcs ?? It seems that here the argument should be clarified ... 
%	%2206 fixed
%	such that $r(\widehat{B}) = B$.
%	By definition, 
%	%2306  not understood !  Maybe better to discuss this 
%	$r$ is also bound covering.
\end{proof}

%0607	New
%0607	Discussion: I think that the space C_{k}(X) is always locally complete making our formulation always equivalent to that of Gabriyelyan..., however, I am not sure.
%Later on, we will show that a locally complete space is tame if and only if $l^{1}$ can't be embedded inside it (Theorem \ref{thm:tame_iff_R} and Lemma \ref{lemma:overline_R2_equivalence}).
%However, we currently not know if an $\NP$ space can be characterized by its Banach subspaces.

%1606 Moved here from before the previous corollary
%2208 I suggest to remove the following line
%Note that every locally convex space can be embedded in $C_{k}(X)$ for some (locally compact) $X$ \cite[page 158]{Jarchow}. 

%1407 
\begin{remark}
A Banach space $V$ is Asplund iff the dual of every separable subspace is separable. In contrast, there exists a separable NP space $E$ with nonseparable dual (Remark \ref{r:SepNP} below). However, it is unclear for us if a lcs is NP iff the dual of every separable \textit{Banach} subspace is separable. This question is interesting also in order to compare our Proposition \ref{p:C(X)} and \cite[Lemma 6.3]{Networks}. 
%1802 open punctioation
The ``only if part" is clear because the class $\NP$ is closed under subspaces.
\end{remark}

%1606 I think that it is relevant to the reverse limit theorem
\begin{remark} \label{r:SepNP} 
	Let $V$ be a $\NP$ lcs. In contrast, to the case of Banach spaces (recall that $\NP$ Banach spaces are exactly Asplund Banach spaces),  
	the dual of a separable linear subspace $E$ of $V$ is not necessarily separable. Indeed, consider the power space $V=\R^{\mathfrak{c}}$. %2208 This lcs is $\NP$ (because $\NP$ is productive). Moreover, 
	Then $\R^{\mathfrak{c}}$ is a reflexive lcs (in particular, $\NP$). $\R^{\mathfrak{c}}$ is 
	separable (because of the %nontrivial 
	Pondiceri Theorem \cite[Thm.~16.4.c]{Willard}).  
	%2208 $\mathfrak{c}$-power of any separable space is separable) but 
	However its dual is not separable. Indeed, by \cite{Schaefer}, its dual $V^*$ in its weak-star topology can be identified with the locally convex direct sum $\oplus_{i \in I} \R_i$ of continuum many copies of $\R$. It is easy to see that $V^*$ in its weak-star topology is not separable. Therefore, $V^*$ in its strong topology cannot be separable. 
	
	%2212 
	%	 Indeed, its dual $V^*=\oplus_{i \in I} \R_i$ is the locally convex direct sum of continuum many copies of $\R$... Let $A$ be a countable subset of $V^*$. Then $A$ is s subset of some countably many "summands" of $\oplus_{i \in I} \R_i$. One may represent $V^*$ as a locally convex direct sum
	%	$W_1 \oplus W_2$, where $ \oplus_{j \in } \R_j$ 
\end{remark}

%2707	New

%0808 I think the  same is true replacing "tame" by "Asplund" or "DLP". I suggest to think about more convenient proof 
%which naturally will be transparent for all three cases. As I guess this can be done using the following: 
%the closure of $B_W$ in $V$ is $B_V$. 
%This observation can be reformulated by saying that every dense subspace in a normed space is LARGE 
%(say that a lc subspace W of a lcs V is "large in V" if every bounded subset in V is contained in the closure of a bounded subset from W). 
%For our purposes it is important to have some sufficient conditions when a lcs is large in its completion. Found some such conditions in the book "Barreled spaces". Among others: if $E$ is a metrizable lcs ... 
% 

\sk 
%0812 \section{$l^{1}$-sequences}
\section{Generalized $l^{1}$-sequences} 
\label{section:generalized_l1_sequence}
%2903	Added from a previous version of the project
\begin{defin} \label{defin:equivalent_to_usual_l1}
	Let $E$ be a locally convex space.
	A bounded sequence $\{x_{n}\}_{n \in \N} \subseteq E$ is said to be 
	%1009q I am not sure that it is fair to say: "USUAL $l^{1}$-sequence". I would prefer sometimes to forget: "usual" ...
	\emph{equivalent to the $l^{1}$-basis} 
	%2707: Added shorter notation
	(or simply an $l^{1}$-\textit{sequence}) if there exist: a continuous seminorm $\rho$ on $E$ and $\delta > 0$, such that for every $c_{1}, \dotsc, c_{n} \in \R$  %1805:
	$$
	\delta \sum\limits_{i=1}^{n} \lvert c_{i} \rvert \leq
	\rho\left( \sum\limits_{i=1}^{n} c_{i} x_{i} \right).
	$$

	Generalizing %As in 
	Definition \ref{d:AllDef1}, 
	we say that a subset $A$ in a locally convex space $E$ satisfies $\Rl$  if it has no bounded $l^{1}$-sequences.
	%2206(Important)	Added
	The space $E$ satisfies $\Rl$ 
	%2306 if
	iff 
	every bounded subset satisfies $\Rl$.
	If there is no embedding of $l^{1}$ into $E$, then we will say that $E$ satisfies $\ORl$.	
\end{defin}

\sk 
%2903	Added from a previous version of the project
\begin{lemma} \label{lemma:open_linear_map_radious}
	Let $(X,\lVert \cdot \rVert)$ be a Banach space and let $Y$ be a locally convex space.
	%0607	Added "continuous"
	Suppose that $\varphi\colon X \to Y$ is a continuous linear
	%0806 do you require the continuity ? 
	%1306 No, although it is not important.
	map.
	%0607	Changed to embedding
	Then $\varphi$ is an embedding if and only if there is a continuous seminorm $\rho$ on $Y$ and $\delta > 0$ such that $
	\rho(\varphi(x)) \geq \delta \lVert x \rVert 
	$ for every $x \in X$. 
	
\end{lemma}
\begin{proof}
	Suppose that this condition is indeed satisfied.
	Clearly $\varphi$ is injective, because 
	for every $0 \neq x \in X$ we have 
	$
	\rho(\varphi(x)) \geq \delta \lVert x \rVert > 0.
	$
	 Hence, $\varphi(x) \neq 0$.  %090122 and $\varphi$ is injective. 
	
	Since $\varphi$ is linear, it is enough to show that 
	%it is open at $0 \in X$. Moreover, it is enough to show that 
	$\varphi(B_{X})$ is 
	%0806 an open 
	a neighborhood of $0$ 
	%0806 
	in $\varphi(X)$.
	We claim that 
	$$
	\varphi(X) \cap B_{\rho}(\delta) \subseteq \varphi(B_{X}).
	$$
	Indeed, if $y \in \varphi(X) \cap B_{\rho}(\delta)$ then there is $x \in X$ such that $\varphi(x) = y$.
	Note that 
	$$
	\delta \geq 
	\rho(y) = 
	\rho(\varphi(x)) \geq \delta \lVert x \rVert,
	$$
	and therefore $\lVert x \rVert \leq 1$. By definition, $x \in B_{X}$ and $y \in \varphi(B_{X})$.
	
	\sk 
	Conversely, suppose that $\varphi$ is an embedding.
	By definition, $\varphi(B_{X})$ is 
	%0806 an open 
	a neighborhood of $0$ 
	%0806 
	in $\varphi(X)$.
	So there are continuous seminorms $\rho_{1}, \dotsc, \rho_{n}$ on $Y$  and $\delta' > 0$ such that
	$$
	\varphi(X) \cap B_{\rho_{1}, \dotsc, \rho_{n}}(\delta') \subseteq
	\varphi(B_{X}),
	$$
	where
	$$
	B_{\rho_{1}, \dotsc, \rho_{n}}(\delta') := 
	\{y \in Y \mid \forall 1 \leq i \leq n\colon \rho_{i}(y) \leq \delta' \} .
	$$
	First, define $\rho := \max\limits_{1 \leq i \leq n} \rho_{i}$ which is clearly another continuous seminorm on 
	%090122 $X$, 
	$Y$, 
	and note that 
	$$
	B_{\rho_{1}, \dotsc, \rho_{n}}(\delta') = B_{\rho}(\delta').
	$$
	Next, suppose that $x \in X$.
	If $x = 0$ then our condition is clearly satisfied.
	Otherwise, $\lVert x \rVert > 0$.
	Write $\widetilde{x} := \frac{2}{\lVert x \rVert} x$
	and note that $\left\rVert \widetilde{x} \right\rVert =2> 1$.
	Thus, $\widetilde{x} \notin B_{X}$.
	%2206(Important) Made proof clearer
	We claim that $\varphi(\widetilde{x}) \notin B_{\rho}(\delta')$ as a consequence.
	Indeed, suppose by contradiction that $\varphi(\widetilde{x}) \in B_{\rho}(\delta')$.
	Thus, 
	$$
	\varphi(\widetilde{x}) \in 
	B_{\rho}(\delta') \cap \varphi(X) \subseteq \varphi(B_{X}).
	$$
	By definition, there exists some $y \in B_{X}$ such that $\varphi(y) = \varphi(\widetilde{x})$.
	However, $\varphi$ is injective so 
	${\widetilde{x} = y \in B_{X}}$, which is impossible since $\widetilde{x} \notin B_{X}$.
	This contradiction shows that 
	%030422 changed delta to delta' by Saak suggestion
	$\varphi(\widetilde{x}) \notin B_{\rho}(\delta')$.
	By definition, $\rho(\varphi(\widetilde{x})) \geq \delta'$ and therefore
	$$
	\rho\left(\varphi\left( \frac{2}{\lVert x \rVert} x \right)\right) = 
	\frac{2}{\lVert x \rVert} \rho(\varphi(x)) \geq \delta'
	$$
	$$
	\Downarrow
	$$
	$$
	\rho(\varphi(x)) \geq \lVert x \rVert \frac{\delta'}{2}.
	$$
	Finally, define $\delta := \frac{\delta'}{2}$.
\end{proof}

%1805 changing the order in the formulation 
%1705	stated separately now
%0806 can we say simply : "Using Lemma \ref{lemma:open_linear_map_radious} we get" (removing the proof) 
%1306 Applying the change that you suggested
%2707: Changed the order of lemma:overline_R_implies_R and lemma:open_linear_map_radious
The following is a direct consequence of Lemma
\ref{lemma:open_linear_map_radious}.
\begin{lemma} \label{lemma:overline_R_implies_R} 
	$\Rl \Rightarrow \ORl$. 
	
	\noindent More precisely, 
	let $\varphi \colon l^{1} \to X$ be an embedding into a lcs $X$.
	Then 
	$\{\varphi(e_n)\}_{n \in \N} \subseteq X$ is an $l^{1}$-sequence. 
\end{lemma}
%\begin{proof}
%	Suppose that $\varphi\colon l^{1} \to X$ is an embedding and write $\varphi(e_{n}) := x_{n}$ for every $n \in \N$.
%	Note that this sequence is bounded since $\{x_{n}\}_{n \in \N} \subseteq \varphi(B_{l^{1}})$.
%	Using Lemma \ref{lemma:open_linear_map_radious}, we can find a continuous seminorm $\rho$ on $X$ and $\delta > 0$ such that for every $\overline{\alpha} \in l^{1}$:
%	$$
%	\rho(\varphi(\overline{\alpha})) \geq \delta \lVert \overline{\alpha} \rVert.
%	$$
%	In particular, for every $\alpha_{1}, \dotsc, \alpha_{n} \in \R$
%	$$
%	\delta \sum\limits_{i=1}^{n} \lvert \alpha_{i} \rvert =
%	\delta \lVert \overline{\alpha} \rVert \leq
%	\rho(\varphi(\overline{\alpha})) =
%	\rho\left( \sum\limits_{i=1}^{n} \alpha_{i} x_{i} \right).
%	$$ 
%\end{proof}

%1705	new
%0506	Completed proof
\begin{ex} \label{ex:RR}
	The converse to Lemma \ref{lemma:overline_R_implies_R}  need not be true,  $\ORl \nRightarrow \Rl$.
\end{ex}
\begin{proof}
	Indeed consider the normed subspace $X$ of $l^{1}$ consisting of finitely supported sequences:
	%1805 replacing \overline{\alpha} by \alpha
	$$
	X := \left\{\alpha \in l^{1} \mid \exists n_{0} \in \N \ \forall n\geq n_{0}:  \alpha_{n} = 0 \right\}.
	$$
	It is easy to see that the standard basis $\{e_{n}\} \subseteq X$ remains an $l^{1}$-sequence.
	However, there is no embedding $\varphi\colon l^{1} \to X$. 
	%0806 
	%	such that $\varphi(e_{n}) = e_{n}$.
	%	Moreover, we will show that $l^{1}$ can't be embedded in $X$ at all.
	By contradiction, suppose that $\varphi \colon l^{1} \to X$ is an embedding.
	By $\dim X$, we mean the Hamel dimension of $X$.
	It is easy to see that $\dim X = \aleph_{0}$ while ${\dim Y = \dim l^{1} = \aleph}$, a contradiction.
\end{proof}

%1705	New for Theorem \ref{thm:tame_iff_R}
%0506	Combined with a once separated corollary
%0806   let's discuss 
%1606	Removed wrong part and changed formulation to being positive
%2206(Important) the %1606 change
\begin{lemma} \label{lemma:preimage_of_l1_sequence}
	Suppose that $X$ and $Y$ are locally convex spaces and $\varphi \colon X \to Y$ is a continuous linear map.
	%If $B \subseteq X$ is bounded and $\varphi(B)$ satisfies $\Rl$, then also $B$ satisfies $\Rl$.
	If $A \subseteq X$ is bounded and satisfies $\Rl$, then so does $\varphi(A)$.
\end{lemma}
\begin{proof}
	By contradiction, suppose that $\varphi(A) \subseteq Y$ does not satisfy $\Rl$.
	By definition, there exist: a sequence 
	${\{y_{n}\}_{n \in \N} \subseteq \varphi(A) \subseteq Y}$, a continuous seminorm $\rho$ on $Y$ and $\delta > 0$ such that for every $c_{1}, \dotsc, c_{n}\in \R$ 
	$$
	\delta\sum_{i=1}^{n} \lvert c_{i} \rvert \leq 
	\rho\left(\sum_{i=1}^{n} c_{i} y_{i} \right) .
	$$
	Since $\varphi$ is continuous and linear, there is a continuous seminorm $\sigma$ on $X$ such that 
	%090122 I think one may take caninically \sigma(x):=\rho(\varphi(x)) 
	$$
	\rho(\varphi(x)) \leq \sigma(x)
	$$
	for every $x \in X$. 
	Choose $x_{n} \in A \subseteq X$ such that $\varphi(x_{n}) = y_{n}$. 
	%0806 
	The sequence $\{x_{n}\}$ is bounded because $A$ is bounded. 
	We claim that $\{x_{n}\} \subseteq A$ is an $l^{1}$-sequence with respect to $\sigma$ and $\delta$, 
	%1606
	which contradicts $A$ satisfying $\Rl$.
	Indeed, for every $c_{1},\dotsc, c_{n} \in \R$ 
	%0607	Fixed alignment
	\begin{align*}
		\sigma\left(\sum_{i=1}^{n} c_{i} x_{i} \right) & \geq 
		\rho\left(\varphi\left(\sum_{i=1}^{n} c_{i} x_{i} \right)\right) = 
		\rho\left(\sum_{i=1}^{n} c_{i} \varphi(x_{i}) \right)  \\
		& = \rho\left(\sum_{i=1}^{n} c_{i} y_{i} \right) \geq
		\delta\sum_{i=1}^{n} \lvert c_{i} \rvert.
	\end{align*}
	
\end{proof}

%090122q  why here you need to recall completeness properties?
%2911 The definition moved to the second section.

%150122 Removed lemma about gauge pre-image

%2903	Added to support locally complete spaces
The following lemma can be extrapolated from \cite[Propositions 3.1 and 3.3]{Ruess}.
%We show it here separately to make further proofs easier to read.
\begin{lemma} \label{lemma:equivalen_to_the_usual_l1_basis}
	Let $X$ be a locally complete lcs.
	If $\{x_{n}\}_{n \in \N} \subseteq X$ is a bounded $l^{1}$-sequence, then there is a linear topological embedding $\varphi \colon  l^{1} \to X$ such that for every $n \in \N$,  $\varphi(e_{n}) = x_{n}$.
	
	%2409	Added to support the localization of Haydon's theorem.
	Moreover, if $\rho$ is a continuous seminorm on $X$ and $\delta > 0$ make $\{x_{n}\}_{n \in \N}$ into an $l^{1}$-sequence, then:
	$$
	  \rho(\varphi(\alpha)) \geq \delta \lVert \alpha \rVert
	$$
	for every $\alpha \in l^{1}$.
\end{lemma}

\begin{proof}
	Let $\{x_{n}\}_{n \in \N}$ be equivalent to the $l^{1}$-basis.
	Define ${\varphi\colon l^{1} \to X}$ by
	%030422	Changed \overline{\alpha} to \alpha
	$
	\varphi(\alpha) :=
	\sum\limits_{n \in \N} \alpha_{n} x_{n},
	$  
	%030422	Added by Saak suggestions
	where $\alpha = (\alpha_{n})_{n \in \N} \in l^{1}$.
	Before we continue, we verify the convergence of this sum.
	Write $B$ for the 
	%0202 adding 
	smallest 
	closed disc containing $\{x_{n}\}_{n \in \N}$.
	Since $\{x_{n}\}_{n \in \N}$ is bounded, so is $B$. 
	%0202 By definition,
	By our assumption, $X$ is locally complete. Hence, 
	 $(X_{B}, q_{B})$ is complete.
	Moreover, $\sum\limits_{n=0}^{N} \alpha_{n} x_{n}  \in X_{B}$, so we can use the Cauchy criterion for convergence.
	Let $\eps > 0$.
	By the definition of $l^{1}$, there exists some $n_{0} \in \N$ such that $\sum\limits_{n=n_{0}}^{\infty} \lvert \alpha_{n} \rvert \leq \eps$.
	Note that for every $n \in \N$, $q_{B}(x_{n}) \leq 1$ because $x_{n} \in 1 \cdot B$.
	Thus, for every $m \geq n_{0}$:
	$$
	q_{B}\left(\sum\limits_{n=n_{0}}^{m} \alpha_{n} x_{n} \right) \leq
	\sum\limits_{n=n_{0}}^{m} \lvert \alpha_{n} \rvert q_{B}(x_{n}) \leq
	\sum\limits_{n=n_{0}}^{m} \lvert \alpha_{n} \rvert \leq 
	\sum\limits_{n=n_{0}}^{\infty} \lvert \alpha_{n} \rvert
	\leq \eps.
	$$
	We have therefore shown that $\sum\limits_{n \in \N} \alpha_{n} x_{n} \xrightarrow{q_{B}} x \in X_{B}$.
	Now, apply Fact \ref{f:local_disc_norm_is_finer} to conclude that $\lim\limits_{n \in \N} \sum\limits_{n \in \N} \alpha_{n} x_{n} = x \in X$ with respect to the original topology of $X$.
	
	Note that $\varphi$ is clearly linear.	
	Next we show that $\varphi$ is continuous.
	Let $\rho$ be a continuous seminorm on $X$, $\eps > 0$ and  $M \in \R$ be a bound for $\{\rho(x_{n}) \}_{n \in \N}$.
	A similar calculation shows that 
	$$
	\rho(\varphi(\alpha)) \leq
	M \sum\limits_{n \in \N} \lvert \alpha_{n} \rvert =
	M \lVert \overline{\alpha} \rVert \leq
	M \frac{\eps}{M} = \eps
	$$
	for every $\alpha \in \frac{\eps}{M}B_{l^{1}}$. 
	Note that 
	for every $\alpha \in l^{1}$ such that $\alpha_{n} = 0$ for every $n > n_{0}$, we have 
	$$
	\rho(\varphi(\alpha)) = 
	\rho\left( \sum\limits_{i=1}^{n_{0}} \alpha_{i} x_{i} \right) \geq
	\delta \sum\limits_{i=1}^{n_{0}} \lvert \alpha_{i} \rvert =
	\delta \lVert \alpha \rVert.
	$$
	Since the finitely supported elements in $l^{1}$ are dense, and by virtue of the continuity of $\varphi$, this inequality also applies for every $\alpha \in l^{1}$.
	Applying Lemma \ref{lemma:open_linear_map_radious}, we conclude that $\varphi$
	%0607	Changed to "embedding"
	is indeed an embedding.
\end{proof}

%0202q   
Recall that we gave definitions of $\rho_M$ and $\coRl$ in Definition \ref{d:rho_M}. 
%0605 Better to say first that one direction is always true.
%0704	new
The following is a direct consequence of Lemma \ref{lemma:overline_R_implies_R} and Lemma \ref{lemma:equivalen_to_the_usual_l1_basis}.

\begin{lemma} \label{lemma:overline_R2_equivalence}
	For locally complete lcs, the conditions $\Rl$ and $\ORl$ are equivalent.
\end{lemma}

%2409	New for the localization of Haydon's Theorem.
\begin{lem} \label{lemma:coRl_l1_embedding}
	Let $E$ be a locally convex space and let $M \subseteq E^{*}$ be an equicontinuous, weak-star compact, disked subset.
	If $M$ does not satisfy $\coRl$, then there exist an embedding $T\colon V \to E$ where $V$ is a dense normed subspace of $l^{1}$ and $\delta > 0$ such that 
	$$
	\delta B_{V^{*}} \subseteq T^{*}(M).
	$$
\end{lem}
\begin{proof}
	Let $\widehat{E}$ be the completion of $E$.
	By definition, there is a bounded sequence $\{x_{n}\}_{n \in \N} \subseteq E$ which is equivalent to the usual $l^{1}$-basis with respect to $\rho_{M}$.
	By Lemma \ref{lemma:equivalen_to_the_usual_l1_basis}, there exists an embedding $T': l^{1} \to \widehat{E}$ with respect to $\rho_{M}$ such that $T'(e_{m}) = x_{m}$.
	Moreover, there exists $\delta > 0$ such that
	$$
	\forall \alpha \in l^{1}: \rho_{M}(T(\alpha)) \geq \delta \lVert \alpha \rVert_{1}.
	$$
	Write $V := (T')^{-1}(E)$ and $T := T'_{\mid V}$.
	
	By contradiction, suppose that $\delta B_{V^{*}} \nsubseteq T^{*}(M)$.
	In this case, there exists $\varphi \in \delta B_{V^{*}} \setminus T^{*}(M)$.
	Recall that 
	%0202 
	by Banach-Grothendieck theorem 
	the dual of $(V^{*},w^*)$ %with the weak-star topology 
	is simply $V$.
	%1012	Simplified with \ref{fact:absolutely_convex_separation}
	Also, $T^{*}(M)$ is clearly absolutely convex and closed.
	Using Fact \ref{fact:absolutely_convex_separation}, we can find $\alpha \in V$ such that
	$$
	\sup_{\theta \in M} \lvert (T^{*}(\theta))(\alpha) \rvert < 
	\lvert \varphi(\alpha) \rvert.
	$$
	Thus we get:
	\begin{align*}
		\delta \lVert \alpha \rVert_{1} \leq 
		&\rho_{M}(T(\alpha)) =
		\sup_{\theta \in M} \lvert \theta(T(\alpha)) \rvert \\
		= & \sup_{\theta \in M} \lvert (T^{*}(\theta))(\alpha) \rvert <
		\lvert \varphi(\alpha) \rvert \\
		\leq & \lVert \varphi \rVert_{1} \lVert \alpha \rVert_{1} \leq
		\delta \lVert \alpha \rVert_{1},
	\end{align*}
	a contradiction.
\end{proof}

\sk 
%0812 \section{Rosenthal type properties in lcs} 
\section{Rosenthal type properties} 
\label{s:RosType} 

%Consequently a Banach space $V$ does not contain an
%$l_1$-sequence if and only if every bounded sequence
%in $V$ has a weak-Cauchy subsequence. 

%0808 
For the definitions of $\Rl$ and $\ORl$, see Definitions \ref{d:AllDef1} and \ref{defin:equivalent_to_usual_l1}. 
%
%\begin{lemma} \label{lemma:quotient_equicontinuous_banach}
%	Let $E$ be a locally convex space and let $M \subseteq E^{*}$ be an equicontinuous subset.
%	Then there exists a seminorm $\rho$ on $X$, a quotient map $\pi: E \to X$ to a Banach space $X$ and a $\rho$-continuous $\Delta: E^{*} \to X^{*}$ such that the following diagram commutes:
%	%0506	Added a commutative diagram
%	% https://tikzcd.yichuanshen.de/#N4Igdg9gJgpgziAXAbVABwnAlgFyxMJZABgBpiBdUkANwEMAbAVxiRAA0QBfU9TXfIRQBGclVqMWbADrSASt14gM2PASJlh4+s1aIQATW7iYUAObwioAGYAnCAFskZEDghJRIBnQBGMBgAK-GpCIFhg2LCKNvZOiC5uSABM1N5+gcGCbOGRrDwxjsnUiYieaf5Bqln6OVhRXBRcQA
%	$$
%	\begin{tikzcd}
%		X \arrow{r}{\varphi} \arrow{d}{\pi} & \R \\
%		X_{\rho} \arrow{ru}{}[swap]{\Delta(\varphi)}          &   
%	\end{tikzcd}
%	$$
%	Moreover, there exists a weak-star continuous 
%	$\Gamma: E^{**} \to X^{**}$ such that the following diagram commutes:
%	$$
%	\begin{tikzcd}
%		M \arrow{r}{x^{**}} \arrow{d}{\Delta} & \R \\
%		\Delta(M) \arrow{ru}{}[swap]{\Gamma(x^{**})}          &   
%	\end{tikzcd}
%	$$
%\end{lemma}
%\begin{proof}
%	
%\end{proof}

\begin{thm} \label{thm:tame_iff_R} 
	$\Tame=\Rl$. 
	
\noindent	More precisely,  
a bounded 
%0806    Can we prove this for all (not necessarily, BOUNDED) subsets ? If YES, then we can formulate:  $\Tame_X=\Rl_X$ for every lcs $X$. 
%1606	It might be possible but we will still need to require the sequences in each definition to be bounded
subset $B \subseteq X$ of a lcs is tame \emph{if and only if} $B$ satisfies $\Rl$. 
\end{thm}
\begin{proof}
	First suppose that $B \subseteq X$ does not satisfy $\Rl$.  We will show that $B$ is not tame.
	Without loss of generality, we can assume that $B = \{x_{n}\}_{n \in \N} \subseteq X$ is a bounded sequence equivalent to the $l^{1}$-basis. 
	Write $\widehat{X}$ for the completion of $X$. 
	%0908 Made clearer by your suggestion
	It is easy to see that $\{x_{n}\}_{n \in \N}$ remains an $l^{1}$-sequence as a subset of $\widehat{X}$ (directly applying Definition \ref{defin:equivalent_to_usual_l1} and extending the seminorm to $\widehat{X}$).
	%0808 probably better to say that this sequence is equivalent to the usual $l^{1}$-basis also as a subsequence in the completion (in the sense of Definition \ref{defin:equivalent_to_usual_l1}) 
	%1109	Made the proof easier.
	By Lemma \ref{lemma:equivalen_to_the_usual_l1_basis}, there is an embedding $\varphi \colon l^{1} \to \widehat{X}$ such that $\varphi(e_{n}) = x_{n} \in X$.
	Write $B' := \{e_{n}\}_{n \in \N} \subseteq l^{1}$.
	Clearly, $B'$ is not tame in $l^1$.
	Since $\varphi$ is an embedding, $B = \varphi(B')$ is also not tame 
	%2609 what is $Y$ ? 
	%1910 Fixed to \widehat{X}.
	in $\widehat{X}$.
	However, by Theorem~\ref{thm:properties_of_tame_class}, this implies that $B$ is not tame in $X$, which is a contradiction.

	Conversely, suppose that $B \subseteq X$ satisfies $\Rl$, we will show that $B$ is tame.
	%1109	Seperated some of the proofs content to a separate lemma.
	%0202 Let $M \subseteq X^{*}$ be an equicontinuous subset, 
	 Let $M \in \eqc(X^{*})$
	and let $\pi\colon X \to V$ and $\Delta\colon M \to V^{*}$ be as in Lemma \ref{lemma:equicontinuous_factor}.
	Since $\Delta$ is weak-star continuous, $\Delta(M)$ is weak-star compact and therefore equicontinuous by the Banach-Steinhaus Theorem.
	Since $B$ satisfies $\Rl$ and considering Lemma \ref{lemma:preimage_of_l1_sequence}, so does $\pi(B)$.
	By Lemma \ref{lemma:specific_l1_when_not_tame}, $\pi(B)$ is tame over $\Delta(M)$.
	%150122	Made simpler with Lemma \ref{lemma:b_small_and_adjoint_maps}
	Applying 
	%0202q hard to understand (let's discuss) 
	Lemma \ref{lemma:b_small_and_adjoint_maps}, we conclude that $B$ is tame over $M = \pi^{*}(\Delta(M))$.
\end{proof}

%0909	Added by your suggestion
%1009  I think it is really good to have a new dichotomy which works with ALL lcs. This deserves much more attention and popularization ! At least into introduction and also to grant this corollary a special name. In fact, you already did this writing into the label: "tame_dichotomy" ! Let's take it as a tentative name. Let me know if you have something even better. 

%
%The following corollary of Theorem \ref{thm:tame_iff_R} is one of our main results. Unexpectedly enough it gives a generalization of Rosenthals's dichotomy to \textit{all} locally convex spaces in terms of the tameness. 

The following corollary of Theorem \ref{thm:tame_iff_R} gives a generalization of Rosenthals's dichotomy to all locally convex spaces in terms of the tameness. 
%1009q  After some rethinking the following seem sto be quite expected and even with a direct straightforward proof. Let's discuss

\begin{thm} \label{t:tame_dichotomy} 
	%1009 
	\emph{\textbf{[Tame dichotomy in lcs]}} 
	Let $E$ be a locally convex space.
	Then every bounded 
	subset in $E$ is either tame, or has a subsequence equivalent to the $l^1$-sequence. %1009  usual $l^{1}$-basis.
%1009q Matan, what if we write:	Then every bounded sequence in $E$ either has a tame subsequence or an $l^1$-sequence.
%Is this formally weaker "sequential version" really weaker ? Equivalent ? 
%1009q you already answered me in email about the previous query. So, I ask similar but different question. 
%Q: Is it true that every bounded sequence in $E$ either contains a fragmented subsequence or a a subsequence equivalent to the $l^1$-sequence.

%1009q every weak Cauchy sequence ia fragmented, right ? 
\end{thm}

%0808 new place 
\begin{defin} \label{d:AllDef2}
	A subset $A \subseteq X$ of a lcs is said to be \textit{Rosenthal}  
	(write $A \in \Ros$) 
	if every bounded sequence in $A$ has a weak Cauchy subsequence.
	We say that $X$ is \textit{Rosenthal} ($X \in \Ros$) if every bounded subset of $X$ is Rosenthal.
	%	Let $A \subseteq X$ be a subset in a lcs $X$. Define the following properties of $A$: 
	%	
	%	\ben
	%		\item[$\Ros$] Every bounded sequence in $A$ has a subsequence which is weakly Cauchy.  
	%		\item[$\Rl$] There is no bounded sequence in $A$ which is equivalent to the usual $l^{1}$-basis. 
	%		\item[$\ORl$] The topological linear space $l^{1}$ cannot be embedded into the linear span of $A$.
	%	\een
	%	
	%	
	%	\sk 
	%	%\item [\label=Tame-LCS] Definition \ref{def:tame_lcs}. \label{def:R:tame}
	%	We say that $X$ is $\Ros$ or, \textit{Rosenthal} if every
	%	%0806 maybe we can delete here BOUNDED 
	%	%1606 But it is necessary.
	%	 bounded subset is $\Ros$. 
	%	Similarly can be defined the lcs with $\Rl$ and $\ORl$. 
	%	\footnote{What about the notation (for subsets) $A \in \Ros_X$ ??{\color{blue} Seems good, maybe not necessary}}
\end{defin}

%0605 it seems that only one direction here is true ... 
%0105	Changed to apply for subsets and added the other direction
%2208 you wrote 'other direction" (a line above). What do you mean ? 
%0109 no idea
\begin{prop} \label{prop:rosenthal_is_tame}
	$\Ros \Longrightarrow \sTame$. 
	
	\noindent 
	More precisely, every bounded Rosenthal subset $B \subseteq X$ 
	%1009 
	(e.g., every bounded weak Cauchy subsequence) 
	%%
	%1511	Changed to mackey. 
	in a lcs $X$ is Mackey tame.
 %(Definition \ref{d:AllDef}) 
\end{prop}
\begin{proof}
	%2208q  Question: is it true that every bounded weak Cauchy sequence in every lcs is (stronly) fragmented ? 
	%If, YES then this gives sharp and more conceptual proof of this proposition. 
	%By the way, this proposition even without any additional work implies that any weak Cauchy CONTAINS a (strongly) fragmented subsequence. So, my question is if already any bounded weak Cauchy sequence itself is (strongly) fragmented ?
	%rereading my comment, I think that we need to clarify more carefully the relationship between these two concepts: 
	%a) weak Cauchy sequence b) (strongly) fragmented sequence. If the interrelation is not too rigid maybe this hints that there exists a chance for a modified general dichotomy in lcs with (strongly) fragmented sequences
	%0109a Every weak Cauchy sequence is rosenthal and therefore tame (hence eventually fragmented). Does this answer the question? I might have misunderstood.
	%1009 must be discussed !
	Suppose that $B \subseteq X$ is a bounded Rosenthal 
	set and let $M \subseteq X^{*}$ be weak-star compact.
	We claim that $B$ is tame over $M$.
	Assuming the contrary, suppose that $\{x_{n}\}_{n \in \N} \subseteq B$ is independent over $M$. We can assume without loss of generality that $\{x_{n}\}_{n \in \N}$ is both weak-Cauchy and independent over $M$.  
	%(because independence of sequences is hereditary).
	Let $a < b \in \R$ be the bounds of independence of $\{x_{n}\}_{n \in \N}$ over $M$.
	Since $M$ is weak-star compact, we can use Fact \ref{lemma:independence_over_compact_sets} to find $\varphi \in M$ such that:
	$$
	\varphi(x_{n}) \in \begin{cases}
		(b, \infty) & n \in 2\N \\
		(-\infty, a) & \text{otherwise} \\
	\end{cases}.
	$$
	This implies that $\{\varphi(x_{n}) \}_{n \in \N}$ is not a Cauchy sequence, which is a contradiction.
\end{proof}

\begin{thm}\label{t:coinc}
	In every locally convex space $X$, the following holds:
	%0605   
	$$
	\Ros \Longrightarrow \sTame   \Longrightarrow \Tame = \Rl \Longrightarrow \ORl.
	$$
%0507	Also, for locally complete spaces $\Rl = \ORl$.
\end{thm}
\begin{proof}
	%2012	Added a reference to lemma:overline_R_implies_R
	Combine Theorem \ref{thm:tame_iff_R}, Proposition \ref{prop:rosenthal_is_tame} and Lemma \ref{lemma:overline_R_implies_R}.  
	%0507 and Lemma \ref{lemma:overline_R2_equivalence}. 
\end{proof} 

\sk 
%0105	Changed to apply for subsets
%1705	Slight reformulation
\begin{thm} \label{thm:bounded_metrizable_tameness}   
	Let $X$ be a locally convex space whose 
	%2707: Added separable
	%0908: Removed separable
	bounded subsets are metrizable. Then ${\Ros=\Tame}$.
	%0806  $X$ is Rosenthal if and only if it is tame. 
	More precisely, 
	If $B \subseteq X$ is a bounded subset then it is Rosenthal if and only if it is tame.
	
\end{thm}
\begin{proof}
	We already established in Proposition \ref{prop:rosenthal_is_tame} that every Rosenthal bounded subset is tame.
	We will see the converse applied in this setting.
	Let $B \subseteq X$ be tame, we will show that it is also Rosenthal.
	Without loss of generality, we can assume that $B = \{x_{n}\}_{n \in \N} \subseteq X$ is a bounded sequence.
	%2012 made simpler with \acx
	Note that $\acx B$ is also bounded, and therefore metrizable.
	
%	Define ${\widehat{B} := B \cup (-B)}$.
%	Note that $\cone \widehat{B}$ which is defined as 
%	$$
%	\cone \widehat{B} := 
%	\left\{ \sum\limits_{i=1}^{n} \alpha_{i} y_{i} \mid \forall 1 \leq i \leq n: \alpha_{i} \in [0, 1], y_{i} \in \widehat{B} \text{ and } \sum\limits_{i=1}^{n} \alpha_{i} \leq 1 \right\},
%	$$
%	is also bounded.
	%0908: removed explanation
%	This can easily be verified since every continuous seminorm $\rho$ on $X$ satisfies
%	$$
%	\rho \left( \sum\limits_{i=1}^{n} \alpha_{i} y_{i} \right) \leq
%	\sum\limits_{i=1}^{n} \lvert \alpha_{i} \rvert \rho \left( y_{i} \right) \leq	  
%	\max\limits_{y \in \widehat{B}} \rho(y) \sum\limits_{i=1}^{n} \lvert \alpha_{i} \rvert \leq
%	\max\limits_{y \in \widehat{B}} \rho(y).
%	$$
	%2707:	Added separable
	%0908: 	Removed separable
	%Also, $\cone \widehat{B}$ is clearly separable (because $B$ is countable).
%	Since every bounded set is metrizable, so is $\cone \widehat{B}$.
	
	Let $\{\eps_{m} \cap \acx B \}_{m \in \N}$ be a descending basis of disks for the neighborhoods of $0 \in \acx B$, and let $\{\eps_{m}^{\circ}\}_{m \in \N}$ be the corresponding polars.
	Every such polar is equicontinuous by definition and also weak-star compact by the Banach–Alaoglu Theorem.	
	Since $B$ is tame, it must be tame over each $\eps_{m}^{\circ}$.
	
	%0908: Made the induction easier to understand
	Define $n_{k}^{(0)} := k$.
	Using induction, we construct subsequences $\left\{n_{k}^{(m)}\right\}_{k \in \N} \subseteq \N$ such that 
	\ben
		\item $\left\{n_{k}^{(m + 1)} \right\}_{k \in \N} \subseteq \left\{n_{k}^{(m)}\right\}_{k \in \N}$.
		\item $\left\{\varphi\left(x_{n_{k}^{(m)}}\right) \right\}_{k \in \N}$ converges for every $\varphi \in \eps_{m}^{\circ}$.
	\een
	The induction step is done by applying Lemma \ref{f:sub-fr} on $\left\{n_{k}^{(m)}\right\}_{k \in \N}$ to find a subsequence $\left\{n_{k}^{(m + 1)} \right\}_{k \in \N}$ that converges on $\eps_{m + 1}^{\circ}$.
	Now, defining $n_{k} := n_{k}^{(k)}$ we get that $\{\varphi(x_{n_{k}}) \}_{k \in \N}$ converges for every $\varphi \in \bigcup\limits_{m \in \N} \eps_{m}^{\circ}$.
	
	We will now show that $\{x_{n_{k}}\}_{k \in \N}$ is a weak Cauchy sequence.
	Let $\varphi \in X^{*}$ and let $\eps$ be a neighborhood of zero such that $\lvert \varphi(\eps) \rvert \leq 1$.
	By the construction of $\{\eps_{m} \}_{m \in \N}$, there is some $m_{0} \in \N$ such that $\eps_{m_{0}} \cap \acx B \subseteq \eps \cap \acx B$.
	
	%2903	a bit easier to understand now.
	Using Lemma \ref{lemma:disk_hahn_banach}, we can find $\widehat{\varphi} \in \eps_{m_{0}}^{\circ}$ such that $\varphi_{\mid \acx B} = \widehat{\varphi}_{\mid \acx B}$.
	However, ${\{x_{n_{k}}\}_{k \in \N} \subseteq \acx B}$ and therefore $\{\varphi(x_{n_{k}}) \}_{k \in \N} = \{\widehat{\varphi}(x_{n_{k}}) \}_{k \in \N}$ converges.
	By definition, $\{x_{n_{k}}\}_{k \in \N}$ is weak Cauchy.
\end{proof}

%0202	Theorem \ref{thm:bounded_metrizable_tameness} is also applicable in the case of Frechet spaces.  

%0806 
\begin{thm} \label{t:BoundMetr} 
%2707	Added "separable"
%0908	Removed "separable"
If all bounded sets of a lcs $X$ are metrizable, then 
%0806  all of those properties are equivalent and the following dichotomy holds:
$$
\Ros = \sTame   = \Tame = \Rl \Longrightarrow \ORl 
$$
and the following dichotomy holds: 
any bounded sequence in $X$ either has a weak Cauchy subsequence or an $l^1$-subsequence. 
\end{thm}
\begin{proof}
	%2206(Important)
	%2206 	There is a subtle point I would like to discuss here.
	%2206 	 Theorem \ref{t:coinc} which is referenced here does not include in its 
	%2206 	 formulation anything about subsets, however, to proce this theorem it
	%2206 	 is needed. It is probably obvious to anyone familiar with the proofs
	%2206 	 but it still might not be trivial at first sight
	%2306  Let's discuss 
	Combine Theorem \ref{t:coinc} and Theorem \ref{thm:bounded_metrizable_tameness}.
%	
%	Moreover, we see that a bounded subset is Rosenthal if and only if it satisfies $\Rl$.
%	Thus, if $\{x_{n}\}_{n \in \N}$ is a bounded sequence, it is either Rosenthal (meaning that every sequence has a weak Cauchy subsequence), or it is not $\Rl$ (meaning that it has a subsequence equivalent to the usual $l^{1}$-basis).
%	The dichotomy is therefore proven.
\end{proof}

Theorem \ref{t:BoundMetr} and Lemma \ref{lemma:overline_R2_equivalence} directly imply a well-known result of Ruess
%030422 added by Saak suggestion
(Fact \ref{fact:Ruess})
 which extends Rosenthal's non-containtment of $l^1$-criteria   
% Rosenthal's dichotomy 
to a quite large class of lcs. 

%2507 adding Ruess dichotomy (as in the introduction) 
%030422	Removed by Saak suggestion
%\begin{f} \label{fact:Ruess} \emph{Ruess \cite[Thm.~2.1 and Prop.~3.3]{Ruess}}  Let $E$ be a locally complete lcs with
%	%0108q what about "separable" 
%	%0908  we decided to remove this
%	 metrizable bounded sets. Then 
%	$$\Ros = \ORl$$ 
%	%2507 
%	and the following dichotomy holds: 
%	any bounded sequence in $X$ either has a weak Cauchy subsequence or a subsequence which spans an isomorphic copy of $l^1$. 
%\end{f}

%\begin{f} \label{fact:Ruess} Ruess \cite[Thm. 2.1]{Ruess} Let $E$ be a locally complete lcs with metrizable bounded sets. Then 
%	$$\Ros = \ORl.$$ 
%\end{f}

\sk 
%1907 Matan, do you know an example of a lcs E which is not tame but contains a dense tame lc subspace ? The same question for a Banach space (and its normed dense subspace). In particular, the case when something is tame but not the completion ...
%2007 No, I have no examples so far. I have a feeling that l1 has a dense tame subset. So far I couldn't find one but I'm trying to run some computer simulations   %2507q for Banach spaces my question is silly because the completion of a tame normed space is tame and the unit ball of the completion is the closure of the unit ball of the original normed space
\begin{thm} \label{p:1And3AreNotReversible} There exists a 
	%0605 I think by the argument I mentioned in my email we can add here "strongly" ... 
	%1606 You are right
	strongly tame complete (even reflexive) lcs which: 
	\begin{enumerate}
		\item [(i)] is not a Rosenthal lcs;
		%in particular, $\Ros \neq \sTame.$ 
		%does not satisfy R1;  
%		
%		The implications (1) and (3) in Lemma \ref{lemma:R_tameness_relation} are not reversible (even for $\NP$ complete spaces). 
		\item [(ii)] does not contain any $l^1$-subsequence;  %(hence, also any subspace which is isomorphic to $l^1$). 
		%0202q So, we can conclude that the completion of Rosenthal lcs is not necessarily Rosenthal 
		%1102 Yes! However, we still don't know about completion of tame spaces.
		%1302new It already was mentioned in the Introduction after Thm \ref{p:1And3AreNotReversibleINTRO}
		
		\item [(iii)] contains a dense, Rosenthal subspace.
	\end{enumerate}  
As a corollary: 
 Rosenthal's dichotomy 
 %0806 what are these two forms ??
 %1606 I have no idea, removed this bit
 %2206(Important) remark from 1606
 does not hold for such locally convex spaces. 
\end{thm}
\begin{proof} 
	%3006	Changed to \R^{C} instead of \R^[0,1] to avoid problems with non-unique binary representation
	Write $C := \{0, 1\}^{\mathbb{N}}$ and consider the space $X := \R^{C}$ with the product topology.
	It is known that $X$ is complete and reflexive.
	By Theorem \ref{thm:properties_of_tame_class}, $X$ is tame.
	%1606	Added by your suggestion
	Also, every reflexive space is barrelled \cite[Proposition 11.4.2]{Jarchow} so by Proposition  \ref{p:tameness_in_barrelled_spaces}, $X$ is strongly tame.
	We will see now that it is not Rosenthal.
	Consider the $n$-th projection $\pi_{n}\colon C \to \{0, 1\}$.
	We claim that $\{\pi_{n}\}_{n \in \N} \subseteq X$ has no weak Cauchy subsequence.
	Indeed, for every subsequence $\{n_{k}\}_{k \in \N} \subseteq \N$ define $\alpha \in C$ whose $n$-th entry is
	$$
	  \alpha_{n} := \begin{cases}
	  	1 & n = n_{k}, k \in 2\N \\
	  	0 & \text{else}
	  \end{cases}
	$$
	%0908	Added explicit mention of the evaluation functional
	and let $e_{\alpha} \colon X \to \R$ be the evaluation functional at $\alpha$.
	It is easy to see that $\{e_{\alpha}(\pi_{n_{k}})\}_{k \in \N}$ does not converge and therefore $\{\pi_{n_{k}}\}_{k \in \N}$ is not weak Cauchy.
	
	%0607	Changed to finite support
	Finally, consider the dense subspace $Y \subseteq X$ of functions with finite support
	%2206	On hindsight, finite support will probably work just as well.
	%2306 Let's discuss 
	$$
	Y := 
	\{f \in \R^{C} \mid \lvert \support(f) \rvert < \aleph_{0} \} \subseteq \R^{C}.
	$$
	We will show that $Y$ is a Rosenthal space.
	Let $\{x_{n}\}_{n \in \N} \subseteq Y$ be a bounded sequence.
	By definition, 
	$$
	S := \bigcup\limits_{n \in \N} \support(x_{n}) \subseteq C
	$$
	is countable so we can enumerate it as $S = \{\alpha_{m}\}_{m \in \N}$.
	%1403	Made shorter
	Using induction and a diagonal argument, it is easy to find a subsequence $\{n_{k}\}_{k \in \N} \subseteq \N$ such that $\{x_{n_{k}}\}_{k \in \N}$ converges pointwise on $S$.
%	Define $n_{k}^{(0)} := k$.
%	Using induction we construct a sequence $\left\{ \left\{n_{k}^{(m)}\right\}_{k \in \N} \right\}_{m \in \N}$ such that
%	\ben
%	\item For every $m \in \N$: \ 
%	$
%	\left\{ \left\{n_{k}^{(m + 1)}\right\}_{k \in \N} \right\}_{m \in \N} \subseteq 
%	\left\{ \left\{n_{k}^{(m)}\right\}_{k \in \N} \right\}_{m \in \N};
%	$
%	\item $\left\{x_{n_{k}^{(m)}}(\alpha_{m})\right\}_{k \in \N}$ converges.
%	\een
%	%0908	Added more details.
%	Suppose that we completed the induction up to $m \in \N$, we will continue to $m + 1$.
%	Note that $\{x_{n}\}_{n \in \N}$ is bounded and therefore
%	$\left\{x_{n_{k}^{(m)}}(\alpha_{m + 1})\right\}_{k \in \N}$ is bounded.
%	Applying the Bolzano-Weierstrass theorem, we can find a subsequence $\left\{ \left\{n_{k}^{(m + 1)}\right\}_{k \in \N} \right\}_{m \in \N}$ such that $\left\{x_{n_{k}^{(m + 1)}}(\alpha_{m + 1})\right\}_{k \in \N}$ converges.
%	
%	Writing $n_{k} := n_{k}^{(k)}$, we get a subsequence $\{x_{n_{k}}\}_{k \in \N}$ that converges pointwise on $S$.
	%030422 Made simpler by Saak suggestion
	Moreover, for every $\alpha \in C \setminus S, k\in \N$ we have $x_{n_{k}}(\alpha) = 0$.
	Therefore, $\{x_{n_{k}}\}$ converges pointwise on $C$ and is therefore weak-Cauchy.
%	We claim that $\{x_{n_{k}}\}_{k \in \N}$ is also weak Cauchy.
%	
%	By the characterization of the dual of the product (Remark \ref{r:dualities}), every functional on $X$ is a linear combination of evaluations.
%	It is therefore enough to show that $\{x_{n_{k}}\}_{k \in \N}$ is convergent over every evaluation.
%	Let $e_{\alpha}$ be the evaluation at $\alpha \in C$.
%	If $\alpha \notin S$ then $x_{n_{k}}(\alpha) = 0$ for every $k \in \N$.
%	Otherwise, $\{x_{n_{k}}(\alpha)\}_{k \in \N}$ converges by the construction of $S$.
	\end{proof}

%1009
\begin{remark}
	Theorem \ref{t:tame_dichotomy} can be considered as a locally convex version of Rosenthal's classical dichotomy (in Banach spaces). As shown in Theorem \ref{thm:bounded_metrizable_tameness}, ``Tame dichotomy" of Theorem \ref{t:tame_dichotomy} implies that Rosenthal's classical dichotomy from Banach spaces can be extended to a much larger class of lcs with metrizable bounded subsets. This strengthens a well-known  result of Ruess \cite{Ruess} (Fact \ref{fact:Ruess}). On the other hand, Rosenthal's dichotomy is not true for general lcs as Theorem \ref{p:1And3AreNotReversible} demonstrates.  
\end{remark}

\sk
%2409	Moved before Haydon because some of these results are needed
\subsection*{Strongest tame image of a lcs}
\label{s:strongestTame} 
%1910 Changed to use the results of the Bornological classes section
In this subsection, we apply the results of Subsection \ref{subsection:co_bornology} to the 
%2012 changeed to bornological: categorical 
bornological class of tame functions.
By \ref{cor:tame_is_polarly_compatible}, this is indeed a polarly compatible bornology.
\begin{defin}
%0202	Let $E$ be a locally convex space.
	An equicontinuous subset $M \subseteq E^{*}$ is said to be \textit{co-tame} if $M \in \cobsmall[\BTame]{E}$.
	Explicitly, $M$ is co-tame if and only if every bounded $A \subseteq E$ is tame over $M$. 
	%2002 
	Similarly can be defined \textit{co-Asplund} subsets. 
\end{defin}

%2002 I think this remark gives an important example of co-tame (co-Asplund) subsets 
The following remark provides a useful class of co-Asplund (hence, co-tame) subsets.  
\begin{remark} \label{r:separ} 
	Let $E$ be a lcs and $K \in \eqc (E^*)$. Suppose that $K$ is separable (or, more generally, \textit{uniformly Lindelof} \cite{Me-fr}) in the standard strong topology on $E^*$. Then 
	$K$ is (weak$^*$, strong)-fragmented by \cite[Prop. 3.10]{Me-fr}. It follows that 
	every bounded subset $B \subset E$ is fragmented on $K$. 
	%2020 One may say maybe "K is co-Asplund"  	
	Hence $K$ is co-Asplund (so, also co-tame). 
	This explains also assertion (3) in Examples \ref{ex:fromNZ}.  
\end{remark}  

\begin{lemma} \label{lemma:co_tame_is_locally_convex_bornology}
	The family of all co-tame subsets $\cobsmall[\BTame]{E}$ is a locally convex bornology.
	In particular, the weak-star closed absolutely convex hull of a co-tame subset is co-tame.
	%100122
	It is also true relative to a given bounded subset $B$.
	This means that if $B$ is tame over $M$, it is also tame over $\overline{\acx}^{w^{*}}(M)$.
\end{lemma}
\begin{proof}
	Application of Lemma \ref{lemma:co_b_small_is_locally_convex} in light of $\BTame$ being polarly compatible.
\end{proof}

\begin{thm} \label{t:str-t} 
	Let $(E, \tau)$ be a locally convex space and $\cobsmall[\BTame]{E}$ 
	%030422	Changed "is" to "be" by Saak suggestion
	be the locally convex bornology of co-tame subsets.
	Then the polar topology $\tau_{tame}$ is the strongest tame locally convex topology, coarser than $\tau$, and
	$$
	\tau_{w} \subseteq \tau_{tame}  \subseteq \tau \subseteq \tau_{\mu}
	$$
	where $\tau_{\mu}$ is the Mackey topology.
\end{thm}
\begin{proof}
	Theorem \ref{thm:strongest_b_topology} and Lemma \ref{lemma:b_topology_relations}.
\end{proof}

%0202 Adding a remark about NP (because we have it in the problem). What about DLP ?
Similar results are valid also for the class $\NP$. 

%2208
\begin{prob} \label{p:TameInj}
	Study $\tau_{tame}$ and $\tau_{NP}$ for remarkable (e.g., classical) locally convex (or, Banach) spaces $(E,\tau)$. 
	%2208q what if $E=l^1$ ? 
	%0109  I've thought about it quite a lot, couldn't find a proof but I guess that l^1_{tame} = l^{1}_{w}. %1009q It would be a remarkable result 
	
	%It would be interesting to describe $eqc_{Fr}(E^*)$ and $eqc_{Tame}(E^*)$ for remarkable lc (or, Banach) spaces $E$. In particular, what if $E=c_0, E=l_1$ ? 
\end{prob}

%1910	New 
\begin{conj} \label{conj:strong_tame_topology}
	%1511	Reformulated
	Let $(l^{1}, \tau^{1})$ be the usual $l^{1}$ space.
	Consider the inclusion map $i_{p}\colon l^{1} \to l^{p}$ and $i_{w}\colon l^{1} \to l^{1}_{w}$.
	We claim that the weak topology induced by $\{i_{p}\}_{p > 1} \cup \{i_{w}\}$ is equal to $\tau^{1}_{tame}$.
\end{conj}

\sk 
\section{Free locally convex spaces revisited} 
\label{s:free_lcs}

%0808 
Given a class $P$ of Banach spaces, a locally convex space 
%2012 removed: (LCS) 
$E$ is
called \textit{multi}-$P$ (see \cite{LU}) 
if E can be isomorphically embedded into a product of spaces that
belong to $P$.  

%0808 probably instead of "multi-reflexive" I would suggest (to Leid-Usp) the name "pro-reflexive"

%For lcs of the type $C_k(X)$ we have a concrete (and somewhat expected) criterion, Proposition \ref{p:C(X)}. Another important construction producing many lcs is the free locally convex space $L(X)$, $L(X,\mu)$ defined for a topological space $X$ and a uniform space $(X,\mu)$, respectively. 

For every compact space $K$ its free lcs  
$L(K)$ is \textit{multi-reflexive}, %0808 (that is, embedded into a product of reflexive Banach spaces), 
as it was proved in a recent paper by Leiderman and Uspenskij \cite{LU}. Since multi-reflexive lcs is DLP 
(by Theorem
%2012 Fixed reference
\ref{thm:dlp_iff_semireflexive}), we obtain that $L(K)$ is 
DLP. In Theorem \ref{t:L(X)isNP}, we give a stronger result using the following two facts.

%2007	written here explicitly
\begin{f}\cite[Proposition 2.7]{GLX} \label{fact:bounded_subsets_in_free_lcs}
	For a subset $A$ of $L(X)$, the following assertions are equivalent:
	%2208 minor remark: itemizing in some cases we write ";" and sometimes just the dot 
	%0109 The easy solution would be to replace all ";" with dots. what do you think? %0609 it is possible. Let's decide during our discussion 
	\ben
		\item $A$ is bounded;
		\item $\support(A)$ has compact closure in the Dieudonn\'{e} completion $\mu X$ %of $X$ 
		and $C_{A} := \sup_{\chi \in A \cup \{0\}}  \lVert \chi \rVert$ is finite;
		\item $\support(A)$ is functionally bounded in $X$ and $C_{A}$ is finite.
	\een
%	Here:
%	$$
%	  C_{A} := \sup_{\chi \in A \cup \{0\}}  \lVert \chi \rVert.
%	$$
\end{f}

%2007	written here explicitly
\begin{f} \label{f:inclusion_of_free_compact_generated}  
	%2507 
	\cite[Thm.~2]{UspFree} 
	If $K \subseteq X$ is a compact subset of a Tychonoff space $X$, 
	%2507 I guess $X$ is Tychonoff, right ?  
	%2707q Yes, should it be explicitly mentioned again?  %0108 it is not really necessary but I prefer to be more precise here. So I add "Tychonoff"
	 then $L(K)$ naturally can be viewed as a subspace of $L(X)$.
\end{f}
\begin{proof}
	%0202q why we need this proof (if we give a reference) ?
	%1102	The reference given here isn't strong enough to show the given statement by itself.
	%1302q one may say something 
	It is easy to see that 
	%2507
	$K$ is C-imbedded into $X$, that is, 
	every continuous function on $K$ can be extended to $X$ (a proof can be found in Proposition \ref{p:C(X)}).
	Moreover, $K$ is compact and therefore, by \cite[Thm.~2]{UspFree}, the inclusion $L(K) \subseteq L(X)$ is an embedding.
\end{proof}

%@@@
%0607	New
\begin{lemma} \label{lemma:dieudonne_complete_dlp}
The free lcs $L(X)$ is 
	%1907q I think we can prove more: that such L(X) is (DLP).  This can be done by the same method; note that multi-reflexive is DLP. 
	DLP (in particular, is $\NP$) for every Dieudonn\'{e} complete space $X$.
\end{lemma}
\begin{proof}
	Let $B \subseteq L(X)$ be bounded and let $M \subseteq C(X) = L(X)^{*}$ be a weak-star compact, equicontinuous subset.
	We will show that $B$ has the DLP over $M$.
	 
	Since $X$ is Dieudonn\'{e} complete and by Fact \ref{fact:bounded_subsets_in_free_lcs}, $\support(B) \subseteq X$ has compact closure.
	Let 
	$$
	K := \overline{\support(B)} \subseteq X.
	$$
	%1107 here it is important to explain why L(K) is embeddded into L(X) (reference ?) 
	%1207 you are right, it is less trivial than I thought. Added a reference.
	%2007: moved to separate fact "Because $K$ is compact, it is also"
	%1907q this place requires some additional explanation; in particular, why you mention here pseudo-Lindel\"{o}f etc. 
	%2007: moved to separate fact pseudo-Lindel\"{o}f.
	By Fact \ref{f:inclusion_of_free_compact_generated}, the inclusion $L(K) \subseteq L(X)$ is an embedding.
	Thus, $B \subseteq L(K) \subseteq L(X)$.
	
	Moreover, $M$ can be viewed as an equicontinuous subset of $L(K)^{*}$.
	We already 
	%1107 established 
	mentioned 
	that $L(K)$ is multi-reflexive by \cite{LU}. Then $L(K)$ is DLP because the class is closed under products and subspaces 
	%2107	Changed reference (see Remark \ref{r:DLPprop}). 
	(see Theorem \ref{thm:properties_of_dlp_class}).
	Therefore $B$ is DLP on $M$. 
	This is true for every bounded $B$ and weak-star compact, equicontinuous $M \subseteq L(X)^{*}$.
	By definition, $L(X)$ is DLP. 
\end{proof}

%100422	Added
\begin{thm} \label{t:L(X)isNP} 
	%130422 	For every Tychonoff space $X$, $L(X)$ is DLP.
$L(X)$ is DLP for every Tychonoff space $X$. 
\end{thm}
\begin{proof}
	Write 
	%1104 $\mathcal{D}X$ 
	$\mu X$
	for the Dieudonn\'{e} completion of $X$.
	From the first sentence of the proof of \cite[Thm.~5]{UspFree}, we know that $L(X)$ is embedded in $L(\mu X)$. 
	By Lemma \ref{lemma:dieudonne_complete_dlp}, $L(\mu X)$ is DLP.
	Now, by Theorem \ref{thm:properties_of_dlp_class}, so is $L(X)$ as its subspace.
\end{proof}

%100422	Added
\begin{remark}
	We are indebted to S. Gabriyelyan for suggesting  
	%the previous 
	Theorem \ref{t:L(X)isNP} and its present proof as a consequence of Lemma \ref{lemma:dieudonne_complete_dlp}.
\end{remark}

%1107   
%%0607	New
%\begin{cor}
%	The space $P = \N^{\N}$ is completely metrizable (being homeomorphic to $\N^\N$) and therefore Dieudonn\'{e} complete.
%	Applying the previous theorem we conclude that $L(P)$ is $\NP$ and therefore tame.
%\end{cor}

%2507 
\begin{remark} \label{r:2} \ 
	\begin{enumerate}
		\item %1107 
		%1104 As a consequence of Theorem \ref{t:L(X)isNP}, 
		By Theorem \ref{t:L(X)isNP},  
		$L(P)$ is DLP (in particular, is NP) for the Polish space $P = \N^{\N}$ of all irrationals. In contrast, 
		another result from \cite{LU} shows that $L(P)$ is not multi-reflexive. 
		%2208 It would be interesting to examine if $L(P)$ is multi-Asplund or, at least, multi-Rosenthal. 
		% not, and if $L(X)$ is always tame.  
		
		\item 
		%1104 
%		Note that every metrizable lcs $X$ (in fact, every lcs $X$ such that $X$ is Dieudonn\'{e} complete) is a linear topological quotient of some lcs which is DLP. Indeed, the identity map $id \colon X \to X$ can be canonically extended to the linear onto map $L(X) \to X$, which is a quotient (open) map (because its restriction $id \colon X \to X$ is an onto factor map).  Now, observe that $L(X)$ is DLP by Theorem \ref{t:L(X)isNP}. 
		Note that every lcs $X$ is a linear topological quotient of some lcs which is DLP. Indeed, the identity map $id \colon X \to X$ can be canonically extended to the linear onto map $L(X) \to X$, which is a quotient (open) map (because its restriction $id \colon X \to X$ is an onto factor map).  %Now, observe that $L(X)$ is DLP by Theorem \ref{t:L(X)isNP}. 
	\end{enumerate}
\end{remark}

%2007	New
%2507 seems good and leads to new interesting questions concerning the completion (taking into account that L(X) very rarely is complete). For example: When the completion of L(X) Is semi-reflexive, reflexive ? What if, X =P is the space of irrationals ? Note that by [Leid-Usp] if X=K is a compact space then L(X) is multi-reflexive. Hence, its completion is reflexive. In a new version of [Leid-Usp] the authors prove that the same is true for locally compact+sigma-compact X

%2208 
\begin{q} \label{q:multi}
	%It would be interesting to 
	Examine if $L(\N^{\N})$ is multi-Asplund or, at least, multi-Rosenthal. 
	% not, and if $L(X)$ is always tame.  Is it true that 
\end{q}

%2208q   As I remember a part of the proof below was based on old results of Uspenskij. Disd you replace it with some direct argument ? 
%0109. Yes, I think that I used the fact that Uspenskij characterized when L(X) is complete. However, semi-reflexivity implies bounded completeness so this is not as easily applicable as I once thought
%1009q  let's discuss 
%1109	Added some explanation and a reference to Uspenski
In \cite[p.~679, Corollary]{UspFree}, Uspenski shows that for a Dieudonn\'{e} complete space, $L(X)$ is complete if and only if there are no infinite compact subsets.
The following is a similar result that shows the ``scarcity" of semi-reflexive in the realm of free locally convex spaces.

%0604  (lest's discuss) 
\begin{thm} \label{t:LX_is_almost_never_semireflexive}
	Let $X$ be a Dieudonn\'{e} complete space.
	Then $L(X)$ is semi-reflexive if and only if $X$ has no infinite compact subset.
\end{thm}
\begin{proof}
	By Fact \ref{f:semi_reflexive_heine_borel}, semi-reflexivity is equivalent to the Heine-Borel property for the weak topology.
	
	First, suppose that $X$ has no infinite compact subset. 
	We will show that every bounded subset is weakly relatively compact.
	Let $B \subseteq L(X)$ be bounded.
	By Fact \ref{fact:bounded_subsets_in_free_lcs}, $B$ has compact support $K \subseteq X$.
	Since every compact subset of $X$ is finite, we conclude that $\support(B) \subseteq K$ is also finite.
	As a consequence, $B \subseteq \spn(\support(B))$ is a bounded subset in a finite dimensional topological space.
	In this case, the weak topology of $\spn(\support(B))$ coincides with the original topology, and the Heine-Borel property is satisfied.
	
	Conversely, assume that $L(X)$ is semi-reflexive.
	By contradiction, assume that $X$ has an infinite compact subset $K \subseteq X$.
	%2012q Does this requires explanation?
	%090122  The full explanation, for every Hausdorff infinite space, is not so hard. So, one may leave it without a proof or a reference. 
	%090122  By the way, good to include it as an exercise for our topology course.  
	Since $K$ is infinite, we can find a discrete countable subset $\{x_{n}\}_{n \in \N} \subseteq K$.
	For every $n \in \N$ define:
	$$
	  \chi_{n} := \sum_{m=1}^{n} x_{m},
	$$
	and consider the set $B := \{\chi_{n} \}_{n \in \N}$.
	Clearly, $B$ is bounded in virtue of Fact \ref{fact:bounded_subsets_in_free_lcs}.
	By our assumption, $B$ is weakly relatively compact, and therefore has an accumulation point $\chi \in L(X)$.
	
	Write $\chi = \sum\limits_{i=1}^{k} \alpha_{i} y_{i}$ and
	choose some $n_{0} \in \N$ such that $x_{n_{0}} \neq y_{i}$ for every $1 \leq i \leq k$.
	By the choice of $\{x_{n}\}_{n \in \N}$, there is some neighborhood $x_{n_{0}} \in U \subseteq X$ such that 
	$$
	  A := \{x_{n}\}_{n_{0} \neq n \in \N} \cup \{y_{i}\}_{i=1}^{k} \subseteq 
	  X \setminus U.
	$$	
	Since $X$ is completely regular, we can find %a continuous function 
	%0202q  what is C_{k}(X) ?
	%1102 	the continuous functions with the compact open topology. 
	%		The topology isn't relevant here so I changed it to C(K)
	${f \in C(X) = (L(X))^{*}}$ such that
	$$
		f(x_{n_{0}}) = 1 \text{ and }
		\forall x \in A: f(x) = 0.
	$$
	By the definition of accumulation point, there is some $n_{0} \leq N \in \N$ such that
	$$
	\lVert f(\chi) - f(\chi_{N}) \rVert < \frac{1}{2}.
	$$
	However, $f(\chi) = 0$ and $f(\chi_{N}) = 1$, which is a contradiction.
\end{proof}

%1907q  I think we can derive the following: every Frechet space V is a linear topological quotient of some lcs which has (DLP). Indeed, the identity map id: V \to V can be extended to L(V) \to V (which is quotient) 
%2007	Is this map open? Why Frechet and not simply complete?   %2507 right ! every topologically complete ... 

%1107 
%Another result from \cite{LU} shows that $L(P)$ is not multi-reflexive for the space $P=\N^{\N}$ of all irrationals. It would be interesting to examine if $L(P)$ is tame ($\NP$, multi-Asplund) or not.   

%%

%3003
%2306 Leiderman-Uspenskij result answrs this question trivially. So we have to replace this question 
%\begin{prob} Let $L(X)$ be the free locally convex space on a compact space $X$ and $B$ is a bounded subset of $L(X)$. 
%	Find sufficient conditions which guarantee that $B$ is tame in $L(X)$. 
%	 What if $X$ is the convergent sequence (with its limit point)? 
%\end{prob}

\sk 
We were informed recently by S. Gabriyelyan that Proposition \ref{t:LX_is_almost_never_semireflexive} is a partial case of a more general result which was proved in his (accepted) joint work with T. Banakh, \cite{Ban-Gabr}.  

%0507 
%1407 moved here 
%\begin{q} \label{q:FreeLCS}
%	%	Let $(X,\mu)$ be a %(Hausdorff) 
%	%	uniform space and 
%	%	let $L(X,\mu)$ be the corresponding free locally convex space and $(X,\mu)\hookrightarrow L(X,\mu)$ is the associated uniform embedding.  
%	Under which conditions on a 
%	%2007	Made the question a bit more specific
%	(non Dieudonn\'{e} complete) $X$ and on a uniform space $(X,\mu)$ are the corresponding free locally convex spaces $L(X)$ and $L(X,\mu)$:
%	%1107  
%	a) tame; b) multi-Rosenthal; c) NP; d) multi-Asplund; e) DLP ?	
%	%2306
%	%	What if $X$ is a scattered compact ? Say, the simplest infinite compact (i.e., convergent sequence with its limit point) ? 
%\end{q}

%%0604  (lest's discuss) 
%%3103
%The class of all tame lcs is closed under products and subspaces (Theorem \ref{thm:properties_of_tame_class}). This implies that 
%one may define \textbf{free tame locally convex space} $L_{tame}(X)$ on every topological space $X$. This can be defined in the same spirit as $L(X)$. Moreover, similarly can be defined also free objects in the classes $\DLP$ and $\NP$.  
%If $L(X)$ is tame (e.g., for every Dieudonn\'{e} complete space $X$ by Theorem \ref{t:L(X)isNP}) then $L_{tame}(X)$ can be identified with $L(X)$. 

\sk 
\section{DFJP factorization for lcs}
%050422	Many changes in this section.

Recall that a continuous linear operator $T \colon E \to F$ between two lcs is said to be \textit{(weakly) compact} if there exists a neighborhood $O$ of zero in $E$ such that $T(O)$ is relatively (weakly) compact in $F$, \cite[p.~94]{Groth}. 
A natural possibility to generalize these definitions is to require that $T(O)$ is small in some other sense, i.e., $T(O)$ belongs to some bornology of small subsets. 
In particular, we have the following definitions. 

\begin{defin} \label{d:TameOperat}  
	We say that a linear continuous map $T \colon E \to F$ between lcs is  \emph{tame} (\textit{NP, DLP}) if there exists a zero neighborhood $U \subseteq E$ such that $T(U) \subseteq F$ is a tame (NP, DLP) subset in $F$.
	%0908q	I think that it should be explicitly said that T(U) needs to be bounded.
	%0908	Also, it might be rather standard, but as a consequence, the identity 
	%0908   map of a tame LCS need not be tame. 
	%2208q you are right
\end{defin}

%2208q  
Note that by our definitions $T(U)$ is bounded. Hence, the identity map of a tame not normable lcs is not tame. 

The following fundamental result is a part of a 
%0704 crucial technical result is based on the 
classical work about DFJP factorization, \cite{DFJP}. 
In \cite{GM-rose,GM-tame,Me-seminW,Me-b}, some simplifications and adaptations were added.  

%Suppose that $\{(X_{n}, \lVert \cdot \rVert_{n})\}_{n \in \N}$ is a sequence of Banach spaces.
%For every $(x_{n})_{n \in \N} \in \prod_{n \in \N} X_{n}$ we write
%$$
%  \lVert(x_{n})_{n \in\N} \rVert_{l^{2}} := \left(\sum_{n \in \N} \lVert x_{n} \rVert_{n}\right)^{1/2}.
%$$
%%0604 let's discuss. I think we can mention this without trying to explain. 
%Like \cite{DFJP}, we define $(\sum_{n \in \N} X_{n})_{l^{2}}$ as the subspace of $\prod_{n \in \N} X_{n}$ on which $\lVert \cdot \rVert_{l^{2}}$ is finite.
%This is evidently also a Banach space.

%\begin{lemma} \label{lemma:open_bounded_gague}
%	Let $(X, \lVert \cdot \rVert)$ be a Banach space and $U \subseteq X$ be a bounded open subset.
%	Then the gauge $q_{U}$ of $U$ is equivalent to the norm of $X$.
%\end{lemma}
%\begin{proof}
%	Before we begin, note that $U$ is open and therefore absorbing, meaning that $q_{U}$ is indeed defined on the entire $X$.
%	Next, $U$ is bounded, then there exists some $\alpha > 0$ such that $\alpha U \subseteq B_{X}$.
%	Similarly, $U$ is open and $B_{X}$ is bounded so there exists $\beta > 0$ such that $B_{X} \subseteq \beta U$.
%	As a consequence, for every $x \in X$:
%	$$
%	  \frac{1}{\alpha} q_{U}(x)= q_{\alpha U}(x) \leq q_{B_{X}}(x) = \lVert x \rVert \leq q_{\beta U}(x) = \frac{1}{\beta}q_{U}(x). 
%	$$
%	By definition, the norms are equivalent.
%\end{proof}

%0606 I think citation is enough. Let"s discuss 
\begin{f} \cite[Lemma~1]{DFJP} \label{fact:DJFP}
	Let $X$ be a Banach space and $W \subseteq X$ be an absolutely convex bounded subset.
	Write $U_{n} := 2^{n} W + 2^{-n}B_{X}$ and $\lVert \cdot \rVert_{n}$ for the gauge of $U_{n}$. 
	Also, consider $[ x ] := (\sum_{n\in \N} \lVert x \rVert_{n}^{2})^{1/2}$ and $Y = \{x \in X \mid [x] < \infty\}$.
	Finally, let $j\colon Y \to X$ be the identity map.
	Then:
	\ben
		\item $W \subseteq B_{Y}$.
		\item $(Y, [\cdot])$ is a Banach space and $j$ is continuous.
		\item $j^{**}\colon Y^{**} \to X^{**}$ is one to one and $(j^{**})^{-1}(X) = Y$.
		\item $Y$ is reflexive if and only if $W$ is weakly relatively compact.
	\een
\end{f}

%0704 I think it is if and only if 
%\begin{lemma} \label{lemma:DJFP_double_dual_to_dual}
%	Let $j \colon X \to Y$ be a continuous linear map between locally convex spaces.
%	If ${j^{*}\colon Y^{*} \to X^{*}}$ is one to one then $j\colon X^{*} \to Y^{*}$ has a dense image.
%\end{lemma}
%\begin{proof}
%	By contradiction, suppose that $j(X)$ is not dense in $Y$.
%	In this case, we can find $0 \neq \varphi \in Y^{*}$ such that $j(X) \subseteq \ker \varphi$.
%	As a consequence, $j^{*}(\varphi) = 0$, in contradiction to $j^{*}$ being one-to-one.
%\end{proof}

\begin{remark} \label{r:j^* is dense}    
	%100422 Rearranging
	Note that in Fact \ref{fact:DJFP}, $j^* \colon X^* \to Y^*$ is dense as a consequence of Lemma \ref{lemma:DJFP_double_dual_to_dual}. %This observation is well known and can be found for example in \cite[Lemma 1.2.2]{Fabian}
\end{remark}

\begin{lemma} \label{lemma:DJFP_double_dual_to_dual}
	Let $f \colon E_1 \to E_2$ be a continuous linear map between locally convex spaces. Then ${f^{*}\colon E_2^{*} \to E_1^{*}}$ is one to one if and only if $f \colon E_1 \to E_2$ has a dense image. 
\end{lemma}
\begin{proof} Observe that $f(E_1)$ is not dense in $E_2$ if and only if 
	there exists $0 \neq \varphi \in E_2^{*}$ such that $f^{*}(\varphi) = 0$. 
\end{proof}

%0604 let's discuss
\begin{lemma} \label{lemma:fragmented_hyperspace}
	Let $K$ be a compact topological space and let $\{A_{n}\}_{n \in \N}$ be a sequence of bounded functions on $K$.
	Write $B$ for the unit ball of $C(K)$ and:
	$$
	  A := \bigcap_{n \in \N} \left(A_{n} + \frac{1}{2^{n}}B\right).
	$$
	If all $\{A_{n}\}_{n \in \N}$ are (eventually) fragmented, so is $A$.
\end{lemma}
\begin{proof} 
	We will only prove the case of eventually fragmented maps. 
	The other case is slightly easier.
	Suppose that $\{x_{m}\}_{m \in \N} \subseteq A$ is an infinite subset.
	We need to find a subsequence $\{m_{k}\} \subseteq \N$ such that $\{x_{m_{k}}\}_{k \in \N}$ is fragmented over $K$.
	By its definition, for every $n, m \in \N$, we can find $y(n, m) \in A_{n}$ such that $x_{m} \in y(n, m) + \frac{1}{2^{n}}B_{n}$ (we use function notation rather than subscript to make it easier to read).
	Using induction, we build a sequence $\left\{\left\{m_{k}^{(n)}\right\}_{k \in \N}\right\}_{n \in \N}$ of descending infinite subsequences.
	We define $m_{k}^{(-1)} := k$ (the trivial subsequence), and at each step choose a subsequence $\left\{m_{k}^{(n + 1)}\right\}_{k \in \N} \subseteq \left\{m_{k}^{(n)}\right\}_{k \in \N}$ such that
	$\left\{y\left(n + 1, m_{k}^{(n + 1)}\right)\right\}_{k \in \N}$ is fragmented.
	Now, define $m_{k} := m_{k}^{(k)}$.
	It is easy to see that $Y_{n} := \left\{y\left(n, m_{k}\right)\right\}_{k \in \N}$ remains fragmented.
	We will show that so does $A' := \{x_{m_{k}}\}_{k \in \N} \subseteq A$.
	
	Let $\eps > 0$ and $T \subseteq K$ be non-empty. 
	We need to find an open $O \subseteq K$ such that $T \cap O$ is not empty and $f(T \cap O)$ is $\eps$-small for every $f \in A'$.			
	There exists some $n_{0} \in \N$ such that $\frac{1}{2^{n_{0}}} < \frac{1}{2}\eps$.
	By our construction, $\left\{y\left(n_{0}, m_{k}\right)\right\}_{k \in \N}$ is fragmented so we can find an open $O \subseteq K$ such that $O \cap T$ is not empty and $f(O \cap T)$ is $\frac{1}{2}\eps$-small for every $f \in Y_{n}$.
	Since $A' \subseteq Y_{n} + \frac{1}{2^{n_{0}}}B$, it is easy to see that for every $f \in A'$, $f(O \cap T)$ is $\eps$-small, as required.
\end{proof}
	
\begin{f} \label{l:DenseDualImage}  \cite{Me-b} 
	Let $j \colon V_1 \to V_2$ be a continuous linear operator between Banach spaces 
	such that the adjoint $j^* \colon V_2^* \to V_1^*$ is norm-dense. Let $F \subset V_1$ be a bounded subset.
	\begin{enumerate} 
		\item If $j(F)$ is a tame subset in $V_2$, then $F$ is a tame subset in $V_1$. 
		%2208 I tried but cannot see if the same is true for DLP (since the setting here is in Banach spaces then we can assume that "relatively weakly compact" instead of "DLP". Maybe some extracondition is needed)
		%	\item If $j(A)$ has DLP in $V_2$ then $A$ has DLP in $V_1$. 
			\item If $j(F)$ is an Asplund subset in $V_2$, then $F$ is an Asplund subset in $V_1$. 
	\end{enumerate}
\end{f}
\begin{proof}
	(1) In order to show that $F$ is a tame subset in $V_1$ it is equivalent to check that $F$ is eventually fragmented on $B_{V_1^*}$
	(Lemma \ref{f:sub-fr}).  
	We have to prove that for every 
	sequence $\{f_n\}_{n \in \N}$ in $F$
	there exists a subsequence $\{f_{n_k}\}_{k \in \N}$ which is fragmented on $B_{V_1^*}$.  
	Equivalently, $(B_{V_1^*},\rho_C)$ is separable (Lemma \ref{t:countDetermined}), where $C:=\{f_{n_k}\}$ 
	%130422  
	and $\rho_{C}(x_1,x_2):=\sup_{f \in C} |f(x_1)-f(x_2)|.$ 
	
	By our assumption, $j(F)$ is a tame subset in $V_2$. Hence, there exists a subsequence $C:=\{f_{n_k}\}$ of $\{f_{n}\}$ such that $j(C)=\{j(f_{n_k})\}$ is a fragmented family on $B_{V_2^*}$. Equivalently, $(B_{V_2^*}, \rho_{j(C)})$ is separable (Lemma \ref{t:countDetermined}).  
	Then $(V_2^*, \rho_{j(C)})$ is also separable. 
	By the definition of the adjoint 
	%130422 operator  
	operator, we have 
	$\lan j(x),v^* \ran= \lan x,j^*(v^*) \ran$ for every $x \in V_1, v^* \in V_2^*$. This implies that $(j^*(V_2^*), \rho_C)$ is separable.
	Then its $\rho_C$-closure $cl_{\rho_C}(j^*(V_2^*))$ is also $\rho_C$-separable. Clearly,
	$cl_{\rho_C}(j^*(V_2^*)) \supset cl_{norm}(j^*(V_2^*))=V_1^*$ (because $C$ is a bounded sequence in $V_1$). Therefore, $(V_1,^*, \rho_C)$ and also $(B_{V_1^*},\rho_C)$ are separable, as desired. 
	
	(2) The case of Asplund subsets is similar (and easier). 
	\end{proof}

\begin{lemma} \label{l:DFJP}
	Let $A$ be a bounded tame (Asplund, relatively weakly compact) subset in a Banach space $X$. Then there exist a Rosenthal (Asplund, reflexive) Banach space $Y$, a continuous linear injective map $j \colon Y \to X$ such that $A$ is a subset of $j(B_Y)$ and the adjoint map $j^* \colon X^* \to Y^*$ is dense.
\end{lemma}
\begin{proof}
	Define $W := \acx A$.
	The case of relatively weakly compact subsets is a consequence of Fact \ref{fact:DJFP}. %0704 and Lemma \ref{lemma:DJFP_double_dual_to_dual}.
	Let us consider the cases where $A$ is tame or Asplund.
	By Lemma \ref{lemma:class_induces_bornology}, $W$ remains tame (Asplund).
	Next, let $j\colon Y \to X$ and $\{U_{n}\}_{n \in \N}$ be as described in Fact \ref{fact:DJFP}. 
	It is 
	%0704 
	well known and 
	easy to see that 
	%0704 $B_{Y} \subseteq U := \bigcap_{n \in \N} U_{n}$.
	$$A \subseteq W \subseteq j(B_{Y})=B_Y \subseteq U := \bigcap_{n \in \N} U_{n}.$$
		%0704 
	Then $j^* \colon X^* \to Y^*$ is dense by Remark \ref{r:j^* is dense}. 
	By Lemma \ref{f:sub-fr}, the tameness (Asplundness) of bounded families of continuous functions on compact sets is equivalent to eventual fragmentability (fragmentability).
	By Lemma \ref{lemma:fragmented_hyperspace}, we conclude that $U$
	%0704 is 
	and $j(B_Y)$ 
	also are tame (Asplund) subsets.
	By Fact \ref{l:DenseDualImage}, $B_{Y}$ is tame (Asplund).
	Therefore, $Y$ is Rosenthal (Asplund).
	%0704 Also, $A \subseteq W \subseteq B_{Y} = j(B_{Y})$
\end{proof}

\sk 
\begin{thm} \label{t:Factoriz} 
		Every tame (NP, DLP) operator $T \colon E \to X$ between a lcs
		%0604 
		$E$ 
		 and a Banach space 
		 $X$ 
		 can be factored through a Rosenthal 
		 %0704 (Asplund, DLP) Banach space.  
		(Asplund, reflexive) Banach space. 
\end{thm}
\begin{proof}
	%2609 As expected, one may use the famous DFJP construction \cite{DFJP} and Lemma \ref{l:DFJP}. 
	%We give details only for the ``tame case". The other two cases are similar. 
	By our assumption, there exists a zero neighborhood $O$ in $E$ such that $T(O)$ is tame (Asplund, weakly relatively compact) in $X$. We apply Lemma \ref{l:DFJP} to the subset $A:=T(O) \subset X$ in order to construct a \textit{Rosenthal (Asplund, reflexive) Banach space} $V$ and a continuous injective linear operator 
	$j \colon V \to X$ such that $A$ is a subset of $j(B_V)$. Now, consider the linear operator $u:=j^{-1} \circ T \colon E \to V$. 
	%0908q Maybe we should explain here why this composition is defined. %2208q  ok 
	Since $u(O) \subset B_V$, we obtain that $u$ is continuous. Then $T= j \circ u$ is the required factorization. 
\end{proof}

It would be interesting to find some additional natural bornologies which are consistent with DFJP-factorization.

\sk 
\section{Generalization of Haydon's Theorem and tame spaces} 
\label{s:haydon}
% 010322
In its original statement Haydon's theorem characterizes Rosenthal Banach spaces 
as follows.
\begin{thm} \emph{(Haydon \cite[Thm.~3.3]{Haydon})} \label{t:HaudonBan} 
	Let $V$ be a Banach space. The following are equivalent:
	\ben
		\item $V$ contains no $l^{1}$-sequence; %0812 maybe better to repace with not containing of an $l^{1}$-sequence (as in the following thm) 
		%1012	Changed
		\item every weak-star compact convex subset of $V^{*}$ is the norm closed convex hull of its extreme points;
		\item for every weak-star compact subset $T$ of $V^{*}$, 
		$$
		  \overline{\co}^{w^{*}}(T)  = \overline{\co}(T).
		$$
	\een
\end{thm}

%1511	New
Our generalized version, in a brief summary, can be expressed as follows.
\begin{prop}
	For a locally convex space $E$, the following are equivalent:
	\ben
		\item $E$ is tame (in virtue of Theorem \ref{thm:tame_iff_R},
		%1611 equivalent to not containing $l^{1}$); 
		equivalent to not containing of an $l^{1}$-sequence);
		\item every equicontinuous, weak-star compact convex subset of $E^{*}$ is the 
		%050422	Changed "norm" to strong by Saak suggestion
		strong closed convex hull of its extreme points. 
		%1302new adding (please check !) 
		That is, 
		$\overline{\co}^{w^{*}}(\ext M)= 
		\overline{\co}(\ext M)$ for every 
		%1302F
		convex 
		$M \in \eqc(E^*)$;
		\item for every equicontinuous, weak-star compact subset $T$ of $E^{*}$, 
		$$
		\overline{\co}^{w^{*}}(T)  = \overline{\co}(T).
		$$
	\een
\end{prop}

\sk 
In fact, Theorem \ref{thm:Haydon} will show an even more %specific 
localized  result for a given equicontinuous subset $M$.
%1302new  A statement of similar nature was proven in \cite[Proposition 3.8]{LocalHaydon} for Banach spaces.
%1302new adding your remark 
%1202R 
There are some results of a similar nature for Banach spaces. Among others: \cite[Proposition 3.8]{LocalHaydon}, \cite[Thm.~9 and 11]{LocalHaydonRiddle83}, \cite[Thm.~1]{Saab82} and \cite[Thm.~4]{LocalHaydonRiddle82}.

%150122 No longer necessary
%1910	Added intro
%The main step of the proof is to show that if $M \subseteq E^{*}$ is weak-star compact, convex and %1611 tame, 
%co-tame, then 
%$$
%  M = \overline{\co} \ext (M).
%$$
%%1511 reformulation
%This will be done in Lemma \ref{prop:cotame_implies_haydon}.
%Note that this extends one of the implications of Haydon's theorem in two directions: a localization to a specific %1611 tame 
%co-tame subset (rather than the entire space being tame) and 
%%1611 the move 
%extending to general locally convex spaces.
%This will be done in three steps:
%\ben
%	\item \emph{localization to a weak-star compact set in a normed spaces (Lemma \ref{lemma:localized_haydon_for_banach})}: a reduction to Haydon's theorem with the factorization of Lemma \ref{lemma:equicontinuous_factorization_and_b_smallness}.
%	\item \emph{localization to a bounded set in a general locally convex space (Lemma \ref{lemma:twice_localized_haydon})}: a second localization to the normed space $E_{B}$ for a given bounded set $B$.
%	\item \emph{proving the main step (Lemma \ref{prop:cotame_implies_haydon})}: application of the previous claim for every bounded subset.
%\een

%1910	New
%150122	major changes
\sk 
\subsection*{Application of the DFJP Construction}
By the Krein–Milman theorem, if $M \in \eqc(E^{*})$ is convex, then $M = \overline{\co}^{w^{*}}(\ext M)$.
In light of Haydon's theorem, we give the following definition.
\begin{defin} \label{d:anti-H} 
	Let $M \in \eqc (E^*)$ be convex. We say that a bounded set $B \subset E$ is \textit{anti-H} 
	%0202  \footnote{tentative name} 
	for $M$ if there 
	%050422	Changed exists to exit by Saak suggestion
	exist a functional $\psi \in M$ and $\eps >0$ such that 
	$$
	U[B,\eps](\psi) \cap \co (\ext M) = \emptyset. 
	$$
	%100122
%	Alternatively, using the 
	%0202q  better to be more precise about \overline{\co}^{B} notation (where exactly); in Section 6 I could not find the definition 
	%1102	This notation is no longer necessary in other places and was removed.
	%		I am removing it here too
	%1302q I don't understand ! 
%	 notation of Section \ref{section:generalized_l1_sequence}, we could also write:
%	$$
%	\psi \notin \overline{\co}^{B}(\ext M).
%	$$
\end{defin}

\sk 
The following is a direct consequence of the previous definition.
\begin{lemma} \label{lemma:application_of_anti_H}
%0202 	Let $E$ be a locally convex space and $M \in \eqc(E^{*})$. Then 
%1302new It was M on the left side 
$\overline{\co}^{w^{*}}(\ext M)= 
\overline{\co}(\ext M)$ if and only if there is no bounded $B \subseteq E$ which is anti-H for $M$.
\end{lemma}

%0202 I think it is clear and we need it only once. So, better not mentioning it  especially
%%100122
%\begin{remark} \label{remark:anti_h_is_monotone}
%	The anti-H property is monotone, namely: if $B \subseteq E$ is anti-H for $M$ and $B \subseteq C$, then $C$ is also anti-H for $M$.
%\end{remark}

\begin{lem} \label{l:HaydImage} 
	Let $T \colon E_1 \to E_2$ be a continuous linear map between lcs,
	${T^*\colon E_2^* \to E_1^*}$ 
	%050422 Changed "is" to "be" trhee times
	be its adjoint, 
	$B$ be a bounded subset in $E_1$ and $M \in \eqc (E_2^*)$ be convex. 
	Then 
	\begin{enumerate} 
		\item $(\varphi_1,\varphi_2) \in U[T(B),\eps] \ \text{in} \ E_2^* \Longleftrightarrow (T^*(\varphi_1), T^*(\varphi_2)) \in U[B,\eps] \ \text{in} \ E_1^*$;
		\item for every $\varphi \in E_{2}$ we have 
		$$
		  U[T(B), \eps](\varphi) = (T^{*})^{-1}(U[B, \eps](T^{*}(\varphi)));
		$$
		%0202q worth to think if this works in both directions without any additional condition. 
		\item if $T(B)$ is anti-$H$ over $M$ then $B$ is anti-$H$ over $T^{*}(M)$;
		\item if the image of $T$ is dense in $E_{2}$, then the converse is also true, namely, $B$ being anti-$H$ over $T^{*}(M)$ implies that $T(B)$ is anti-$H$ over $M$.
	\end{enumerate} 
\end{lem}
\begin{proof} \ 
	\ben
	\item By definition of the adjoint $T^* \colon E_2^* \to E_1^*$, we have  
	$$
	\lan b, T^*(\varphi) \ran=\lan T(b), \varphi \ran \ \ \forall b \in B.
	$$
	Now apply the descriptions of $U[B,\eps], U[T(B),\eps]$ according to Definition \ref{d:strongTOP}. 
	\item Using (1) we get:
	\begin{align*}
		\varphi' \in U[T(B), \eps](\varphi) & \iff
		(\varphi, \varphi') \in U[T(B), \eps] \\
		& \iff (T^{*}(\varphi), T^{*}(\varphi')) \in U[B, \eps] \\
		& \iff T^{*}(\varphi') \in U[B, \eps](T^{*}(\varphi)) \\
		& \iff \varphi' \in (T^{*})^{-1}(U[B, \eps](T^{*}(\varphi))).
	\end{align*}
	\item By definition, there exists $\varphi \in M$ such that:
	$$
	U[T(B), \eps](\varphi) \cap \co(\ext M) = \emptyset.
	$$
	Write $\psi := T^{*}(\varphi)$.
	Suppose by contradiction that $B$ is not anti-$H$ for $T^{*}(M)$.
	Thus, we can find 
	$$
	\psi' \in U[B, \eps](\psi) \cap \co (\ext T^{*}(M)).
	$$
	By Lemma \ref{lemma:extreme_points_of_compact_map}.1, ${\co (\ext T^{*}(M)) \subseteq T^{*}(\co (\ext M))}$.
	We can therefore find $\varphi' \in \co(\ext M)$ such that ${\psi' = T^{*}(\varphi')}$.
	Also, by (2):
	$$
	  \psi' = T^{*}(\varphi') \in U[B, \eps](T^{*}(\varphi)) \Rightarrow
	  \varphi' \in (T^{*})^{-1}(U[B, \eps](T^{*}(\varphi))) = U[T(B), \eps](\varphi).
	$$
	Using all the information on $\varphi'$ so far, we get
	$$
	\varphi' \in U[T(B), \eps](\varphi) \cap \co(\ext M) = \emptyset,
	$$
	a contradiction.
	\item Since $T$ has a dense image, $T^{*}$ is injective and so Lemma \ref{lemma:extreme_points_of_compact_map}.2 is
	%0202q lets' discuss (that lemma is for T, this might be confusing)  
	 applicable:
	 %050422	Changed minor error  ("f" to "T") by Saak suggestion
	\begin{align*}
		B \text{ is anti-}H \text{ for } T^{*}(M) & \iff
		\exists\ \varphi \in M: U[B, \eps](T^{*}(\varphi)) \cap \co(\ext T^{*}(M)) = \emptyset \\
		%1102 changed to be one directional
		& \underset{(2)}{\Longrightarrow} \exists\ \varphi \in M: T^{*}(U[T(B), \eps](\varphi)) \cap \co(\ext T^{*}(M)) = \emptyset \\
		& \underset{\text{\ref{lemma:extreme_points_of_compact_map}.2}}{\iff} \exists\ \varphi \in M:  T^{*}(U[T(B), \eps](\varphi)) \cap T^{*}(\co(\ext M)) = \emptyset \\
		& \underset{\star}{\iff} \exists\ \varphi \in M:  U[T(B), \eps](\varphi) \cap \co(\ext M) = \emptyset \\
		& \iff T(B) \text{ is anti-}H \text{ for } M. 
	\end{align*}
	%050422 Made more clear by Saak suggestion
	Note that the equivalence marked by $\star$ is true in virtue of another application of $T^{*}$ being injective.
	\een
\end{proof}

\begin{prop} \label{prop:localization_of_haydon_with_DFJP}
	%150122 Made simpler
	Let $V$ be a Banach space, $A \subseteq V$ be bounded and ${M \in \eqc(V^{*})}$ be convex.
	%Furthermore, suppose that $M \in \eqc(V^{*})$ such that $B_{V^{*}} \subseteq \overline{\acx}^{w^{*}}(M)$.
	If $A$ is anti-H with respect to $M$, then 
	%0202 it 
	$A$ is not tame in $V$.
\end{prop}
\begin{proof}
	By contradiction, assume that $A$ is tame in $V$.
	Let $W$ and $j\colon W \to V$ be the Rosenthal Banach space and the map described in the factorization Lemma \ref{l:DFJP} 
	%0202 
	(caution: the notation was $V$ and not $W$ in Lemma \ref{l:DFJP}). 
	Thus, $A \subseteq j(B_{W})$ and therefore $j(B_{W})$ is also anti-H over $M$. %0202 (Remark \ref{remark:anti_h_is_monotone}).
	By Lemma 
	%0202q  \ref{l:HaydImage}.4, 
	 \ref{l:HaydImage}.3, 
	$B_{W}$ is anti-$H$ over $j^{*}(M)$.
	However, $j^{*}(M)$ is a weak-star compact subset of a Rosenthal Banach space - a contradiction to Haydon's Theorem \ref{t:HaudonBan}.
\end{proof}

\begin{prop} \label{prop:cotame_implies_haydon}
	%0202 
	Let $E$ be a lcs. 
	%150122 No longer using "haydon property" since we use it later for something else
	If $M \in \eqc(E^{*})$ is co-tame and convex, then 
	%050422	Changed to \overline{\co}
	$M = \overline{\co} (\ext M)$.
\end{prop}
\begin{proof}
	Clearly, $M \supseteq \overline{\co} (\ext M)$.
	To show the converse, assume by contradiction that $M \supsetneq \overline{\co}(\ext M)$.
	By Lemma \ref{lemma:application_of_anti_H}, we can find a bounded $B \subseteq E$ which is anti-$H$ for $M$.
	Let $V$, $\pi\colon E \to V$ and $\Delta\colon \spn M \to V^{*}$ be the maps described in Lemma \ref{lemma:equicontinuous_factor}.
	Write $N := \Delta(M) \in \eqc(V^{*})$ and note that $M = \pi^{*}(N)$ is weak-star compact in $V^{*}$ and therefore equicontinuous by the Banach-Steinhaus theorem.
	Thus, we can say that $B$ is anti-$H$ over $\pi^{*}(N)$.
	%0202  The image of $\pi$ is dense onto its image, 
	The map $\pi$ is dense,  
	so by Lemma \ref{l:HaydImage}.4 we conclude that $\pi(B)$ is anti-$H$ over $N$.
	
	However, $M$ is co-tame so $B$ is tame over $M = \pi^{*}(N)$ by definition.
	%0202q the following can be mentioned separately directly after (or before) Lemma \ref{lemma:b_small_and_adjoint_maps} (as a corollary of this lemma)
	%1102 I am not sure that it provides added insight. Just using both Lemmas 
	By Lemma \ref{lemma:b_small_and_adjoint_maps}, $\pi(B)$ is tame over $N$.
	Moreover, by Lemma \ref{lemma:co_tame_is_locally_convex_bornology}, $\pi(B)$ is tame over 
	$$ 
	  \overline{\acx}^{w^{*}} N = \overline{\acx}^{w^{*}} \Delta(M) = B_{V^{*}}.
	$$
	Thus, $\pi(B)$ is tame in $V$.
	Proposition \ref{prop:localization_of_haydon_with_DFJP} gives the contradiction.
\end{proof}

\sk 
%0812 \section{Haydon's Theorem for lcs}
\subsection*{The Haydon Property and other results}
\begin{defin} \label{d:HaydProp} 
	%2409	Changed the definition a bit to suit localization.
	Let $E$ be a locally convex space and let 
	%$M \subseteq E^{*}$ be an equicontinuous, weak-star compact subset.
	$M \in \eqc(E^{*})$. 
	We say that $M$ has the \emph{Haydon property} if 
	$$
	\overline{\co}^{w^{*}}(N) = \overline{\co}(N)
	$$ 
	for every weak-star closed $N \subseteq M$.
	If every 
	$M \in \eqc(E^{*})$
	%equicontinuous, weak-star compact subset of $E^{*}$ 
	has the Haydon property, then we say that so does $E$.
\end{defin}

%1302new
\begin{remark}
	As in Lemma \ref{l:HaydImage}, one may check for the operator $T \colon E_1 \to E_2$, that is, if $M$ has the Haydon property in $E_2^*$ then $T^*(M)$ has the Haydon property  in $E_1^*$. Moreover, we can define ``anti-Haydon" subsets violating the condition from Definition \ref{d:HaydProp} (similar to ``anti-H" from Definition \ref{d:anti-H} which violates the equality $\overline{\co}^{w^{*}}(\ext M)= 
	\overline{\co}(\ext M)$). In these terms, $B$ is anti-Haydon for $T^*(M)$ if and only if $T(B)$ is anti-Haydon  for $M$. Every anti-Haydon subset is anti-H but the converse is not always true. As a counterexample, one may consider $E=C[0,1]$ and $M:=B_{E^*}$. 
\end{remark}

\sk 
\begin{lem} \label{lemma:dense_Haydon_subspace}
	Let $F \leq E$ be a dense large subspace of $E$.
	If $F$ has the Haydon property, then so does $E$.
	%2409	Weakened this result because the localization of Haydon's theorem made it unnecessary.
	%Then $F$ has the Haydon property if and only if $E$ does.
\end{lem}
\begin{proof}
	%If $E$ has the Haydon property, then so does $F$ in virtue of Lemma \ref{lemma:subspace_haydon_property}.
	
	Suppose that $F$ has the Haydon property, and let $M \in \eqc(E^{*})$. 
	% $M \subseteq E^{*}$ be a weak-star compact, equicontinuous subset.
	Consider the inclusion map $i \colon F \hookrightarrow E$ and its adjoint $i^{*} \colon E^{*} \to F^{*}$, the restriction map.
	By Lemma \ref{lemma:large_subspace_strong_isomorphism}, $i^{*}$ is a strong isomorphism. 
	Moreover, $i^{*}$ is weak-star continuous (Fact \ref{f:adjoint_is_continuous}).
	Since $M$ is weak-star compact, $i^{*}$ is a closed map on $M$.
	Therefore:
	$$
	i^{*}(\overline{co}^{w^{*}}(M)) = 
	\overline{co}^{w^{*}}(i^{*}(M)) = 
	\overline{co}(i^{*}(M)) = 
	i^{*}(\overline{co}(M)).
	$$
	Since $i^{*}$ is a bijection (and therefore injective), we conclude that $\overline{co}^{w^{*}}(M) = \overline{co}(M)$.
	Note that we used the fact that $F$ has the Haydon property.
\end{proof}

%2409	New for the localization of Haydon's Theorem.
\begin{lem} \label{lemma:haydon_implies_coRl}
	Let $E$ be a locally convex space and let $M \subseteq E^{*}$ be a disked, equicontinuous, weak-star compact subset.
	If $M$ satisfies Haydon's property, then it also satisfies $\coRl$.
\end{lem}
\begin{proof}
	By contradiction, suppose that $M$ does not satisfy $\coRl$.
	%0202 Let $V$, $T \colon V \to E$ and $\delta > 0$ be as in Lemma \ref{lemma:coRl_l1_embedding}.
	By Lemma \ref{lemma:coRl_l1_embedding},  
	there exist an embedding $T\colon V \to E$ where $V$ is a dense normed subspace of $l^{1}$ and $\delta > 0$ such that 
	$
	\delta B_{V^{*}} \subseteq T^{*}(M).
	$
	Recall that $l^{1}$ does not satisfy the Haydon property.
	By virtue of Lemma \ref{lemma:dense_Haydon_subspace}, neither does $V$.
	By definition, there exists 
	%an equicontinuous, weak-star compact 
	$N' \in \eqc(B_{V^{*}})$ such that:
	$$
	\overline{\co}^{w^{*}} (N') \supsetneq \overline{\co} (N').
	$$
	Without loss of generality, we may assume that $N' \subseteq \delta B_{V^{*}}$.
	
	We know that $T^{*}(M) \supseteq \delta B_{V^{*}}$ and therefore
	$$
	T^{*}(M) \supseteq \delta B_{V^{*}} \supseteq N'.
	$$
	Define $N := M \cap \left((T^{*})^{-1}(N')\right) \subseteq M$.
	%0202 It is easy to see that $T^{*}(N) = N'$.
	Then $T^{*}(N)=T^{*}(M) \cap N' = N'$. 
	The adjoint $T^{*}$ is strongly continuous 
	%in virtue of Fact \ref{f:adjoint_is_continuous} 
	and therefore:
	$$
	T^{*}(\overline{\co} N) \subseteq 
	\overline{\co}(T^{*}(N)).
	$$
	Since $T^{*}$ is also weak-star continuous (Fact \ref{f:adjoint_is_continuous}), $N$ is a closed subspace of $M$, and therefore equicontinuous and weak-star compact.
	Since $T^{*}$ is a closed map over 
	%0202 
	weak-star compact $M$, 
	%2911	Added next line for clarification
	%0202 (which is compact) and therefore: 
	we get 
	$$
	T^{*}(\overline{\co} N) \subseteq 
	\overline{\co}(T^{*}(N)) =
	\overline{\co}(N') \subsetneq
	\overline{\co}^{w^{*}}(N') = 
	\overline{\co}^{w^{*}}(T^{*}(N)) = 
	T^{*}(\overline{\co}^{w^{*}}N).
	$$
	In particular, $T^{*}(\overline{\co} N) \subsetneq T^{*}(\overline{\co}^{w^{*}}N)$, and therefore $\overline{\co} N \subsetneq \overline{\co}^{w^{*}}N$.
	By definition, $M$ does not satisfy Haydon's property. This contradiction completes the proof. 
\end{proof}

%1910	Moved here
The following is a locally convex analogue of the equivalence (1) $\Leftrightarrow$ (2) in Lemma \ref{l:RosBanSpCharact}.

%0604  (lest's discuss) needs a little bit more reorganization because still it is not so clear what is the "third assertion"
\begin{prop} \label{prop:cotame_then_double_dual_is_fragmented}
	Let $E$ be a lcs. 
	\begin{itemize}
		\item 
	$E$ is tame if and only if every $x^{**} \in E^{**}$ is weak-star fragmented over every $M \in \eqc(E^{*})$.

	\item a weak-star compact, equicontinuous $M \in \eqc(E^{*})$ is co-tame if and only if every $x^{**} \in E^{**}$ is weak-star fragmented on $M$; 
	\item a bounded $B \subseteq E$ is tame on $M$ if and only if every $x^{**} \in \overline{B}^{w^{*}}$ is weak-star fragmented on~$M$ (the closure here is taken with respect to the weak star topology of $E \subseteq E^{**}$).
\end{itemize} 
	%an equicontinuous, weak-star compact subset  
\end{prop}
\begin{proof}
	First note that the second assertion follows from the third since every $x^{**} \in E^{**}$ is contained in the weak-star closure of some bounded $B \subseteq E$ (\cite[Thm.~5.4 p.~143]{Schaefer}).
	
	The rest of the proposition is a consequence of Lemma \ref{f:sub-fr}.
\end{proof}

\sk 
%0909q This theorem could probably be strengthen a bit to apply for individual weak-star equicontinuous subsets.
%0909q If this is interesting in your opinion I will add it
\begin{thm} \label{thm:Haydon} 
	%1503 
	\textbf{\emph{(Generalized Haydon Theorem)}} 
	
	\nt 
	For a locally convex space $E$, the following are equivalent:
	\ben
	\item [(i)] $E$ satisfies $\Rl$.
	\item [(ii)]  $E$ is tame.
	\item [(iii)]  Every weak-star compact, equicontinuous convex subset of $E^{*}$ is the strong closed convex hull of its extreme points.
	\item [(iv)] For every weak-star compact, equicontinuous subset $T$ of $E^{*}$, we have:
	$$
	\overline{\co}^{w^{*}}(T) = \overline{\co}(T).
	$$
	\een
	
	\sk 
	%2409	New for the localization of Haydon's Theorem.
	Specifically, if $M \in \eqc(E^{*})$, %is equicontinuous and weak-star compact,
	 then the following are equivalent:
	\ben
	\item $M$ satisfies $\coRl$.
	\item $M$ is co-tame.
	\item For every weak-star closed, convex $N \subseteq \overline{\acx}^{w^{*}}(M)$ we have:
	$$
	N = \overline{\co}(\ext N).
	$$
	\item $\overline{\acx}^{w^{*}}(M)$ has the Haydon property.
		  %1511	added
		  Explicitly, for every weak-star closed $N \subseteq \overline{\acx}^{w^{*}}(M)$, we have:
		  $$
		    \overline{\co}^{w^{*}}(N) = \overline{\co}(N).
		  $$
	\item Every $x^{**} \in E^{**}$ is a fragmented map over $M$.
	\een
\end{thm}
\begin{proof}
	\ 
	\ben
	%050422	Added details by Saak suggestion
	\item[$(1) \Rightarrow (2)$] By definition, there is no $l^{1}$-sequence with respect to $\rho_{M}$.
	Suppose that $B \subseteq E$ is bounded, and let $r\colon E \to C(M)$ be the restriction map.
	As a consequence, $r(B)$ contains no $l^{1}$ sequences in $C(M)$.
	%0604 Here we need to be careful. I think a little bit more explanation is necessary
	By Lemma \ref{f:sub-fr}, $r(B)$ is tame on $M$. 
	This is true for every bounded $B \subseteq E$ so $M$ is indeed co-tame.
	
	\item[$(2) \Rightarrow (3)$] Proposition \ref{prop:cotame_implies_haydon} and Lemma \ref{lemma:co_tame_is_locally_convex_bornology}.
	\item[$(3) \Rightarrow (4)$] Suppose that $N \subseteq \overline{\acx}^{w^{*}}(M)$ is a weak-star closed subset.
	As a consequence, it is also weak-star compact.
	Write $N' := \overline{\co}^{w^{*}}(N)$.
	By \cite[Lemma 9.4.5]{TVS}, the extreme points of the closed convex hull of a set lies in the closure of the original set.
	Thus:
	$$
	\ext N' \subseteq \overline{N}^{w^{*}} = N.
	$$
	Note that $N'$ is a closed convex subset of $\overline{\acx}^{w^{*}}(M)$ and therefore we can apply $(3)$:
	$$
	N' = 
	\overline{\co}(\ext N') \subseteq 
	\overline{\co}(N) \subseteq 
	\overline{\co}^{w^{*}}(N) =
	N'. 
	$$
	In particular, for every closed $N \subseteq \overline{\acx}^{w^{*}}(M)$, 
	$$
	\overline{\co}(N) = \overline{\co}^{w^{*}}(N), 
	$$
	as required.
	\item[$(4) \Rightarrow (1)$] By Lemma \ref{lemma:haydon_implies_coRl}, $\overline{\acx}^{w^{*}}(M)$ satisfies $\coRl$.
	It is easy 
	%0202q  some hint ? 
	%1102	 If $A \subseteq B \in \eqc(E)$ then $\rho_{A} \leq \rho_{B}$ 
	to see that so does ${N \subseteq \overline{\acx}^{w^{*}}(M)}$.
	\item[$(2) \Leftrightarrow (5)$] Proposition \ref{prop:cotame_then_double_dual_is_fragmented}.
	\een
\end{proof}

\sk 
\section{Representations of group actions} 
\label{s:repr} 

\subsection*{Representations of dynamical systems on lcs} 
\label{s:repres} 

%1503 new place with some simplifications 
%Later on, 
We apply properties of tame locally convex spaces to the theory of representations of dynamical systems. 
%0604  (lest's discuss) adding some background 
Let $G$ be a topological 
%130422 space 
group and $G \times X \to X$ be a continuous action of $G$ on a topological space $X$. Then we say that $X$ is a $G$-\textit{space}. If, in addition, $X$ is compact, then we say that $X$ is a \textit{dynamical $G$-system}. 
By the approach of A.~K\"{o}hler \cite{Ko}, a dynamical $G$-system $X$ is said to be \textit{tame} (\textit{regular}, in the original terms of K\"{o}hler) if the orbit 
$
fG=\{fg:g \in G\}, 
$
as a family of functions on $X$, is tame. 
%0604 
Or, equivalently, if $fG$ is a tame subset in the Banach space $C(X)$. 

Theorems \ref{t:repres} and \ref{t:repres_is_tame} establish the close relation between tame dynamical systems and representations on tame locally convex spaces. 
A compact $G$-system $X$ is representable on a tame lcs 
%130422 $E$ 
iff $(G,X)$ is tame as a dynamical system.  
Similarly, by Theorem \ref{t:repres on NP}, a compact $G$-system $X$ is representable on an NP  
%1907 
%2707q In other places we don't include "resp", should it be included here? %0108 removing. Prefer as simpler as possible
(reflexive) %2707q: This should be DLP? %0108q We must not write DLP and can say reflexive (because, really we construct reflexive lcs in that representation). Banach spaces are reflexive iff they are DLP. For lcs it is not true.  
lcs $E$ iff $X$ is hereditarily nonsensitive (weakly almost periodic).   
Such results are well-known (see \cite{GM-rose},\cite{GM1}, \cite{Me-nz}) for \textit{metrizable} tame (hereditarily nonsensitive, weakly almost periodic) $G$-systems with Rosenthal (Asplund, reflexive) Banach spaces $V$.

\sk 

%(see for example, \cite{GM-MTame}) 
%If, $X$ is metrizable then (see \cite{GM-rose}) the $G$-system $X$ is tame 
% if and only if it is representable on a Rosenthal Banach space $V$, where we can think $X$ as a subspace of the dual $V^*$ in its weak-star topology and $G$ as a subgroup of the group $\Iso(V)$ of linear isometries of $V$. 
%%In this theory \textit{tame family} $F \subset \R^X$ of real functions $X \to \R$ on a space $X$ is one of the key tools. Recall that $F$ is said to be \textit{tame} if it does not contain an     independent sequence of functions. 

Recall that a %strongly continuous 
(\emph{proper}) \emph{representation}
of a $G$-space $X$  on a Banach space $(V,||\cdot||)$ is a pair $(h,\al)$, 
where $h\colon G \to \Iso(V)$ is a strongly
continuous co-homomorphism and $\a \colon X \to
V^*$ is a weak star continuous bounded $G$-mapping (resp.
\emph{embedding}) with respect to the {\em dual action\/} 
$$G \times V^* \to V^*, \  (g \varphi)(v):=\varphi(h(g)(v))=\langle vg,\varphi \rangle= \langle v,g\varphi \rangle.$$
%0202 such that $\a(X)$ is bounded. 

%\textit{Proper representation} means that $\a \colon X \to V^*$ is a topological embedding (where $E$ carries the weak-star topology).  

\begin{f} \label{f:ReprMetr} 
	Let $X$ be a compact metrizable $G$-space. $X$ admits a proper representation on 
	\begin{enumerate}
		\item \cite{Me-nz}  a reflexive Banach space iff $X$ is a WAP dynamical $G$-system;
		\item \cite{GM1}  an Asplund Banach space iff $X$ is a HNS dynamical $G$-system;
		\item \cite{GM-rose} a Rosenthal Banach space iff $X$ is a tame dynamical $G$-system.
	\end{enumerate}
\end{f}

Many details about these results, as well as 
 the definitions of HNS (hereditarily nonsensitive) and WAP (weakly almost periodic) dynamical systems, can be found in \cite{GM-survey}. 
 %0808 
 These results, in the framework of a unified approach representing 
 ``small subsets", appear in \cite{Me-seminW,Me-b}. 
 
 In this section, we extend Fact \ref{f:ReprMetr} to nonmetrizable dynamical systems and suitable locally convex spaces. 
More definitions and remarks are in order. 
%0604 
For every lcs $E$ denote by $GL(E)$ the group of all continuous linear automorphisms of $V$. 

\ben
\item The right action of $G$ (induced by $h\colon G \to \Iso(V)$) on $(V,||\cdot||)$ and the corresponding dual action on the dual Banach space $V^*$ are equicontinuous. Moreover, one may
attempt to give a slightly 
more general definition. Namely,  defining $h \colon G \to GL(V)$ as a  co-homomorphism  which is equicontinuous. 
Then one may modify the norm getting again the ``isometric version" but under the new norm. Namely, define 
$$
||v||_{new}:=\sup \{||vg||: g \in G\}.
$$
Since the action is (uniformly) equicontinuous, this norm generates the same topology.

\item If $X$ is compact then $\a(X)$ is weak-star compact in $V^*$. Since $V$ is a Banach space then $\a(X)$ is automatically bounded and equicontinuous.
%0202 (by Banach-Steinhaus theorem).  Moreover, $\a(X)$ is an equicontinuous subset of the dual Banach space $V^*$. 
\een

\sk 
\begin{defin}
By a (proper) \emph{representation}
	of a $G$-space $X$  on a lcs $E$ we mean a pair $(h,\al)$, 
	where $h\colon G \to GL(E)$ is a strongly continuous 
	co-homomorphism, $h(G)$ is equicontinuous and $\a \colon X \to V^*$ is a weak star continuous $G$-mapping (resp. embedding) 
	with respect to the {\em dual action\/} 
	$$G \times E^* \to E^*, \  (g \varphi)(v):=\varphi(h(g)(v))=\langle vg,\varphi \rangle= \langle v,g\varphi \rangle$$
	such that $\a(X)$ is equicontinuous. 
\end{defin}

\begin{remarks} \ 
	\begin{enumerate} 
		\item The right action $E \times G \to E$ is continuous. It is equivalent to strong continuity of $h$ (because the action is equicontinuous). 
		\item It is well-known that the dual action $G \times E^* \to E^*$, in general, is not continuous (where $E^*$ carries its standard strong dual topology), even for Banach spaces and linear isometric actions.   
		%0507 \footnote{take, for example, the circle group $G=\T$ and its left action on itself. Now, for its regular representation on $C(\T)$ the dual action $G \times C(\T)^* \to C(\T)^*$ is not continuous};
		 A sufficient condition is that $E$ is an Asplund Banach space (see \cite{Me-fr}). The same is not true for Rosenthal Banach spaces.  
		 \item 
		 %0507 Since $E$ is not necessarily barreled, we cannot omit the condition that $\a(X)$ is equicontinuous.  
		 If $E$ is barreled and $X$ is compact, we can omit the condition that $\a(X)$ is equicontinuous.
		\item It is easy to see that if $X \in \eqc(E^*)$ is 
		%an equicontinuous weak-star compact  subset of $E^*$ which is
		 $G$-invariant, then the induced action $G \times X \to X$ is continuous.  For Banach spaces $E$ it is well-known; see, for example, the proof in \cite{Me-cs} (for the isometric representation). 
			\end{enumerate}
\end{remarks}

\begin{thm} \label{t:repres} 
	Every compact tame dynamical $G$-system admits a proper representation on a tame lcs. 
\end{thm}
\begin{proof} Let $X$ be a compact tame dynamical $G$-system. 
	As we already know by \cite{GM-rose}, $X$ is Rosenthal-approximable. That is,  there exists a G-embedding of $X$ into a $G$-product $\prod_{i \in I} X_i$ of Rosenthal-representable $G$-systems $X_i$. Let $(h_i,\a_i)$ be a proper representation of $(G,X_i)$ on a Rosenthal Banach space $V_i$. 
	Then the lcs direct sum $V:=\oplus_{i \in I} V_i$ is a tame lcs according to 
	%2012 Fixed reference
	Theorem \ref{thm:properties_of_tame_class}.4. 
	
	%0701
	Indeed, first of all note that algebraically the dual $V^*$ is the product
	%2206 Added *
	$\prod_{i \in I} V_{i}^{*}$ with the corresponding duality 
	$$
	V \times V^* \to \R, \ (v,u) \mapsto \sum_{i \in I} \langle v_i,u_i \rangle.
	$$
	Furthermore, the compact space $\prod_{i \in I} X_i$ naturally is embedded into $V^*$ with the weak-star topology. 
	One of the main steps here is to show that $\prod_{i \in I} X_i$ is an equicontinuous subset of $V^*$. 
	%1304
	This follows by Remark \ref{r:dualities}.2. 
	%1304 removing proof of the claim
	
		We have a naturally defined coordinate-wise linear action of $G$ on $V=\oplus_{i \in I} V_i$. Using the above mentioned description of the topology on $V=\oplus_{i \in I} V_i$, it is easy to show that this action is equicontinuous. 
		So, we have an equicontinuous strongly continuous co-homomorphism $h\colon G \to GL(V)$. 
		Finally, observe that the embedding $\a\colon \prod_{i \in I} X_i  \hookrightarrow V^*$ is  weak-star continuous and equivariant.  
	\end{proof}

%0806
Note that if in Theorem \ref{t:repres} the compact $G$-space $X$ in addition  is metrizable, then we can suppose that $V$ is a Banach space. However, it is not true if $X$ is not metrizable (even for trivial $G$-actions). 

\begin{thm} \label{t:repres_is_tame}
	Every compact $G$-space $X$ which admits a proper representation on a tame lcs is tame as a dynamical system. 
\end{thm}
\begin{proof}
	%3003
	Let $(h,\al)$ be a proper representation of the $G$-system $X$ on a tame lcs $E$, 
	where ${h\colon G \to GL(E)}$ is a strongly continuous 
	co-homomorphism, 
	$h(G)$ is equicontinuous and $\a \colon X \to
	V^*$ is a weak star continuous $G$-embedding 
	with respect to the {\em dual action\/} 
	$E^* \times G \to E^*$, %(g\varphi)(v):=\varphi(h(g)(v))= \langle vg,\varphi \rangle = \langle v,g\varphi \rangle$$
	such that $\a(X)$ is equicontinuous.  
	For every $v \in V$, we have the induced continuous function 
	$$f_v\colon X \to \R, x \mapsto  \langle v,\alpha(x) \rangle.$$ Then the orbit $vG$ is bounded because $h(G)$ is an equicontinuous subset of $GL(E)$. Since $E$ is a tame lcs, its bounded subsets are tame. Hence, $vG$ is a tame family for every  equicontinuous compact subset in the dual. In particular, it is true for $\alpha(X)$.  This implies that the orbit $f_vG$ of $f_v$ is a tame family on $X$. Therefore, $f_v$ is a tame function on $X$ (in the sense of \cite{GM-rose}). Since $V$ separates the points of $\alpha(X)$, we obtain that $\mathrm{Tame}(X)$ separates points of $X$. This means that the $G$-system $X$ is tame in the sense of \cite{GM-rose}. 
	\end{proof}

Similarly, making use of Remark \ref{r:dualities}, the following theorem can be proved. 
\begin{thm} \label{t:repres on NP}
	A compact dynamical $G$-system $X$ is representable on 
	%1907  involving lcs with DLP and WAP systems
	\begin{enumerate} 
		\item $E \in \NP$ if and only if $(G,X)$ is hereditarily nonsensitive;  
		\item $E \in (DLP)$ if and only if $(G,X)$ is weakly almost periodic.  
	\end{enumerate} 
\end{thm}

%2306   
 Note that, like $\Tame$, the classes $\NP$, $\DLP$ %(Corollary \ref{c:Ros-pr} and Remark \ref{r:DLPprop}) 
 and reflexive lcs 
also are closed under lc direct sums. So, we can assume in (2) that $E$ is a reflexive lcs. 
For brevity, we omit the details.

\sk  
%%%%%%%%%%%%%%%%%%%%%%%%%%%%%%%%%%%%
%\vskip 0.7cm 
\bibliographystyle{amsplain}

\end{document}